\documentclass[12pt,titlepage]{scrartcl}

\usepackage{style}
\usepackage{thesis}

\begin{document}
	
	\hypersetup{pageanchor=false}
	\begin{titlepage}
		\title{Bounded Cohomology of Groups acting on Cantor sets}
		
		\author[]{
			Konstantin Andritsch\\[1cm]
			
			Thesis presented for the degree\\
			Master of Science ETH in Mathematics\\[1cm]
			
			Supervised by Prof. Alessandra Iozzi\\
			Co-Supervised by Francesco Fournier-Facio}
		\affil[]{Department of Mathematics ETH Zurich}
		
		\date{09 August 2022.\ }
		
	\end{titlepage}

	\maketitle
	\hypersetup{pageanchor=true}
	\newpage
	
	\pagenumbering{roman}
	\setcounter{page}{2}

	\begin{abstract}
		\thispagestyle{plain}
		\begingroup
		\let\clearpage\relax
		
 \begin{center}
	\bfseries Abstract
\end{center}

We study the bounded cohomology of certain groups acting on the Cantor set. More specifically, we consider the full group of homeomorphisms of the Cantor set as well as Thompson's group $\V$. We prove that both of these groups are boundedly acyclic, that is the bounded cohomology with trivial real coefficients vanishes in positive degrees. Combining this result with the already established $\ZZ$-acyclicity of Thompson's group $V$, will make $V$ the first example of a finitely generated group, in fact the first example of a group of type $F_\infty$, which is universally boundedly acyclic.\\
We exploit recent results of the theory of bounded cohomology developed in \cite{moraschini2021amenability, fournierfacio2021binate, monod2021lamplighters} which allow to deduce the aforementioned statements of bounded acyclicity.\\
Prior to studying these concrete groups we review basic facts of bounded cohomology and describe how actions of discrete groups on boundedly acyclic modules can verify the vanishing of the bounded cohomology of these groups  \cite{moraschini2021amenability}. Moreover, we adapt the idea, developed in \cite{monod2021bounded}, that fat points and a generic relation on these points enable to calculate the bounded cohomology, to our situation.\\
Before proving bounded acyclicity, we gather various properties of the groups under consideration and certain subgroups thereof. As a consequence the proofs of bounded acyclicity will be relatively short.\\
It will turn out that the approaches to handle these groups are very similar. This suggests that there could be a unifying approach which would imply the bounded acyclicity of a larger class of groups acting on the Cantor set, including the discussed ones.

		\vspace{1in} 
 \begin{center}
	\bfseries Acknowledgements
\end{center}

First and foremost I would like to thank Francesco Fournier-Facio, without his assistance I probably would have never found my way into the fascinating topic of Bounded Cohomology. Thank you for all the good advice during the process of this thesis and for suggesting to work on this specific problem.  Moreover, I have to thank you for recommending to attend the conferences in Bielefeld, Zurich and Regensburg. Thank you!\\ 
Further, I need to thank Alessandra Iozzi. She made it possible for me to work on this project. Also, I am very grateful for your generosity and kindness throughout this semester.\\
A special thanks goes to Benedikt Andritsch for many useful suggestions making my work more understandable and easier to read.\\
Last but not least, I want to thank Charlotte Meyer for her never ending support in any shape or form, for all of her advice and for listening to all of my ideas regardless of whether they worked out or not.

		\endgroup
	\end{abstract}

	\tableofcontents
	\newpage
	\listoffigures
	
	\cleardoublepage
	\pagenumbering{arabic}
	\setcounter{page}{1} 
	
 \section{Introduction}
	
	We start by giving a short overview of the origin and historic development of bounded cohomology. However, this should not in the least be understood as an exhaustive elaboration of all relevant key points in the theory of bounded cohomology.	
	
\subsection{Motivation and Origins of Bounded Cohomology}
	
	The very first definition of bounded cohomology was given by \citeauthor{johnson1972cohomology} \cite{johnson1972cohomology} and Trauber for locally compact groups and also for discrete groups in the context of Banach algebras. Nevertheless, it was not until the seminal works by \citeauthor{gromov1982volume} \cite{gromov1982volume} and \citeauthor{Ivanov1987FoundationsOT} \cite{Ivanov1987FoundationsOT} in the eighties that bounded cohomology became an active field of research. In his paper \citeauthor{gromov1982volume} extended the definition of bounded cohomology to topological spaces and established several deep theorems. Moreover, he developed a striking application of the theory of bounded cohomology to Riemannian geometry. For example \citeauthor{gromov1982volume} used bounded cohomology to give estimates for the minimal volume and to calculate the simplicial volume of a manifold. \citeauthor{Ivanov1987FoundationsOT}s paper is a successful attempt to separate the abstract theory of bounded cohomology from the Riemannian geometric part in \cite{gromov1982volume}. Some remarkable facts \citeauthor{gromov1982volume} and \citeauthor{Ivanov1987FoundationsOT} deduce is that bounded cohomology of topological spaces only depend on the fundamental group of the space and further that amenable fundamental groups imply vanishing of bounded cohomology. Although, according to \citeauthor{gromov1982volume}, the latter statement was essentially already known by Trauber who proved that the bounded cohomology of amenable groups (see \Cref{def:amenability}) is equal to zero. Further, \citeauthor{gromov1982volume} and \citeauthor{Ivanov1987FoundationsOT} exhibited that bounded cohomology admits a description via certain injective resolutions (see \Cref{sec:rel-inj-res}).\\
	
	In \citeyear{monod2001continuous} and \citeyear{Burger2002}, similarly as \citeauthor{Ivanov1987FoundationsOT} did for discrete groups and trivial coefficients, \citeauthor{monod2001continuous}, and \citeauthor{Burger2002} laid the foundations of a systematic and functorial approach for continuous bounded cohomology of locally compact groups and coefficients in Banach modules in \cite{monod2001continuous} and \cite{Burger2002}, respectively. Although, in this thesis we will only consider discrete groups and not locally compact groups in general, the theory of continuous bounded cohomology plays a significant role for several results in the context of bounded cohomology of discrete groups. For instance, vanishing results of bounded cohomology in degree 2 of lattices in certain ambient groups can be obtained by relating the bounded cohomology of the lattice to the continuous bounded cohomology of the ambient group \cite[Corollary 13.5.5]{monod2001continuous}.\\
	
	Most recently in \citeyear{2017Frigerio} \citeauthor{2017Frigerio} wrote a comprehensive book \cite{2017Frigerio} on the theory of bounded cohomology of discrete groups and its applications. Calculating the bounded cohomology of a group is considered to be a hard task in general. Anyhow, in the recent past there has been a lot of progress in this direction \cite{fournierfacio2021binate, monod2021lamplighters, 2017_LOEH, moraschini2021amenability, 2021boundedcohomologyoffinitelypresentergroups}.

\subsubsection{Applications and Connections to other Research Fields}
	
	As mentioned above, in \citeyear{gromov1982volume} \citeauthor{gromov1982volume} introduced the concept of simplicial volume, a homotopy invariant associated to oriented manifolds, which is deeply connected to the geometric structures the manifold can carry. The simplicial volume of a closed manifold $M$ is defined as the $\ell^1$-seminorm $\norm[]{M} := \norm[]{[M]}$ on singular homology of the real fundamental class $[M]$ of $M$. For example, we have:
	
	\begin{theorem}[{\cite[Corollary 7.11]{2017Frigerio}}]
		Let $M$ be a closed, orientable $n$-manifold. Let $c:\HHbR[n][M]\to\HHR[n][M]$ be the comparison map. Then $\norm[]{M}=0$ if and only if $c\equiv0$ and $\norm[]{M}>0$ if and only if $c$ is surjective.
	\end{theorem}
	
	Further, \citeauthor{gromov1982volume} established the following estimate between the minimal volume 
	\begin{align*}
		\operatorname{MinVol}(M) := \inf\{\operatorname{vol}(M,g)~|~&g \text{ a Riemannian metric with sectional}\\
		&\text{curvature bounded between } -1 \text{ and } 1\}
	\end{align*}
	of a manifold and its simplicial volume.
		
	\begin{theorem}[{\cite{gromov1982volume}}]
		Let $M$ be a closed, orientable $n$-manifold. Then the inequality
		\[ \norm[]{M} \leq \operatorname{MinVol}(M)\cdot (n-1)^n\cdot n! \]
		holds.
	\end{theorem}
	
	So in particular, manifolds with positive simplicial volume also have positive minimal volume.\\
	
	Another field in which bounded cohomology proved to be a useful tool is in providing rigidity results for representations. A classical fact in bounded cohomology is that a group epimorphism induces an injective map in bounded cohomology with real coefficients in degree 2. Hence, if the second bounded cohomology $\HHbR[2]$ of a certain group $\group$ is finite dimensional, the image of any representation $\rho:\group\to\group'$ must have finite-dimensional bounded cohomology in degree 2 as well. \citeauthor{Burger1999} proved in \cite{Burger1999} that $\HHbR[2]$ is finite-dimensional if $\group$ is a uniform irreducible lattice in a higher rank semisimple Lie group. This can be used to show that any representation of an irreducible higher rank lattice in a connected semisimple Lie group into a mapping class group is finite \cite{BestvinaFujiwara2002}.\\
	
	The last classical application of bounded cohomology we would like to mention is the study of group actions on the unit circle. By fundamental results of \citeauthor{etienne1986groupes} \cite{etienne1986groupes} semi-conjugacy classes of the representation of a group into the homeomorphism group of the unit circle can be completely described by their bounded Euler class.
	
	\begin{theorem}[{\cite[Théorème A]{etienne1986groupes}}{\cite[Theorem 10.15]{2017Frigerio}}]
		Let $\rho$, $\eta$ be representations of a group $\group$ into $\Homeo^+(S^1)$, the group of orientation preserving homeomorphisms of the circle $S^1$. If and only if $E_b(\rho) = E_b(\eta)$, then $\rho$ is semi-conjugate to $\eta$.
	\end{theorem}
	
	Here $E_b(\cdot)$ denotes the integral bounded Euler class of a representation. It is defined as follows: consider a representation $\rho:\group\to\Homeo^+(S^1)$; then $E_b(\rho)$ is defined to be the element $\HHbf[\rho](E_b)$ that is, the image of the bounded Euler class under the map in bounded cohomology induced by the representation $\rho$. For further readings and the definition of the bounded Euler class we refer to \cite[Chapter 10]{2017Frigerio}.\\	
	In particular, it can be shown that representations with non-zero bounded Euler class in $\HHbc[2][][\ZZ]$ give \enquote{interesting circle actions}, i.e. circle actions without global fixed points.
		
\subsubsection{Vanishing and Non-Vanishing of Bounded Cohomology}
	
	Probably the first vanishing results in bounded cohomology were given by Johnson--Trauber:
	
	\begin{theorem}[Johnson--Trauber, \cite{johnson1972cohomology}]
		If $\group$ is amenable, then $\HHbR[n] = 0$ for all $n\geq 1$.
	\end{theorem}
	
	However, the vanishing of bounded cohomology in positive degrees with trivial real coefficients is not a characterization of amenability. In fact, amenability is equivalent to the vanishing of bounded cohomology in positive degrees for \highlight{all} dual Banach modules $\V$ \cite[Theorem 5.2]{johnson1972cohomology}. Hence, amenable groups are in some sense invisible for bounded cohomology.\\
	
	The first example of a non-amenable group having no bounded cohomology in positive degree with real coefficients is due to \citeauthor{MR787909}:
	
	\begin{theorem}[{\cite[Theorem 3.1]{MR787909}}]
		Denote by $\Homeo_c(\RR^k)$ the group of homeomorphisms of $\RR^k$ with compact support. It holds that $\HHbR[n][\Homeo_c(\RR^k)] = 0$ for all $n\geq 1$.
	\end{theorem}
	
	Until 2017 the example of \citeauthor{MR787909} has been the only such example. This is probably also due to the fact that it is very hard to compute bounded cohomology. Several concepts which make ordinary cohomology computable cannot exist for bounded cohomology as the following example shows:\\	
	The circle $S^1$ has vanishing bounded cohomology, i.e. $\HHbR[n][S^1] = 0$ for all $n\geq 1$, for instance since the fundamental group of the circle $S^1$ is abelian, hence amenable. Further, the fundamental group of the wedge of two circles $S^1\vee S^1$ is the free group with two generators $F_2$, hence by Gromov's Mapping Theorem \cite[Section 3.1]{gromov1982volume} we have $\HHbR[n][S^1\vee S^1]\cong\HHbR[n][F_2]$. However, even today it is open whether $\HHbR[n][F_2]$ is $0$ or not for all $n\geq4$ \cite[Question 18.3]{burger2008bounds}. Nevertheless, it is known that the bounded cohomology of $F_2$ in degrees $2$ and $3$ is not only not $0$, but even $\infty$-dimensional. In \cite{brooks1981some}, \citeauthor{brooks1981some} constructed infinitely many quasimorphisms on $F_2$, which were shown to define linearly independent elements in $\HHbR[2][F_2]$ by \citeauthor{mitsumatsu1984bounded} \cite{mitsumatsu1984bounded}. In \cite{Extending2015Frigerio} the authors show that $F_2$ has infinite dimensional bounded cohomology with real coefficients in degree $3$. This shows in particular that concepts as excision, or Mayer--Vietoris sequences are not applicable to bounded cohomology.\\
	
	However, in recent years a lot of new examples for non-amenable groups with vanishing bounded cohomology have been found. First, by adapting the proof of \citeauthor{MR787909}, \citeauthor{2017_LOEH} proved in \citeyear{2017_LOEH} that mitotic groups $\group$ are boundedly acyclic \cite[Theorem 1.2]{2017_LOEH}, that is $\HHbR[n] = 0$ for all $n\geq1$. Thus, mitotic groups give the first countable examples of non-amenable groups with vanishing bounded cohomology with real coefficients. Second, in \citeyear{2021boundedcohomologyoffinitelypresentergroups},  based on the work of the aforementioned authors, \citeauthor{2021boundedcohomologyoffinitelypresentergroups} were able to generate the first non-amenable finitely generated \cite[Theorem 2]{2021boundedcohomologyoffinitelypresentergroups} and non-amenable finitely presented \cite[Theorem 4]{2021boundedcohomologyoffinitelypresentergroups} boundedly acyclic groups. Third, the same authors proved further that even all binate groups have vanishing bounded cohomology.
	
 	\begin{theorem}[{\cite[Theorem 1.3]{fournierfacio2021binate}}]
 		All binate groups $\group$ satisfy $\HHbR[n][\group] = 0$ for all $n\geq 1$.
 	\end{theorem}
 
	Observe that binate groups contain mitotic groups as well as $\Homeo_c(\RR^k)$. Noticeably, all of these results build upon the acyclicity of these groups.\\ 
	
	Almost at the same time \citeauthor{monod2021lamplighters} showed that also for all \enquote{lamplighter} groups the bounded cohomology vanishes:
	
	\begin{theorem}[{\cite[Theorem 3]{monod2021lamplighters}}] \label{INT_thm:lamplighters}
		Let $\group$ be any group and consider the wreath product
		\[ W := \group\wr\ZZ = \left(\bigoplus_{\ZZ}\group\right)\rtimes\ZZ. \]
		Then $\HHbc[n][W]$ vanishes for all $n\geq1$ and all separable dual Banach $W$-modules $\module$.
	\end{theorem}
	
	In contrast to the results above this \namecref{INT_thm:lamplighters} was proved using ergodic methods and is not using acyclicity of the group under consideration. Observe that these ergodic methods allow to establish these vanishing results not only for real coefficients but for a large class of coefficient modules.\\
	Since Thompson's group $F$ contains coamenable (see \Cref{def:coamenability}) lamplighter groups an immediate \namecref{INT_cor:thompson-F} of the \namecref{INT_thm:lamplighters} above is
	
	\begin{corollary}[{\cite[Theorem 1 and 2]{monod2021lamplighters}}] \label{INT_cor:thompson-F}
		For Thompson's group $F$ we have $\HHbc[n][F]=0$ for all $n\geq1$ and all separable dual Banach $F$-modules $\module$.\\
		In particular, $\HHbR[n][F]=0$ holds for all $n\geq1$.
	\end{corollary}

	Regarding the characterization of amenability described above this is almost enough to settle the most important open question about Thompson's group $F$, whether $F$ is amenable or not, in the positive. What is left to verify is the vanishing of the bounded cohomology of $F$ with coefficients in all \highlight{non-separable} dual Banach $F$-modules. However, as \citeauthor{monod2021lamplighters} emphasizes the proofs established in \cite{monod2021lamplighters}  also work for groups that share similarities with $F$ but are known not to be amenable.\\
	
	By now there are also examples of groups where the bounded cohomology with real coefficients is known in every degree but which is neither $0$ in all positive degrees, nor unimaginably large in any degree, like the second and third bounded cohomology of the free group with two generators $F_2$. For example the group $\Homeo_+(S^1)$ of orientation preserving homeomorphisms of the unit circle.
	
	\begin{theorem}[{\cite[Theorem 1.1]{monod2021bounded}}]
		We have $\HHbR[][\Homeo_+(S^1)] \cong \RR[E_b]$, the polynomial ring generated by the bounded Euler class.
	\end{theorem}
	
	Note that the ring structure is considered with the cup product in bounded cohomology. Moreover, the bounded Euler class $E_b$ lives in $\HHbR[2][\Homeo_+(S^1)]$, so it is of degree $2$. Thus it holds that $\HHbR[n][\Homeo_+(S^1)] = 0$ for all odd $n\geq1$ and $\HHbR[n][\Homeo_+(S^1)] \cong \RR$ for all even $n\geq1$. Interestingly, the same is true for Thompson's group $T$.
	
	\begin{theorem}[{\cite[Theorem 1.6]{fournierfacio2021binate}}]
		If the Thompson group $F$ has vanishing bounded cohomology in positive degrees, then $\HHbR[][T]$ (with the cup product structure) is isomorphic to the polynomial ring $\RR[x]$ with $\abs{x} = 2$ and the bounded Euler class of $T$ is a polynomial generator of $\HHbR[][T]$. Moreover, the canonical semi-niorm on $\HHbR[][T]$ then is a norm.
	\end{theorem}
	
	Since by \Cref{INT_cor:thompson-F} Thompson's group $F$ has vanishing bounded cohomology in positive degrees, the bounded cohomology of $T$ is determined.\\

	Unsurprisingly, there are still a lot of open questions in the theory of bounded cohomology. Some might appear to be easily answered at first sight. One of these is the following:  
	
	\begin{question}
		Do there exist groups $\group$ which are not \bAc, but for which the bounded cohomology with real coefficients vanishes from a certain degree on? That is, $\HHbR[n]=0$ for all big enough $n\geq1$, but not all $n\geq1$?
	\end{question}
	
\subsection{Present work}

	As described above, non-zero bounded cohomology classes can give interesting information about the underlying group. However, it can also be enlightening to know that a group does not have any (non-trivial) bounded cohomology classes (with real coefficients) at all. Such groups are typically referred to as \bAc groups.
	
	\begin{definition}[bounded acyclicity]\label{INT_def:bAc}
		Let $\group$ be a group and $n \geq 1$. We call $\group$ \highlight{\bAc[n]} if $\HHbR[i] \cong 0$ for all $1 \leq i \leq n$.\\
		Moreover, $\group$ is called \highlight{\bAc} or \highlight{\bAc[\infty]} if it is \bAc[n] for all $n\geq 1$.
	\end{definition}
	
	In this thesis we will prove a vanishing result of bounded cohomology for two groups acting on the Cantor set.

	\subsubsection{Homeomorphisms of the Cantor set}
	
	The first goal of this thesis, after a basic introduction into bounded cohomology, is to prove the following \namecref{INT_thm:Homeo(K)-is-bAc} which states that the full homeomorphism group $\Homeo(\Cantor)$ of the Cantor set $\Cantor$ is \bAc.

	\begin{theorem}\label{INT_thm:Homeo(K)-is-bAc}
		Let $\Cantor$ be the Cantor set, then $\Homeo(\Cantor)$ is \bAc, i.e. $\HHbR[n][\Homeo(\Cantor)] \cong 0$ for all $n\geq1$.
	\end{theorem}
	
	We will prove this theorem using two main ingredients. First, we will consider the natural action on the Cantor set $\Cantor$ and derive that the rigid stabilizer subgroup $\CantorGrpStab$ of an arbitrary point $z\in\Cantor$ is boundedly acyclic:

	\begin{lemma}\label{INT_lem:stabilizer->dissipated->bAc}
		For every $z\in\Cantor$, $\CantorGrpStab$ is binate. In particular, $\CantorGrpStab$ is boundedly acyclic.
	\end{lemma}
	
	\citeauthor{Berrick1989UniversalGB} \cite{Berrick1989UniversalGB} and \citeauthor{varadarajan1985pseudo} \cite[Theorem 1.7]{varadarajan1985pseudo} proved independently that all binate groups are acyclic. Hence, this statement also implies that $\CantorGrpStab$ is acyclic in the usual sence, i.e. $\HH_n(\group;\ZZ) = 0$ for all $n\geq 1$.\\
	The proof of this \namecref{INT_lem:stabilizer->dissipated->bAc} will essentially only exploit the fact that the subset which is fixed by an element of the group under consideration is not dense in $\Cantor$. In particular, we will be able to deduce the following corollary:
	
	\begin{corollary}\label{INT_cor:tuple-stabilizer->dissipated}
		Let $k\in\NN$. For any $k$-tuple $(x_1,\dots,x_k)\in\Cantor^k$ the subgroup
		\[ \CantorGrpStab[(x_1,\dots,x_k)] = \{g\in\CantorGrp~|~g \text{ fixes a neighborhood of } x_i \text{ pointwise}, 1\leq i\leq k\} \]
		is binate, so in particular \bAc.
	\end{corollary}

	Second, we will use the result of \citeauthor{moraschini2021amenability} in \cite[Proposition 2.5.4]{moraschini2021amenability} (see \Cref{prop:bounded-cohomology-through-acyclic-resolutions}) stating that one can calculate the bounded cohomology of a given group, if this group acts on a chain complex of \bAc modules.\\
	
	To connect these two steps we will adapt the idea of \citeauthor{monod2021bounded} in \cite{monod2021bounded}, in which they handled the group $\Homeo_+(S^1)$, to the totally disconnected setting of the Cantor set and consider a group action on a set of so called \highlight{fat points} $\FatPoints$. A certain relation $\generic$ on these fat points will give rise to a chain complex $\linf{\FatPoints_{\bullet}^{\generic}}$ of $\Homeo(\Cantor)$-modules. Although both approaches are similar, $\Homeo_+(S^1)$ ends up being not \bAc. This difference arises due to the fact that in the totally disconnected setting we have high transitivity of the group action on the set of fat points, whereas the elements of $\Homeo_+(S^1)$ are orientation preserving and hence $\Homeo_+(S^1)$ acts transitively only on \highlight{cyclically ordered tuples} of the corresponding set of fat points. To prove the bounded acyclicity of these $\Homeo(\Cantor)$-modules we will use the \namecref{INT_lem:stabilizer->dissipated->bAc} and \namecref{INT_cor:tuple-stabilizer->dissipated} above together with the fact that the fat point stabilizer subgroups coincide with the rigid stabilizer subgroups (\Cref{lem:action-of-CantorGrp-on-FatPoints}).\\
	
	\subsubsection{Thompson's Group \texorpdfstring{$V$}{V}}
	
	Another major result of this thesis will be established in \Cref{sec:thompsons-group-V}. Here we will use a very similar approach as in the case of $\Homeo(\Cantor)$ to prove that Thompson's group $\V$ has vanishing bounded cohomology as well, answering Question 6.3 from \cite{fournierfacio2021binate}.
		
	\begin{theorem}\label{INT_thm:V-is-bAc}
		Thompson's group $\V$ is \bAc.
	\end{theorem}

	The bounded cohomologies of Thompson's groups $F$ and $T$ were already known before (see \cite[Theorem 1]{monod2021lamplighters} and \cite[Theorem 1.6]{fournierfacio2021binate}, respectively). Observe, that we have $F\subseteq T\subseteq V$ and although $F$ and $\V$ are \bAc, $T$ has bounded cohomology classes in every even degree.\\
	Again, comparing \cite[Theorem 1.1]{monod2021bounded} and \cite[Proposition 6.9]{fournierfacio2021binate} with our approach the difference between the bounded cohomology of Thompson's groups $\V$ and $T$ essentially arise due to the fact that $\V$ acts highly transitively on tuples of distinct points, whereas $T$ is in some sense orientation preserving.\\
	
	Using the result of $\ZZ$-acyclicity (\cite[Theorem 6.4]{SzyminWahl2019}) of Thompson's group $\V$ together with \Cref{INT_thm:V-is-bAc} makes $\V$ the first \highlight{finitely generated} universally \bAc group, as well as the first example of a universally \bAc group of type $F_\infty$:

	\begin{corollary}\label{INT_cor:V-is-universally-bAc}
		$\V$ is \highlight{universally boundedly acyclic}, that is $\HHbc[n][\V][\mathbb{K}] = 0$ for all $n\geq1$ and all complete valued fields $\mathbb{K}$.
	\end{corollary}	
	
	The other known examples where \highlight{binate} groups \cite[Theorem 1.4]{fournierfacio2021binate}, which cannot be finitely generated, as we will see in \Cref{sec:BC-of-thompsons-group-V}.

\subsection{Organisation}
	In \Cref{sec:bounded-cohomology} we first give an introduction into the basics of bounded cohomology of discrete groups with coefficients in normed modules over a normed  ring. Afterwards we discuss how bounded cohomology can be calculated using relatively injective resolutions. Further we introduce bounded acyclicity and end the section by giving several methods to show that groups or modules have this property.
	
	\Cref{sec:homeomorphisms-on-K} addresses Cantor sets and homeomorphisms between them. We treat some topological properties, for example that closed and open subsets of Cantor sets are themselves Cantors sets. Then we show that rigid stabilizer subgroups are boundedly acyclic which helps us to deduce the main goal of this section, which is to show group of all homeomorphisms of the Cantor set is as well boundedly acyclic.
	
	At the beginning of \Cref{sec:thompsons-group-V} we introduce the reader to the Thompson groups $F,T$ and $V$. Later on we focus on $V$ and, after showing some elementary properties, use a similar strategy as in \Cref{sec:homeomorphisms-on-K} to approach the bounded cohomology of Thompson's group $V$. Namely, we show that point stabilizers are boundedly acyclic which together with results from \Cref{sec:basic-def-of-BC} gives the vanishing of bounded cohomology of $V$. In the end of this section we carry out this argumentation in more detail and hence prove that Thompson's group $V$ is boundedly acyclic as well.
	
	In \Cref{sec:outlook-and-open-questions} we briefly discuss the similarities and differences between the approaches in \Cref{sec:homeomorphisms-on-K} and \Cref{sec:thompsons-group-V}. We conclude by posing a question about a generalized version of two main problems at hand for which a positive answer would imply the vanishing of bounded cohomology for a variety of groups including the ones discussed in this thesis. 
	
	\Cref{APX_sec:cantor-sets-and-maps-between-them} takes care of more involved topological properties of Cantor sets. The result that every two Cantor sets are homeomorphic is proven in this section.
	
	\Cref{APX_sec:Ultrafilters-and-ultralimits} reviews the basic definitions of filters, ultrafilters and ultralimits. These concepts are needed in the proof of the vanishing of (bounded) cohomology of a semi-simplicial set associated to generic relations.

\subsection{Notations and Conventions}
		
	In this thesis we will use the following notations and conventions:\\
	
	$\group$ will always denote a discrete group and $\subgroup$ a subgroup of $\group$.\\
	If $\group$ acts on some set $X$, we will denote the action of $g\in\group$ on $x\in X$ by $g\cdot x$, except if the action is just multiplication on the left (or right), then we will in general omit the point and write $gx$ (or $xg$).\\
	Moreover, we denote the support $\{x\in X~|~g\cdot x \neq x\}$ of an element $g\in\group$ by $\supp(g)$.\\
	
	The commutator $[a,b]$ of two group elements of any group $\group$ is defined as $[a,b]:= a^{-1}b^{-1}ab$.\\
		
	$\linf{\cdot}[\cdot]$ denotes the set of all bounded functions from some set to a normed space.\\
	
	For a chain map $\alpha:\mathcal{C}^{\bullet}\to \mathcal{D}^{\bullet}$ of cochain complexes, we will denote the function in cohomology, induced by $\alpha$, by $\HHf[\alpha]:\HH[][\mathcal{C}^{\bullet}]\to\HH[][\mathcal{C}^{\bullet}]$.\\
	If, in addition, $\mathcal{C}^{\bullet}$ and $\mathcal{D}^{\bullet}$ are normed cochain complexes and $\alpha$ is bounded, then it also induces a map in bounded cohomology, which we consequently notate with $\HHbf[\alpha]:\HHb[][\mathcal{C}^{\bullet}]\to\HHb[][\mathcal{C}^{\bullet}]$.\\
	
	Moreover, $\Cantor$ will always denote a Cantor set. Although there is of course no \enquote{unique} Cantor set, we will often speak of \highlight{the} Cantor set. When we work in the concrete dyadic model of the Cantor set we emphasize this choice of model with a subscript $2^\NN$, i.e. we denote this model by $\CantorModel$.\\
	
	In \Cref{sec:thompsons-group-V}, we extensively use the dyadic model $\CantorModel$ for the Cantor set. There we will work with finite sequences of $0$'s and $1$'s. For the constant sequences of zeros and ones of length $M$ we will write $\sroot{M}$ and $\sroot{M}[\one]$, respectively.\\
	
	Finally, in some proofs we have sub-statements, which are stated as \textbf{\mathenvfont{Claim.}} and proved afterwards. In order to avoid confusion with the end of the proof of the superior statement, we terminate the proof of the claim with $\dashv$ and not with $\qed$ as commonly used.

	\newpage
	
 \section{Bounded cohomology}\label{sec:bounded-cohomology}
	
	We start by giving rigorous definitions of (bounded) cohomology of discrete groups with coefficients in (normed) modules over a ring. Afterwards we show a different viewpoint to bounded cohomology which is sometimes easier to approach than direct calculations. Then we discuss bounded acyclicity, which is the most important property studied in this thesis.
	
\subsection{Basic Definitions of Bounded Cohomology} \label{sec:basic-def-of-BC}
	
	In the whole section $\group$ will denote a group equipped with the discrete topology. Although all of the constructions in this section work analogously also for right group actions, unless otherwise stated, \highlight{all} actions by groups will be understood as left actions. Furthermore, $\ring$ is a (not-necessarily commutative) ring, $\groupring$ the group ring associated to $\ring$ and $\group$, and $\module$ a (left) \ringmodule. We will mostly give general definitions for arbitrary rings but later focus on $\ring=\RR$. When $\group$ acts by $\ring$-linear maps on $\module$, then $\module$ can naturally be viewed as an \groupmodule, where multiplication works as follows:	
	\begin{align*}
		\left(\sum_{g\in\group}a_gg\right)\cdot v = \sum_{g\in\group}a_gg(v)
	\end{align*}
	with $a_g\in\ring$ and $v\in\module$ $\forall g\in\group$, 
	where the sum on the left hand side is finite (hence so is the sum on the right hand side).\\
	For every \groupmodule we can define the submodule $\module^\group$ of \highlight{$\group$-invariants} by
	\[	\module^\group := \{v\in\module~|~g(v) = v\quad\forall g\in\group \}. \]
	Note that every \ringmodule $\module$ is automatically also an \groupmodule, since we can let $\group$ act trivially on $\module$. We call such modules \highlight{trivial} \groupmodule\s.\\
	Given $\group$ and $\module$ as above, we can define the sequence of \groupmodule\s $(\Cc[n])_{n\geq0}$ given by $\Cc[n] :=\{f:\group^{n+1} \to \module\}$. This sequence of modules, together with the simplicial differential maps $\codiff[n]:\Cc[n]\to\Cc[n+1]$ defined by
	\[ \codiff[n]f(g_0\dots,g_{n+1}) := \sum_{i=0}^{n+1} (-1)^if(g_0,\dots,\widehat{g}_i,\dots,g_{n+1}), \]
	gives a cochain complex $(\C,\codiff)$, since $\codiff[n]\circ\codiff[n-1] = 0$. We can equip $\Cc[n]$ with the \highlight{diagonal $\group$-action}, i.e.
	\[ (g\cdot f)(g_0\dots g_n) := g\cdot f(g^{-1}g_0,\dots,g^{-1}g_n), \]
	which makes $(\Cc,\codiff)$ into a cochain complex of \groupmodule\s. Observe that the $\group$-invariants $\Cc^\group$ form a subcomplex of the cochain complex $(\Cc,\codiff)$. Using these notations, we can define group cohomology:
	
	\begin{definition}[Cohomology of Groups]
		The \highlight{cohomology} of a group $\group$ with coefficients in an \groupmodule $\module$ is the cohomology of the cochain complex  $(\Cc^\group,\codiff)$.
		Denote the cocycles by $\Zc[n] := \ker(\codiff[n]:\Cc[n]^\group\rightarrow\Cc[n+1]^\group)$ and the coboundaries by $\Bc[n] := \im(\codiff[n-1]:\Cc[n-1]^\group\to\Cc[n]^\group)$ for $n\geq 1$ as well as $\Bc[0] = 0$, then we define
		\[	\HHc[n] := \frac{\Zc[n]}{\Bc[n]}. \]
	\end{definition}
	
	Briefly speaking bounded cohomology of discrete is essentially the analogue of group cohomology, when we only consider \highlight{bounded} functions. First, we only consider \highlight{normed} \groupmodule\s $(\module$,~$\norm[\module]{\cdot})$, with a $\group$-invariant norm $\norm[\module]{\cdot}$. Next, we can equip $\Cc[n]$ with the ordinary sup-norm $\norm{\cdot}$ as follows:
	\[ \norm{f} := \sup\{\norm[\module]{f(g_0,\dots,g_n)}~|~(g_0,\dots,g_n)\in\group^{n+1}\} \]
	This makes $\Cc[n]$ into a normed \groupsubmodule.\\
	Finally, we define the \groupsubmodule\s $\Cbc[n] :=\{f\in\Cc[n]~|~\norm{f}<\infty\}$ of $\Cbc[n]$. Observe that the differential maps $\codiff[n]:\Cc[n]\to\Cc[n+1]$ are all bounded functions with norms equal to $(n+1)$, hence they restrict to maps $\codiff[n]:\Cbc[n]\to\Cbc[n+1]$ leading to the normed cochain complex $(\Cbc,\codiff)$.
	
	\begin{definition}[Bounded Cohomology of Groups]
		The \highlight{bounded cohomology} of a group $\group$ with coefficients in a normed \groupmodule $(\module$,~$\norm[\module]{\cdot})$ is the cohomology of the cochain complex  $(\Cbc^\group,\codiff)$.
		If we define the \highlight{bounded cocycles} by $\Zbc[n] := \ker(\codiff[n]:\Cbc[n]^\group\to\Cbc[n+1]^\group)$ and the \highlight{bounded coboundaries} by $\Bbc[n] := \im(\codiff[n-1]:\Cbc[n-1]^\group\to\Cbc[n]^\group)$ for $n\geq 1$ and $\Bbc[0] = 0$, then we can define
		\[	\HHbc[n] := \frac{\Zbc[n]}{\Bbc[n]}. \]
	\end{definition}
	
	
	
\subsection{How to Compute Bounded Cohomology} 
	
\subsubsection{Relatively Injective Resolutions}\label{sec:rel-inj-res}
	
	We will now describe how the classical homological approach of calculating group cohomology via injective resolutions carries over to the bounded cohomological setting. This approach was introduced by \citeauthor{Ivanov1987FoundationsOT} in \cite{Ivanov1987FoundationsOT}. See also \cite{2017Frigerio} for further information.\\	
	In the following a bounded $\group$-equivariant, $\ring$-linear map $\iota:U\to W$ between normed \groupmodule\s will be called a \highlight{$\group$-morphism}.\\
	
	\begin{remark}\label{rem:equivariance-preserves-invariants}
		If $\iota:U\to W$ is a $\group$-equivariant map, then this induces (by restriction) a map $\iota:U^\group\to W^\group$ between the corresponding $\group$-invariant submodules. Moreover, note that $U^\group$ is a closed subgroup of $U$. Therefore, in case $\iota:U\to W$ is bounded, so is the induced map.
	\end{remark}

	Further, a $\group$-morphism  $\iota:U\to W$ between normed \groupmodule\s is called \highlight{strongly injective} if there is a bounded $\ring$-linear map $\sigma:W\to U$, such that $\norm[]{\sigma} \leq 1$ and which is a verse to $\iota$, i.e. $\sigma\circ\iota = id_{U}$. Note that in general $\sigma$ need not be $\group$-equivariant.
	
	\begin{definition}[relatively injective module]
		A normed \groupmodule $\module$ is \highlight{relatively injective} if for every strongly injective $\group$-morphism $\iota:U\to W$ between normed \groupmodule\s $U$ and $W$, and every $\group$-morphism $\alpha:U\to \module$, there is a $\group$-morphism $\beta:W\to \module$ such that $\norm[]{\beta}\leq\norm[]{\alpha}$ and $\beta\circ\iota = \alpha$.
		\[
			\begin{tikzcd}
				U && W \\
				\module
				\arrow["\beta", dashed, from=1-3, to=2-1]
				\arrow["\alpha"', from=1-1, to=2-1]
				\arrow["\iota"', from=1-1, to=1-3]
				\arrow["\sigma"{description},bend right, from=1-3, to=1-1]
			\end{tikzcd}
		\]		
	\end{definition}

	For a given normed \groupmodule $\module$, we have already seen many relatively injective \groupmodule\s, namely $\Cbc[n]$ for all $n\geq0$:
	
	\begin{example}\label{exp:standard-rel-inj-module}
		Fix $n\geq0$. The \groupmodule $\Cb[n]$ is relatively injective.\\
		Indeed, let $\iota:U\to W$ be a $\group$-morphism between \groupmodule\s $U$ and $W$, and denote by $\sigma:W\to \module$ the corresponding bounded $\ring$-linear map. For a given $\group$-morphism $\alpha:U\to \Cc[n]$ define $\beta:W\to \Cc[n]$ as follows:
		\[ \beta(w)(g_0,\dots,g_n) := \alpha(g_0^{-1}\sigma(g_0w))(g_0,\dots,g_n) \]
		for all $g_0,\dots,g_n\in\group$ and $w\in W$.
		
		Observe that 
		\begin{align*}
			\beta(\iota(v))(g_0,\dots,g_n) &= \alpha(g_0^{-1}\sigma(g_0\iota(v)))(g_0,\dots,g_n)= \alpha(g_0^{-1}\sigma(\iota(g_0v)))(g_0,\dots,g_n) \\
			&= \alpha(g_0^{-1}g_0v)(g_0,\dots,g_n) = \alpha(v)(g_0,\dots,g_n) 
		\end{align*}
		for all $g_0,\dots,g_n\in\group$ and $v\in \module$. The second equality follows from the $\group$-equivariance of $\iota$. Hence $\beta\circ\iota = \alpha$. Moreover, it is easy to see that $\norm[]{\beta} \leq \norm[]{\alpha}$. Finally, we verify that $\beta$ is $\group$-equivariant:
		\begin{align*}
		\beta(gw)(g_0,\dots,g_n) &= \alpha(g_0^{-1}\sigma(g_0gw))(g_0,\dots,g_n) =  \alpha(g(g_0g)^{-1}\sigma(g_0gw)))(g_0,\dots,g_n) \\ 
		&=(g\cdot\alpha((g_0g)^{-1}\sigma(g_0gw)))(g_0,\dots,g_n) \\
		&= g\alpha((g_0g)^{-1}\sigma(g_0gw))(g^{-1}g_0,\dots,g^{-1}g_n) \\ &= g\beta(w)(g^{-1}g_0,\dots,g^{-1}g_n) =  (g\cdot\beta(w))(g_0,\dots,g_n)
		\end{align*}
		for all $g,g_0\in\group$ and $w\in W$. Here we used that $\alpha$ is $\group$-equivariant in the third step.
	\end{example}

	\begin{definition}[\groupcomplex]
		A normed \highlight{\groupcomplex} is a sequence $(\module^n)_{n\geq0}$ of normed \groupmodule\s and $\group$-morphisms $d^n:\module^n\to \module^{n+1}$, s.t. $d^{n+1}\circ d^n = 0$ for all $n\geq0$:
		\[ \ChainComplex[0]{\module^0, \module^1, \module^2}[][d^][0] \]
		We denote this normed \groupcomplex by $(\module^\bullet,d^\bullet)$.
	\end{definition}
	
	For a normed \groupcomplex $(\module^\bullet,d^\bullet)$, we define its bounded cohomology $\HHb[][\module^\bullet]$ as follows. Set $\Zb[n] := \ker(d^n)\cap(\module^n)^\group$ and $\Bb[n] := d^{n-1}((\module^{n-1})^\group)$ and define
	\[ \HHb[n][\module^\bullet] := \frac{\Zb[n]}{\Bb[n]}. \]
	Note that the norm on $\module^n$ restricts to a norm on $\Zb[n]$. However, this induces a priori only a seminorm on $\HHb[n][\module^\bullet]$, which is a norm if and only if $\Bb[n]$ is closed in $\Zb[n]$.
	
	\begin{remark}
		Observe that of course $(\Cbc,\codiff)$ is a normed \groupcomplex and that the two notations for bounded cohomology coincide 
		\[ \Zbc[n] = \Zb[n][\Cbc], \quad \Bbc[n] = \Bb[n][\Cbc], \quad \HHbc[n] = \HHb[n][\Cbc] \]
		for all $n\geq0$.\\
		The seminorm on $\HHbc$ is called \highlight{canonical seminorm}.
	\end{remark}
	
	\begin{definition}[resolution]
		A \highlight{resolution} of a normed \groupmodule $\module$ is a \groupcomplex $(\module^\bullet,d^\bullet)$ together with a $\group$-morphism $\epsilon:\module\to \module^0$, such that the augmented sequence
		\[ \ALES{\module}[]{\module^0, \module^1, \module^2}[d^] \]
		is exact. The function $\epsilon$ is called augmentation map.\\
		We denote such a resolution of $\module$ with \resolution.\\
	\end{definition}

	We say that \resolution is a \highlight{strong} resolution, if there is a contracting homotopy $\Theta_{\bullet}$ with $\norm[]{\Theta_n}\leq 1$ for all $n\geq0$. More precisely, it is a sequence of (norm one) bounded linear operators $\Theta_{n}:\module^n\to\module^{n-1}$ for $n\geq0$ (here $\module^{-1} := \module$), such that $d^{n-1}\Theta_{n} + \Theta_{n+1}d^{n} = id_{\module^{n}}$ for $n\geq0$  (here $d^{-1} := \epsilon$).
	\[ \revChainComplex[NO]{\module, \module^0, \module^1, \module^2}[][\Theta_]  \]
	Finally, we say that a resolution \resolution is \highlight{relatively injective}, if $\module^n$ is relatively injective for all $n\geq1$

	\begin{lemma}\label{lem:standard-resolution-of-module}
		For a \groupmodule $\module$ define $\epsilon:\module\to\Cc[0]$ by $\tau(v)(g) = v$ for all $g\in\group, v\in \module$. Then \resolution[V][\tau][\Cbc][{\codiff[]}] is a relatively injective strong resolution of $\module$.\\
		We call \resolution[V][\tau][\Cbc][{\codiff[]}] the standard resolution of $\module$.
	\end{lemma}
	\begin{proof}
		We already know that $\Cbc[n]$ is relatively injective for all $n\geq0$. Hence, we only have to find a contracting homotopy $\Theta_{\bullet}$ with $\norm[]{\Theta_n}\leq 1$ for all $n\geq 0$.\\
		It can be checked that $\Theta_0:\Cbc[0]\to\module$, $\Theta_0(f) = f(1) \in V$, together with $\Theta_n:\Cbc[n]\to\Cbc[n-1]$, $\Theta_n(f)(g_1,\dots,g_n) = f(1,g_1,\dots,g_n)$ for all $g_1,\dots,g_n\in\group$ and $f\in\Cbc[n]$, for $n\geq1$ gives such a contracting homotopy (see \cite[Section 4.7]{2017Frigerio} for more details).
	\end{proof}

	The following is a result in homological algebra established in \cite{Ivanov1987FoundationsOT,2017Frigerio}:
	
	\begin{proposition}\label{prop:extension-between-resolutions}
		Let $\alpha:\module\to W$ be a $\group$-morphism of normed \groupmodule\s. If \resolution is a relatively injective strong resolution and \resolution[W][\eta][][e] is a relatively injective resolution, then we can extend $\alpha$ to a bounded $\group$-invariant chain map $\alpha^\bullet:\module^\bullet\to W^\bullet$.\\
		Moreover, any two such extensions are chain homotopic via bounded $\group$-equivariant maps. In particular, $\alpha$ induces a well-defined map
		\[ \HHbf[\alpha]:\HHb[][\module^\bullet]\to\HHb[][W^\bullet]\]
		in bounded cohomology, which depends only on $\alpha$ and not on the extension.
	\end{proposition}
	
	From this we can immediately deduce the following corollary, which enables another way of calculating the bounded cohomology of a group $\group$ besides the approach described in \Cref{sec:basic-def-of-BC}.
	
	\begin{corollary}\label{cor:BC-via-resolution}
		If \resolution is a relatively injective strong resolution of a normed \groupmodule $\module$, then for every $n\geq0$ there is a canonical isomorphism 
		\[ \HHb[n][\module^\bullet] \cong \HHbc[n]  \]
		in bounded cohomology.\\
		In particular $\HHb[n][\module^\bullet]$ depends only on the group $\group$ and not on the $\groupring$-complex $(\module^\bullet,d^\bullet)$.
		Further, this isomorphism is bi-Lipschitz with respect to the seminorms on $\HHb[n][\module^\bullet]$ and $\HHbc[n]$.
	\end{corollary}

	Since this is an important result in bounded cohomology, the proof is elaborated in more detail. 
	\begin{proof}
		First of all, due to \Cref{lem:standard-resolution-of-module} we know that \resolution[\module][\tau][\Cbc][{\codiff[]}] is a relatively injective strong resolution of $\module$ as well. Therefore, by \Cref{prop:extension-between-resolutions} we can extend the identity $id:\module\to\module$ (which is a $\group$-equivariant bounded map) to a bounded $\group$-invariant chain map $\alpha^\bullet:\module^\bullet\to\Cbc$.\\
		\[
			\begin{tikzcd}
			V & {V^0} & {V^1} & {V^2} & {V^3} & \cdots \\
			V & {\Cbc[0]} & {\Cbc[1]} & {\Cbc[2]} & {\Cbc[3]} & \cdots
			\arrow["\epsilon", from=1-1, to=1-2]
			\arrow["{d^0}", from=1-2, to=1-3]
			\arrow["{d^1}", from=1-3, to=1-4]
			\arrow["{d^2}", from=1-4, to=1-5]
			\arrow["{d^3}", from=1-5, to=1-6]
			\arrow["\tau", from=2-1, to=2-2]
			\arrow["{\codiff[0]}", from=2-2, to=2-3]
			\arrow["{\codiff[1]}", from=2-3, to=2-4]
			\arrow["{\codiff[2]}", from=2-4, to=2-5]
			\arrow["{\codiff[3]}", from=2-5, to=2-6]
			\arrow["id", from=1-1, to=2-1]
			\arrow["{\alpha^0}", from=1-2, to=2-2]
			\arrow["{\alpha^1}", from=1-3, to=2-3]
			\arrow["{\alpha^2}", from=1-4, to=2-4]
			\arrow["{\alpha^3}", from=1-5, to=2-5]
			\end{tikzcd}
		\]
		
		If $\Theta^\bullet$ is a contracting homotopy of $\module^\bullet$,
		\[ \revChainComplex[NO]{\module, \module^0, \module^1, \module^2}[][\Theta_], \]
		considering \Cref{exp:standard-rel-inj-module}, we can even give an inductive formula for $\alpha^\bullet$ as follows:
		\begin{align*}
			\alpha^0(v_0)(g_0) &:= \tau\circ id(g_0\Theta_0(g_0^{-1}v_0))(g_0) = g_0\Theta_0(g_0^{-1}v_0)\\
			\alpha^n(v_n)(g_0,\dots,g_n) &:= \codiff[n-1]\circ \alpha^{n-1}(g_0\Theta_n(g_0^{-1}v_n))(g_0,\dots,g_n).
		\end{align*}
		These formulas can be depicted by the following two diagrams, where $\group\cdot\Theta_0$ and $\group\cdot\Theta_n$ represent the action of $\group$ on $\Theta_0$ by the element $\group_0$ from $\group$, respectively on $\Theta_n$ by the first coordinate $\group_0$ of the element $(g_0,\dots,g_n)$ from $\group^{n+1}$
		\[
			\begin{tikzcd}[execute at end picture={\draw[->] (-4.5,1) to[out=150,in=30,loop,looseness=2] node[midway,above]{\tiny{$g_0$}} (-3.6,1); \draw[->] (2.5,1) to[out=150,in=30,loop,looseness=2] node[midway,above]{\tiny{$g_0$}} (3.4,1);}]
				V^{ } & {V^0} && {V^{n-1}} & {V^n} \\
				V & {\Cbc[0]} && {\Cbc[n-1]} & {\Cbc[n]}
				\arrow["id"', from=1-1, to=2-1]
				\arrow["{\group\cdot\Theta_0}"', from=1-2, to=1-1]
				\arrow["\tau"', from=2-1, to=2-2]
				\arrow["{\alpha^0}", dashed, from=1-2, to=2-2]
				\arrow["{\alpha^{n-1}}"', from=1-4, to=2-4]
				\arrow["{\group\cdot\Theta_n}"', from=1-5, to=1-4]
				\arrow["{\alpha^n}", dashed, from=1-5, to=2-5]
				\arrow["{\codiff[n-1]}"', from=2-4, to=2-5]
			\end{tikzcd}
		\]

		Similarly, the identity $id:\module\to\module$ can be extended to a bounded $\group$-invariant chain map $\beta^\bullet:\Cbc\to\module^\bullet$
		\[
			\begin{tikzcd}
				V & {\Cbc[0]} & {\Cbc[1]} & {\Cbc[2]} & {\Cbc[3]} & \cdots\\
				V & {V^0} & {V^1} & {V^2} & {V^3} & \cdots
				\arrow["\tau", from=1-1, to=1-2]
				\arrow["{ \codiff[0]}", from=1-2, to=1-3]
				\arrow["{\codiff[1]}", from=1-3, to=1-4]
				\arrow["{\codiff[2]}", from=1-4, to=1-5]
				\arrow["{\codiff[3]}", from=1-5, to=1-6]
				\arrow["\epsilon", from=2-1, to=2-2]
				\arrow["{d^0}", from=2-2, to=2-3]
				\arrow["{d^1}", from=2-3, to=2-4]
				\arrow["{d^2}", from=2-4, to=2-5]
				\arrow["{d^3}", from=2-5, to=2-6]
				\arrow["id", from=1-1, to=2-1]
				\arrow["{\beta^0}", from=1-2, to=2-2]
				\arrow["{\beta^1}", from=1-3, to=2-3]
				\arrow["{\beta^2}", from=1-4, to=2-4]
				\arrow["{\beta^3}", from=1-5, to=2-5]
			\end{tikzcd}
		\]
		
		The identity $id:\module\to\module$ clearly induces the bounded $\group$-invariant chain map $id_{\module^\bullet}^\bullet:\module^\bullet\to\module^\bullet$. However, since $\beta^0\circ\alpha^0$ is the identity on $\module$ too, the chain maps $id_{\module^\bullet}$ and $\beta^\bullet\circ\alpha^\bullet$ are chain homotopic (see \Cref{prop:extension-between-resolutions}) and hence induce the same map in bounded cohomology, that is $\HHbf[\beta^\bullet]\circ\HHbf[\alpha^\bullet] = \HHbf[\beta^\bullet\circ\alpha^\bullet] = \HHbf[id_{\module^\bullet}^\bullet] = id_{\HHb[][\module^\bullet]}$.\\
		Similarly we can deduce that $\HHbf[\alpha^\bullet]\circ\HHbf[\beta^\bullet] = id_{\HHbc}$. Hence, $\HHbf[\alpha^\bullet]$ and $\HHbf[\beta^\bullet]$ are isomorphisms and inverses of one another.\\
		Hence we get the desired isomorphism
		\[ \HHbf[\alpha^n]:\HHb[n][\module^\bullet] \xrightarrow{\cong} \HHbc[n] \quad \text{for all } n\geq0 \]
		and the map is bi-Lipschitz, since $\HHbf[\alpha^n]$ is bounded and it has a bounded inverse $\HHbf[\beta^n]$.
	\end{proof}

	Note that every relatively injective strong resolution \resolution of a \groupmodule $\module$ induces a possibly non-isometric seminorm on $\HHbc$. However, since the isomorphisms constructed in the previous \namecref{cor:BC-via-resolution} are bi-Lipschitz maps, all of those seminorms are equivalent. Nevertheless, it is an interesting question to ask which resolutions calculate the canonical seminorm exactly and how to characterize these distinct resolutions. It is possible to strengthen the result of \Cref{cor:BC-via-resolution} \cite[Theorem 7.3.1]{monod2001continuous},\cite[Theorem 4.16]{2017Frigerio} and conclude that the constructed isomorphism does not increase the seminorm. Thus the canonical seminorm on $\HHbc$ is the infimum of all those seminorms which are induced by relatively injective strong resolutions of the \groupmodule $\module$.
	
	\begin{theorem}\cite[Theorem 7.3.1]{monod2001continuous}\label{thm:norm-non-increasing-map-between-resolutions}
		Let $\module$ be a normed \groupmodule\s. If \resolution is a relatively injective strong resolution, then the identity map $id$ on $V$ can be extended to a bounded $\group$-invariant chain map $\alpha^\bullet:\module^\bullet\to \Cbc$ such that $\norm{\alpha^n}\leq1$ for all $n\geq0$.\\
		In particular, $\alpha$ induces a well-defined \highlight{norm non-increasing} map
		\[ \HHbf[\alpha]:\HHb[][\module^\bullet]\to\HHbc\]
		in bounded cohomology, which depends only on $\alpha$ and not on the extension.
	\end{theorem}
	
	It is possible (\cite[Theorem 7.3.1]{monod2001continuous}) to define $\alpha^n$ inductively such that the norm $\norm{\alpha^n}$ is always smaller or equal than $1$. This will then imply that the map $\HHbf(\alpha)$ in bounded cohomology induced by $\alpha$ is indeed norm non-increasing.
	
	\begin{corollary}\cite[Theorem 4.17]{2017Frigerio}\label{cor:norm-non-increasing-map-between-resolutions}
		Let $\module$ be a normed \groupmodule\s and \resolution a relatively injective strong resolution. If the identity map $id$ on $V$ can be extended to a bounded $\group$-invariant chain map $\alpha^\bullet:\Cbc\to\module^\bullet$ such that $\norm{\alpha^n}\leq1$ for all $n\geq0$, then $\alpha^\bullet$ induces an \highlight{isometric} isomorphism
		\[ \HHbf[\alpha]:\HHbc\to\HHb[][\module^\bullet] \]
		in bounded cohomology.\\
		In particular, resolutions \resolution with this property calculate the canonical seminorm exactly.
	\end{corollary}
	
	So far, for a given group $\group$, we only considered normed \ringmodule\s $\module$, for which we had an action by $\ring$-linear maps on $\module$ from the entire group $\group$. However, there is a natural way of constructing a normed \groupmodule out of a normed \groupmodule[][\subgroup] $\module$ for some subgroup $\subgroup\leq\group$. It turns out that this construction leads to an (isometric) isomorphism in bounded cohomology (see \Cref{prop:Eckmann-Shapiro}). 
	
	\begin{definition}[Induction module]\label{def:induction-module}
		Let $\group$ be a group and $\subgroup\leq\group$ a subgroup. Moreover, let $\module$ be a normed \groupmodule[][\subgroup]. Consider the normed \groupmodule[][\subgroup] $\linf{\group}[\module]$ of bounded functions from $\group$ to $\module$ endowed with the diagonal $\subgroup$-action. Then the space of $\subgroup$-invariants $\inductionmodule := \linf{\group}[\module]^\subgroup$ can be equipped with the \highlight{right translation} $\group$-action, i.e. $(g\cdot f)(\groupel) := f(\groupel g)$ for all $g,\groupel\in\group$, which gives $\inductionmodule$ the structure of a normed \groupmodule.\\
		This normed \groupmodule $\inductionmodule$ is called \highlight{induction module}.
	\end{definition}

	\begin{remark}\label{rem:equiv-induction-module-to-left-cosets}
		Of course, for a subgroup $\subgroup\leq\group$, every normed \groupmodule $\module$ is also a normed \groupmodule[][\subgroup] (by restriction of the action). Hence, it is natural to ask what the relation between the normed \groupmodule $\module$ and the induction module $\inductionmodule$ is. We have that the map
		\begin{align*}
			\alpha:\inductionmodule&\to\linf{\group/\subgroup}[V]\\
			f&\mapsto\alpha(f)(g\subgroup) := gf(g^{-1})
		\end{align*}
		is a $\group$-equivariant isomorphism, where the $\group$-action on $\linf{\group/\subgroup}[V]$ is the usual diagonal $\group$-action.\\
		First, observe that $\alpha(f)$ is well-defined, because 
		\[ \alpha(f)(gh\subgroup) = ghf(h^{-1}g^{-1}) = g(h\cdot f)(g^{-1}) = gf(g^{-1}) = \alpha(f)(g\subgroup), \]
		for all $h\in\subgroup$, where the second last equality follows from $f$ being $\subgroup$-invariant with respect to the diagonal $\subgroup$-action. Moreover, $\alpha(f)$ is clearly bounded, hence $\alpha(f)\in\linf{\group/\subgroup}[V]$.\\
		Second, $\alpha$ is indeed $\group$-equivariant:
		\[ \alpha(g\cdot f)(\groupel\subgroup) = \groupel(g\cdot f)(\groupel^{-1}) = \groupel f(\groupel^{-1}g) = gg^{-1}\groupel f(\groupel^{-1}g) = g\alpha(f)(g^{-1}\groupel) = (g\cdot\alpha(f))(\groupel) \]
		for all $g,\groupel\in\group$, $f\in\inductionmodule$. Here, the first $\group$-action $g\cdot f$ is the right translation action, whereas the last $\group$-action $g\cdot\alpha(f)$ is the diagonal $\group$-action.\\
		Lastly, the map $\beta:\linf{\group/\subgroup}[V]\to\inductionmodule$, $f\mapsto\beta(f)(g) := gf(g^{-1}\subgroup)$ is an inverse of $\alpha$, as can be checked easily.
	\end{remark}
	
	\begin{proposition}[Eckmann--Shapiro Lemma {\cite[Proposition 10.1.3]{monod2001continuous}}]\label{prop:Eckmann-Shapiro}
		Let $\group$ be a discrete group, $\subgroup\leq\group$ a subgroup and $\module$ an \groupmodule[][\subgroup]. Then the following chain map
		\begin{align*}
			 i^n:\Cbc[n]^\subgroup & \longrightarrow \Cbc[n][][\inductionmodule]^\group \\
			 i^n(f)(g_0,\dots,g_n)(g) &:= f(gg_0,\dots,gg_n)
		\end{align*}
		induces an isometric isomorphism in bounded cohomology, i.e.
		\[ \HHbf[i]:\HHbc[][\subgroup] \xrightarrow{\cong} \HHbc[][][\inductionmodule] \]
	\end{proposition}
	\begin{proof}
		The maps $i^n$ are well defined, because firstly
		\begin{align*}
			&h\cdot (i^n(f)(g_0,\dots,g_n))(g) = hi^n(f)(g_0,\dots,g_n))(h^{-1}g)\\
			&= hf(h^{-1}gg_0,\dots,h^{-1}gg_n)) = h\cdot f(gg_0,\dots,gg_n)\\
			&= f(gg_0,\dots,gg_n) = i^n(f)(g_0,\dots,g_n)(g)
		\end{align*}
		for all $h\in\subgroup$, $g_0,\dots,g_n\in\group$ and $f\in\Cb[n]^\subgroup$, since $f\in\Cb[n]^\subgroup$ is $\subgroup$-invariant. Hence, $(i^n(f)(g_0,\dots,g_n)$ is in $\inductionmodule$ for all $f\in\Cb[n]^\subgroup$ and $g_0,\dots,g_n\in\group$. Secondly
		\begin{align*}
			&(g\cdot i^n(f))(g_0,\dots,g_n)(g') =
			g\cdot i^n(f)(g^{-1}g_0,\dots,g^{-1}g_n)(g')\\
			&= i^n(f)(g^{-1}g_0,\dots,g^{-1}g_n)(g'g) = f(g'gg^{-1}g_0,\dots,g'gg^{-1}g_n) = i^n(f)(g_0,\dots,g_n)(g'),
		\end{align*}
		so  $i^n(f)$ is also $\group$-invariant.\\
		The fact that $i^\bullet$ is a chain map is also an easy verification.\\
		Moreover, it can be shown that the maps
		\begin{align*}
			j^n: \Cbc[n][][\inductionmodule]^\group &\longrightarrow \Cbc[n]^\subgroup \\
			j^n(f)(g_0,\dots,g_n) &:= f(1,g_0^{-1}g_2,\dots,g_0^{-1}g_n)(g_0)
		\end{align*}
		form a bounded chain map which is the inverse of $i^\bullet$ on the cochain level. Hence, $\HHbf[i^\bullet]$ is an isomorphism. This isomorphism is also isometric, since it already is isometric on the cochain level:
		\begin{align*}
			\norm{i^n(f)} &= \sup\{\norm{i^n(f)(g_0,\dots,g_n)}~|~g_0,\dots,g_n\in\group\}\\
			&= \sup\{f(gg_0,\dots,gg_n)~|~g,g_0,\dots,g_n\in\group\}\\
			&= \sup\{f(g_0,\dots,g_n)~|~g_0,\dots,g_n\in\group\} = \norm{f}.
		\end{align*}
		The second last equality is easily seen to be true.
	\end{proof}
	
	In the following we will need the result of the previous \namecref{prop:Eckmann-Shapiro} especially in the case of $\module$ being the trivial \groupmodule[\RR][] $\RR$.
	
	\begin{corollary}\label{cor:Eckmann--Shapiro-for-RR}
		There is an isometric isomorphism in bounded cohomology
		\[ \HHbR[][\subgroup] \xrightarrow{\cong} \HHbc[][][{\inductionmodule[][][\RR]}] = \HHbc[][][\linf{\group}^\subgroup] \]
	\end{corollary}

\subsection{Amenability and Bounded Acyclicity}
	
	\begin{definition}[bounded acyclicity]\label{def:bAc}
		Let $\group$ be a group and $n \geq 1$. We call $\group$ \highlight{\bAc[n]} if $\HHbR[i] \cong 0$ for all $1 \leq i \leq n$.\\
		Moreover, $\group$ is called \highlight{\bAc} or \highlight{\bAc[\infty]} if it is \bAc[n] for all $n\geq 1$.
	\end{definition}
	
	\begin{remark}[{$\HHbR[0]$ and $\HHbR[1]$}]
		Note that we can easily calculate the bounded cohomology with real coefficients $\RR$ in degrees $0$ and $1$ for every group $\group$. Here we understand $\RR$ as trivial \groupmodule[\RR].\\
		Indeed, by definition we have 
		\[ \HHbR[0] = \frac{\ZbR[0]}{\BbR[0]} \quad\text{and}\quad \HHbR[1] = \frac{\ZbR[1]}{\BbR[1]}. \]
		Also by definition $\BbR[0] = 0$ and $\ZbR[0] = \{ f:\CbR[0]^\group~|~\codiff[0](f) = 0 \}$. Since $\group$ acts transitively on itself, a function $f:\group\to\RR$ is in $\CR[0][]^\group$ if and only if it is constant, hence it is automatically also bounded. Therefore, $\CbR[0]^\group = \CR[0][]^\group \cong \RR$. However, then $\codiff[0](f)(\groupel_0,\groupel_1) = f(\groupel_0) - f(\groupel_1) = 0$ holds for all functions in $\CbR[0]^\group$ and we get $\HHbR[0] \cong \ZbR[0] = \CbR[0]^\group \cong \RR$.\\
		Regarding $\HHbR[1]$, we observe that $\BbR[1] = 0$, since $\codiff[0]$ is the zero map as was described above. Thus, also $\HHbR[1] \cong \ZbR[1]$ holds.\\
		Observe that for a $\group$-invariant map $f\in\CbR[1]$, it holds that
		\[ f(\groupel_0,\groupel_1) = (\groupel_0\cdot f)(\groupel_0,\groupel_1) = f(1,\groupel_0^{-1}\groupel_1) \]
		for all $\groupel_0,\groupel_1\in\group$. If $f$ is additionally in the kernel of $\codiff[1]$, then
		\begin{align*}
			0 &= \codiff[1](f)(\groupel_0,\groupel_1,\groupel_2) = f(\groupel_1,\groupel_2) - f(\groupel_0,\groupel_2) + f(\groupel_0,\groupel_1) \\
			&= f(1,\groupel_1^{-1}\groupel_2) - f(1,\groupel_0^{-1}\groupel_2) + f(1,\groupel_0^{-1}\groupel_1)
		\end{align*} 
		for all $\groupel_0,\groupel_1\groupel_2\in\group$. By choosing $\groupel_1 = 1$ we deduce that the following equation holds for all $\groupel_0,\groupel_2\in\group$:
		\[   f(1,\groupel_0^{-1}\groupel_2) = f(1,\groupel_0^{-1}) + f(1,\groupel_2). \]
		This means precisely that the map $\groupel \mapsto f(1,\groupel)$ is a bounded homomorphism from $\group$ to the additive group $\RR$. However, the only bounded homomorphism from any group into $(\RR,+)$ is the trivial one. Hence, $f\equiv 0$.\\
		To sum up, we get $\HHbR[0] = \RR$ and $\HHbR[1] = 0$ for all groups $\group$.
		In particular, every $\group$ is \bAc[1].
	\end{remark}
	
	However, of course there are groups which are not \bAc. One such example are non-abelian free groups. In \cite[§3]{brooks1981some}, the author showed that $\HHbR[2][F_2]$ contains infinitely many linearly independent elements, where $F_2$ denotes the free group with two generators. However, it was already known \cite{johnson1972cohomology}, that $\HHbR[2][F_2]\neq0$.
	
	\begin{proposition}[non-abelian free groups are not \bAc]\label{prop:non-abelian-free-group-is-not-bAc}
		Let $\group$ be a non-abelian free subgroup, then $\HHbR[2]$ is infinite dimensional.\\
		In particular, non-abelian free subgroups are not \bAc.
	\end{proposition}
	
	Nevertheless, there is a class of groups which is in some way indiscernible from the viewpoint of bounded cohomology, namely \highlight{amenable} groups. Indeed, for amenable groups the bounded cohomology with coefficients in many normed \groupmodule[\RR]\s vanishes in all positive degrees, see \Cref{prop:BC-of-amenable-groups}.
	
	For the following definition recall \Cref{APX_def:mean} the definition of a mean  on an arbitrary set.
	
	\begin{definition}[amenability] \label{def:amenability}
		A group $\group$ is called \highlight{amenable}, if it admits an \highlight{invariant mean} $m$, i.e. $m(\groupel\cdot f) = m(f)$ for all $\groupel\in\group$ and $f\in\linf{\group}$.
	\end{definition}

	\begin{remark}
		There are a lot of different definitions for amenability within the field of discrete groups. Also in the wider context of locally compact groups \cite{pier1984amenable} there are several definitions for amenability. Futher, there is also a notion of \highlight{amenable group actions} introduced by \citeauthor{Zimmer1978AmenableEG}, for example in \cite{ZIMMER198758,Zimmer1978AmenableEG}, which is also used in bounded cohomology (e.g. \cite{monod2001continuous}). However, the given definition is the most useful one in the context of bounded cohomology of discrete groups.
	\end{remark}
	
	\begin{example}\label{exp:abelian-and-finite-groups-are-amenable}
		All abelian as well as all finite groups are amenable.\\
		For a finite group $\group$ the \enquote{arithmetic mean} of a function $f\in\linf{\group}$ can serve as mean $m$ in the sense of amenability, i.e. $m(f) := \tfrac{1}{\abs{\group}}\sum_{\groupel\in\group}f(\groupel)$. This is easily seen to be an invariant mean.\\
		To prove the existence of a mean for abelian groups requires more effort and can, for example, be done by using the Markov--Kakutani Fixed Point Theorem \cite{paterson2000amenability, 2017Frigerio}.
	\end{example}
	
	A normed module $\module$ is a \highlight{dual} normed \groupmodule[\RR] if it is isomorphic to the topological dual $W^*:= \{f:W\to\ring~|~f \text{ is } \ring\text{-linear and continuous}\}$ of a normed \groupmodule[\RR] $W$. In particular, note that $\RR$ is isomorphic to $\RR^*$ (seen as trivial \groupmodule[\RR]).
	
	\begin{proposition}\label{prop:BC-of-amenable-groups}
		Let $\group$ be an amenable group and $\module$ a \highlight{dual} normed \groupmodule[\RR]. Then $\HHbc[n] = 0$ for all $n\geq1$. In particular, amenable groups are \bAc.
	\end{proposition}
	\begin{proof}
		Recall that we can calculate $\HHbc[n]$ via the cohomology of the cochain complex
		\[ \ChainComplex[0]{\Cbc[0]^\group,\Cbc[1]^\group,\Cbc[2]^\group,\Cbc[3]^\group} \]
		We will prove this theorem by constructing a bounded chain homotopy between the identity and the zero map of $\Cbc^\group$ using the $\group$-invariant mean $\mean$ on $\linf{\group}$ and the fact that $V$ is isomorphic to the topological dual normed \groupmodule[\RR] of some \groupmodule[\RR] $W$. For ease of exposition, we assume $\module = W^*$.\\
		For all $n\geq1$ define 
		\begin{align*}
			\Theta_n:\Cbc[n]^\group&\to\Cbc[n-1]^\group\\
			(f:\group^{n+1}\to W^*) &\mapsto (\Theta_n(f):\group^n\to W^*)
		\end{align*}
		such that $\Theta_n(f)$ maps a tuple $(g_0,\dots,g_{n-1})\in\group^n$ to the function
		\[ w \mapsto \Theta_n(f)(g_0,\dots,g_{n-1})(w):= \mean(g'\mapsto f(g',g_0,\dots,g_{n-1})(w)) \]
		for all $w\in W$. In other words, we define $\Theta_n(f)(g_0,\dots,g_{n-1})\in W^*$ to be the function which for a fixed $w\in W$ \enquote{averages} the value of $f(g',g_0,\dots,g_{n-1})(w)$ over all $g'\in\group$. Moreover, we set $\Theta_0:\Cbc[0]^\group\to0$ to be the zero map.
		
		It is easy to see that $\Theta_n$ is well-defined and bounded. Indeed, $g'\mapsto f(g',g_0,\dots,g_{n-1})(w)$ is in $\linf{\group}$ and $\Theta_n(f)(g_0,\dots,g_{n-1})$ is bounded by the norm of $g'\mapsto f(g',g_0,\dots,g_{n-1})$, which is finite as $f\in\Cb[n]^\group$. It remains to be shown that $\Theta_n$ is a $\group$-morphism and that $id = \codiff[n-1]\circ\Theta_n + \Theta_{n+1}\circ\codiff[n]$\\
		We have for all $n\geq0$, $g,g_0,\dots,g_{n-1}\in\group$ and $w\in W$
		\begin{align*}
			\Theta_n(g\cdot f)(g_0,\dots,g_{n-1})(w) &= \mean(g'\mapsto g\cdot f(g',g_0,\dots,g_{n-1})(w))\\
			&=\mean(g'\mapsto gf(g^{-1}g',g^{-1}g_0,\dots,g^{-1}g_{n-1})(w))\\
			&= \mean(g\cdot(g'\mapsto f(g',g^{-1}g_0,\dots,g^{-1}g_{n-1})(w)))\\
			&= g(g^{-1}\cdot\mean(g'\mapsto f(g',g^{-1}g_0,\dots,g^{-1}g_{n-1})(w)))\\
			&=g\mean(g'\mapsto f(g',g^{-1}g_0,\dots,g^{-1}g_{n-1})(w))\\
			&=g\Theta_n(f)(g^{-1}g_0,\dots,g^{-1}g_{n-1})(w)) = g\cdot\Theta_n(f)(g_0,\dots,g_{n-1})(w),
		\end{align*}
		since $\mean$ is $\group$-invariant. Thus, $\Theta_n$ is a $\group$-morphism.\\
		Finally, we check that $id = \codiff[n-1]\circ\Theta_n + \Theta_{n+1}\circ\codiff[n]$:
		\begin{align*}
			(\codiff[n-1]\circ\Theta_n &+ \Theta_{n+1}\circ\codiff[n])(f)(g_0,\dots,g_n)(w)=\\ &=\codiff[n-1]\circ\Theta_n(f)(g_0,\dots,g_n)(w) + \Theta_{n+1}\circ\codiff[n](f)(g_0,\dots,g_n)(w)\\
			&=\codiff[n-1]\mean(g'\mapsto f(g',g_0,\dots,g_i,\dots,g_n)(w)) + \mean(g'\mapsto \codiff[n](f)(g_0,\dots,g_n)(w))\\
			&= f(g_0,\dots,g_n)(w)
		\end{align*}
		which follows by the linearity of $\mean$ and that constant functions are sent to the corresponding constant value under $\mean$.\\
		This proves that $\HHbc[n] = 0$ for all $n\geq1$.
	\end{proof}

	\begin{corollary}\label{cor:abelian-and-finite-groups-are-bAc}
		Abelian and finite groups are \bAc.
	\end{corollary}
	
	Amenable groups have a lot of good properties. One such property is that they behave well with taking extensions.
	
	\begin{proposition}\label{prop:extensions-amenability}
		Let $\GrpSES[N, \group, Q]$ be a short exact sequence of groups. Then $\group$ is amenable if and only if $N$ and $Q$ are amenable.
	\end{proposition}
	\begin{proof}
		Denote the maps from the short exact sequence by $\iota:N\to\group$ and $\pi:\group\to Q$. We prove this \namecref{prop:extensions-amenability} by explicitly constructing the required means.
		\enquote{$\Rightarrow$}: Let $\mean_\group$ be the $\group$-invariant mean on $\linf{\group}$. Define the linear functional
		\begin{align*}
			\mean_Q:\linf{Q}&\to\RR,\\
			(f:Q\to\RR)&\mapsto (\mean_Q(f) := \mean_\group(f\circ\pi)).
		\end{align*}
		Then $\norm[op]{\mean_Q}\leq\norm[op]{\mean_\group}=1$ and this together with $\mean_Q(\mathbbm{1}_Q) = \mean_\group(\mathbbm{1}_Q\circ\pi) = \mean_\group(\mathbbm{1}_\group) = 1$ yields $\norm[op]{\mean_Q} = 1$. Further, if $f\geq0$, then $f\circ\pi\geq0$ and hence $\mean_Q(f) = \mean_\group(f\circ\pi) \geq 0$. Thus, $\mean_Q$ is indeed a mean on $\linf{Q}$. The $Q$-invariance follows from
		\begin{align*}
			\mean_Q(q\cdot f) = \mean_\group((q\cdot f)\circ\pi) = \mean_\group(g\cdot (f\circ\pi)) = \mean_\group(f\circ\pi) = \mean_Q(f) 
		\end{align*}
		for all $q\in Q$ where $g$ is any preimage of $q$ under $\pi$ which exists, since $\pi$ is surjective. Therefore, $Q$ is amenable.\\
		To construct an $N$-invariant mean on $\linf{N}$ we first choose some preimages of $\pi$, i.e. we fix a function $s:Q\to\group$ such that $\pi\circ s(q) = q$ for all $q\in Q$. Note that $s$ is injective but in general not a homomorphism. Define the linear functional
		\begin{align*}
			\mean_N:\linf{N}&\to\RR,\\
			(f:N\to\RR)&\mapsto (\mean_N(f) := \mean_\group(g\mapsto f(\iota^{-1}(gs(\pi(g)^{-1}))))).
		\end{align*}
		$\mean_N$ is well defined since $gs(\pi(g)^{-1})\in\ker(\pi)=\im(\iota)$. We have $\norm[op]{\mean_N} \leq \norm[op]{\mean_\group}=1$ and this together with $\mean_N(\mathbbm{1}_N) = \mean_\group(g\mapsto f(\iota^{-1}(gs(\pi(g)^{-1}))) = \mean_\group(\mathbbm{1}_\group) = 1$ yields $\norm[op]{\mean_N} = 1$. Finally, if $f\geq0$, then $g\mapsto f(\iota^{-1}(gs(\pi(g)^{-1}))\geq0$ and hence $\mean_N(f) \geq 0$. Thus, $\mean_N$ is indeed a mean on $\linf{N}$. The $N$-invariance follows from
		\begin{align*}
		\mean_N(n\cdot f) & = \mean_\group(g\mapsto f(n^{-1}\iota^{-1}(gs(\pi(g)^{-1})))\\
		&= \mean_\group(g\mapsto f(\iota^{-1}(\iota(n)^{-1}gs(\pi(g)^{-1})))\\
		&= \mean_\group(\iota(n)\cdot (g\mapsto f(\iota^{-1}(gs(\pi(g)^{-1}))))  = \mean_\group(f) 
		\end{align*}
		for all $n\in N$ where the second to last equality follows from $\iota(N) = \ker(\pi)$ and the last equality follows from the $\group$-invariance of $\mean_\group$. Therefore, $N$ is amenable.\\
		
		\enquote{$\Leftarrow$}: Let $N$ and $Q$ be amenable. Fix again a function $s:Q\to\group$ such that $\pi\circ s(q) = q$ for all $q\in Q$. Define the linear functional 		
		\begin{align*}
			\mean_\group:\linf{\group}&\to\RR,\\
			(f:\group\to\RR)&\mapsto (\mean_\group(f) := \mean_Q(q\mapsto \mean_N(n\mapsto f(s(q)\iota(n))))).
		\end{align*}
		Then $\mean_\group$ is a $\group$-invariant mean on $\linf{\group}$. It is easy to check that $\norm[op]{\mean_\group} \leq \norm[op]{\mean_Q}=1$ and $\mean_\group(\mathbbm{1}) = 1$, hence $\norm[op]{\mean_\group} = 1$. Moreover, for $f\geq0$ we have that $n\mapsto f(s(q)\iota(n))\geq0$ which implies $q\mapsto \mean_N(n\mapsto f(s(q)\iota(n)))\geq0$ and hence $\mean_\group(f)\geq0$, using that $\mean_Q$ is a mean. The $\group$-invariance follows from the $N$-invariance of $\mean_N$ as follows. Fix $g\in\group$. Then for a fixed $q\in Q$ it holds that $s(q)^{-1}gs(\pi(g)^{-1}q)\in\ker(\pi)=\im(\iota)$. Therefore, $s(\pi(g)^{-1}q) = g^{-1}s(q)\iota(n')$ for some $n'$ in $N$. Hence, for all $g\in\group$
		\begin{align*}
			\mean(g\cdot f) &= \mean_Q(q\mapsto \mean_N(n\mapsto 	f(g^{-1}s(q)\iota(n)))\\
			&= \mean_Q(q\mapsto \mean_N(n\mapsto f(g^{-1}s(q)\iota(n')\iota(n)))\\
			&= \mean_Q(q\mapsto \mean_N(n\mapsto f(s(\pi(g)^{-1}q)\iota(n))) = \mean(f)
		\end{align*}
		using first that $\mean_N$ is $N$-invariant and then that $\mean_Q$ is $Q$-invariant.\\
		This verifies that $\group$ is amenable and finishes the proof.
	\end{proof}

	It can further be shown that the class of amenable groups is also closed under taking subgroups. Unfortunately, for the class of \bAc groups this is not true. Here are some examples of \bAc groups, which contain a non \bAc subgroup.
	
	\begin{example}\label{exp:non-bAc-subgroups}\hfill
		\begin{enumerate}[nosep]
			\item Since by \Cref{prop:non-abelian-free-group-is-not-bAc} non-abelian free subgroups are not \bAc, every \bAc $\group$ which contains a non-abelian free subgroup serves as example. Most of the non-amenable examples of \bAc groups have free subgroups, for example most \bAc homeomorphism groups and all mitotic groups.
			\item In \Cref{sec:BC-of-thompsons-group-V} we will show that Thompson's group $V$ is \bAc, however, Thompson's group $T\leq V$ is not \bAc (\Cref{rem:BC-of-Thompsons-group-F-and-T}). Also $\V$ and $T$ contain free subgroups.
			\item All algebraically closed groups are mitotic \cite[Corollary 4.4]{baumslag1980topology} and therefore \bAc \cite[Theorem 1.2]{2017_LOEH}. However, it is well-known that all countable groups can be embedded into a countable algebraically closed group. Hence any countable non \bAc group is a subgroup of some countable \bAc algebraically closed group.
		\end{enumerate}		
	\end{example}
	
	There is one further property we are going to need to prove the bounded acyclicity of a certain group. That is the notion of coamenability:
	
	\begin{definition}[coamenability] \label{def:coamenability}
		Let $\group$ be a group. A group $\subgroup\leq\group$ is called \highlight{coamenable}, if there is a $\group$-invariant mean on $\linf{\group/\subgroup}$.
	\end{definition}

	It is easy to see that if we have subgroups $\subgroup_0\leq\subgroup_1\leq\group$ and $\subgroup_0$ is coamenable in $\group$, then so is $\subgroup_1$. Indeed, the projection $\pi:\group/\subgroup_0\to\group/\subgroup_1$ pushes every $\group$-invariant mean $\mean_0$ on $\linf{\group/\subgroup_0}$ to a $\group$-invariant mean $\mean_1$ on $\linf{\group/\subgroup_1}$, by defining $\mean_1(f) := \mean_0(f\circ\pi)$. However, $H_0$ is in general not coamenable in $H_1$ \cite[Theorem 1]{monod2003coamenability}.
	
	\begin{remark}\label{rem:coamenable-characterization}
		Similarly as for amenability there are different characterizations for coamenability. For example, the existence of an invariant mean on $\linf{\group/\subgroup}$ is equivalent to the property that every continuous affine $\group$-action on a convex compact subset of a locally convex space with an $\subgroup$-fixed point has also a $\group$-fixed point \cite{Eymard1972}.
	\end{remark}

	\begin{proposition}[{\cite[Proposition 3]{monod2003coamenability}, \cite[Proposition 8.6.2, 8.6.6]{monod2001continuous}}]\label{prop:coamenable->bAc}
		Let $\group$ be a group. If $\subgroup\leq\group$ is coamenable, then the inclusion map $\subgroup\to\group$ induces an isometrically injective map $\HHbR \hookrightarrow \HHbR[][\subgroup]$ in bounded cohomology.\\
		In particular, if $\subgroup$ is \bAc, so is $\group$.
	\end{proposition}
	\begin{proof}
		By definition, we can isometrically calculate $\HHbR$ as the bounded cohomology of the normed cochain complex $(\CbR^\group,\codiff)$.\\
		However, via the inclusion map $\subgroup\to\group$ we can view $\CbR$ also as cochain complex of normed \groupmodule[\RR][\subgroup]\s. Therefore, the resolution \resolution[\RR][\epsilon][\CbR][{\codiff[]}] is also a relatively injective strong resolution of the (trivial) normed  \groupmodule[\RR][\subgroup] $\RR$. Hence, by \Cref{cor:BC-via-resolution}
		\[ \HHb[n][\CbR^\subgroup] \cong \HHbR[n][\subgroup] \]
		holds for all $n\geq0$. Additionally, the inclusion $\CbR[n][H]\hookrightarrow\CbR[n]$ for all $n\geq1$ is in norm bounded by $1$ and gives a bounded chain map extending the identity $id$ on $\RR$. Thus, by \Cref{cor:norm-non-increasing-map-between-resolutions} the isomorphism $\HHb[n][\CbR^\subgroup] \cong \HHbR[n][\subgroup]$ is in fact also isometric.\\
		Moreover, we have an obvious inclusion
		\[
			\begin{tikzcd}
				0 & {\CbR[0]^\group} & {\CbR[1]^\group} & {\CbR[2]^\group} & {\CbR[3]^\group} & \cdots \\
				0 & {\CbR[0]^\subgroup} & {\CbR[1]^\subgroup} & {\CbR[2]^\subgroup} & {\CbR[3]^\subgroup} & \cdots
				\arrow["", from=1-1, to=1-2]
				\arrow["{\codiff[0]}", from=1-2, to=1-3]
				\arrow["{\codiff[1]}", from=1-3, to=1-4]
				\arrow["{\codiff[2]}", from=1-4, to=1-5]
				\arrow["{\codiff[3]}", from=1-5, to=1-6]
				\arrow["", from=2-1, to=2-2]
				\arrow["{\codiff[0]}", from=2-2, to=2-3]
				\arrow["{\codiff[1]}", from=2-3, to=2-4]
				\arrow["{\codiff[2]}", from=2-4, to=2-5]
				\arrow["{\codiff[3]}", from=2-5, to=2-6]
				\arrow[hook, from=1-2, to=2-2]
				\arrow[hook, from=1-3, to=2-3]
				\arrow[hook, from=1-4, to=2-4]
				\arrow[hook, from=1-5, to=2-5]
			\end{tikzcd}
		\]
		of the cochain complex $\CbR^\group$ into the cochain complex $\CbR^\subgroup$ and we denote this bounded chain map by $\iota^\bullet$ and thus the corresponding map in cohomology with $\HHbf[\iota]:\HHbR[]\to\HHb[][\CbR^\subgroup]$. This map is conjugated to the map $\HHbR \hookrightarrow \HHbR[][\subgroup]$ induced by the inclusion via the isometric isomorphisms $\HHbR[] = \HHbR[]$ and $\HHb[][\CbR^\subgroup]\cong\HHbR[][\subgroup]$ (see \cite[Proposition 8.4.2]{monod2001continuous}). Observe that $\iota^n$ is in norm bounded by $1$ for all $n\geq0$, so $\HHbf[\iota]$ does not increase the seminorm. What is hence left to show is that $\HHbf[\iota]$ is injective and also isometric. This can be done by constructing a left inverse function $\HHb[][\CbR^\subgroup]\rightarrow \HHbR$.\\
		Let $\mean$ denote the $\group$-invariant mean on $\linf{\group/\subgroup}$. Note that, for fixed $f\in\CbR[n]^\subgroup$ and $g_0,\dots,g_n\in\group$, the map $g'\mapsto g'\cdot f(g_0,\dots,g_n)$ from $\group$ to $\RR$ factors through $\group/\subgroup$, because $f$ is $\subgroup$-invariant. Now the map $\rho^\bullet$, which is for all $n\geq0$ defined by
		\begin{align*}
			\rho^n:\CbR[n]^\subgroup&\to\CbR[n]^\group &&\\
			(f:\group^{n+1}\to\RR)&\mapsto(\rho^n(f):\group^{n+1}\to\RR)\\
			&\quad~\,(g_0,\dots,g_n) \mapsto m(g'\mapsto g'\cdot f(g_0,\dots,g_n))
		\end{align*}
		is a well-defined bounded chain map such that $\HHbf[\rho]\circ\HHbf[\iota] = id$.\\
		Before we can show that $\rho^\bullet$ is a bounded chain map, we need to check that $\rho^n$ is well-defined. That is, we have to verify that $\rho^n(f)$ is $\group$-invariant. Let $f\in\CbR[n]^\subgroup$ and $g,g_0,\dots,g_n\in\group$, then
		\begin{align*}
			g\cdot\rho^n(f)(g_0\dots,g_n) &= g\rho^n(f)(g^{-1}g_0\dots,g^{-1}g_n) = g\mean(g'\mapsto g'\cdot f(g^{-1}g_0\dots,g^{-1}g_n))\\
			&=g\mean(g'\mapsto g^{-1}(gg'\cdot f(g_0,\dots,g_n))) = g\mean(g^{-1}\cdot(g'\mapsto g'\cdot f(g_0,\dots,g_n)))\\ 
			&= g\cdot \mean(g'\mapsto g'\cdot f(g_0,\dots,g_n)) = \mean(g'\mapsto g'\cdot f(g_0,\dots,g_n))\\
			&= \rho^n(f)(g_0\dots,g_n)
		\end{align*}
		where the second last equality holds, because $\mean$ is $\group$-invariant. In order to see that $\rho^\bullet$ is a chain map we calculate
		\begin{align*}
			\codiff[n]\rho^n(f)(g_0\dots,g_{n+1}) &= \sum_{i=0}^{n+1} (-1)^i \rho(f)(g_0\dots,\widehat{g}_i,\dots,g_{n+1})\\
			&= \sum_{i=0}^{n+1} (-1)^i \mean(g' \mapsto g'\cdot f(g_0\dots,\widehat{g}_i,\dots,g_{n+1})).
		\end{align*}
		Moreover we exploit that the mean $\mean$ as well as the action of $\group$ on $\CbR[n]$ are $\RR$-linear to get
		\begin{align*}
			\rho^n(\codiff[n]f)(g_0\dots,g_{n+1}) &= \mean(g'\mapsto g'\cdot (\sum_{i=0}^{n+1} (-1)^i \rho(f)(g_0\dots,\widehat{g}_i,\dots,g_{n+1})))\\
			&= \sum_{i=0}^{n+1} (-1)^i \mean(g' \mapsto g'\cdot f(g_0\dots,\widehat{g}_i,\dots,g_{n+1}))
		\end{align*}
		and we see that the expressions coincide, so $\codiff[n]\circ\rho^n= \rho^n\circ\codiff[n]$. The fact that $\rho^n$ is a bounded map is easy to see.\\
		Finally, for $f\in\CbR[n]^\group$ we have
		\begin{align*}
			\rho^n(\iota^n(f))(g_0\dots,g_n) &= \mean(g'\mapsto g'\cdot \iota^n(f)(g_0\dots,\widehat{g}_i,\dots,g_n)) \\
			&= \mean(g'\mapsto \iota^n(g'\cdot f)(g_0\dots,\widehat{g}_i,\dots,g_n))\\
			&= \mean(g'\mapsto \iota^n(f)(g_0\dots,\widehat{g}_i,\dots,g_n))\\
			&= \iota^n(f)(g_0\dots,\widehat{g}_i,\dots,g_n) = f(g_0\dots,\widehat{g}_i,\dots,g_n)
		\end{align*}
		since $\iota$ is a $\group$-morphism, $f$ is $\group$-invariant and a mean evaluates constant functions to the corresponding constant value. Hence, $\rho^n\circ\iota^n$ is the identity even on the cochain level for all $n\geq0$, so we deduce that the map $\HHbf[\rho]\circ\HHbf[\iota]$ is also the identity in cohomology.\\
		Moreover, the map $\rho^n$ is in fact of norm one for all $n\geq0$, so $\HHbf[\rho]$ does not increase the seminorm either.\\
		This shows that $\HHbf[\iota]$ is isometrically injective and concludes the proof.
	\end{proof}
	
	\begin{remark}\label{rem:coamenable->vanishing-BC}
		The result above holds not only for $\RR$-coefficients but more generally for $\module$-coefficients for all dual normed \groupmodule[\RR]\s $\module$ (\cite[Proposition 8.6.6]{monod2001continuous}).
	\end{remark}

\subsection{Bounded Cohomology via Boundedly Acyclic Modules}\label{subsec:BC-via-bAc-modules}
	
	\begin{definition}[boundedly acyclic module]\label{def:bAc-modules}
		Let $\group$ be a group, $\module$ a normed \groupmodule[\RR] and $n \geq 1$. We say that $\module$ is \highlight{\bAc[n]} if $\HHbc[i] \cong 0$ for all $1 \leq i \leq n$.\\
		Moreover, $\module$ is \highlight{\bAc} or \highlight{\bAc[\infty]} if it is \bAc[n] for all $n\geq 1$.
	\end{definition}
	
	\begin{lemma}[\cite{fournierfacio2021binate}]\label{lem:transitive-actions->stabilizer}
		Let $\group$ be a group acting transitively on a space $X$. If the stabilizer of every point in $X$ is \bAc[n] then $\HHbc[k][][\linf{X}] = 0$ for all $1\leq k\leq n.$
	\end{lemma}
	\begin{proof}
		Fix an element $x\in X$ and denote the stabilizer of $x$ by $\subgroup$. Since $\group$ acts transitively on $X$ we have $\group/\subgroup \cong X$.
		Therefore
		\[ \HHbc[k][][\linf{X}] \cong \HHbc[k][][\linf{\group/\subgroup}] \cong \HHbc[k][][\linf{\group}^\subgroup] \cong \HHbR[k][\subgroup] = 0, \]
		where the second last isomorphism is due to $\linf{\group/\subgroup}\cong\linf{\group}^\subgroup$ (\Cref{rem:equiv-induction-module-to-left-cosets}) and the last isomorphism is exactly the one from \Cref{cor:Eckmann--Shapiro-for-RR}.
	\end{proof}
	
	\begin{remark}\label{rem:gen-of-bAc-stabilizer->bAc-module}
		The statement above even holds under weaker hypotheses that consider actions with infinitely many orbits fulfilling some technical condition. We refer the interested reader to \cite[Definition 5.13, Proposition 5.14]{li2022bounded}.\\
		However, in our situation it is enough to assume that we are in the special case of transitive actions in the sequel.
	\end{remark}
	
	Note that \Cref{lem:transitive-actions->stabilizer} allows to come up with various \bAc[(n)] modules for a given group $\group$. The only thing necessary to do is to find a transitive action on some set $X$ such that the stabilizer of a point is \bAc[(n)]. 
	This is especially useful in combination with the following result, which describes how bounded cohomology of a group can be calculated by considering certain resolutions of \bAc[i] modules (for varying $i$). 
	
	\begin{proposition}[{\cite[Proposition 2.5.4]{moraschini2021amenability}}]\label{prop:bounded-cohomology-through-acyclic-resolutions}
		Let $\group$ be a discrete group and $n\in\NN\cup\{\infty\}$. Suppose that we have given a cochain complex $C^\bullet$ of \groupmodule[\RR]\s such that $C^i$ is \bAc[(n-i)] for every $1\leq i \leq n$ and such that the sequence
		\[ \LES{\RR, C^0, C^1, C^2} \]
		is exact.\\
		Then there is a canonical isomorphism
		\[ \HH[i][C^{\bullet,\group}]\xrightarrow{\cong} \HHbR[i] \quad \text{for } 0\leq i\leq n \]
		as well as an injective map
		\[ \HH[n+1][C^{\bullet,\group}] \hookrightarrow \HHbR[n+1], \]
		where, as usual, $C^{\bullet,\group}$ denotes the cochain complex $\ChainComplex[0]{C^{0,\group}, C^{1,\group}, C^{2,\group}}$ of $\group$-invariants.
	\end{proposition}

	In comparison to \Cref{cor:BC-via-resolution} the isomorphisms obtained by this \namecref{prop:bounded-cohomology-through-acyclic-resolutions} might not be bi-Lipschitz maps. Hence, in general one cannot use this result if one wants to calculate the canonical norm of elements in $\HHbR$. However, if one only cares about vanishing results for bounded cohomology, as we do in this thesis, it is just as suitable.

\subsection{Dissipated and Binate Groups}
	
	In \cite{MR787909} \citeauthor{MR787909} gave the first example of a non-amenable group, which is \bAc, namely the group $\Homeo_c(\RR^n)$ of compactly supported homeomorphisms of $\RR^n$. Their proof is based upon a refinement of \citeauthor{MATHER1971297}'s proof of the acyclicity of $\Homeo_c(\RR^n)$ \cite{MATHER1971297}. A generalization of a construction made in their proofs leads to the larger class of dissipated and binate groups \cite{berrick2002topologist}, which give a unified approach to show that certain groups are \bAc. In this subsection we explain these concepts in more detail.
	
	The following is a generalization of compactly supported homeomorphisms of a given set:	
	
	\begin{definition}[boundedly supported group]\label{def:boundedly_supported}
		Let $X$ be a directed union of subsets $(X_i)_{i\in I}$ and let $\group$ act faithfully on $X$. Define $\group_i := \{g\in\group~|~ \supp(g) \subset X_i \}$ for all $i\in I$, where $\supp(g) := \{x\in X~|~ g\cdot x \neq x \}$ for all $g\in\group$.\\
		$\group$ is called \highlight{boundedly supported} if it is the directed union of all $\group_{i}$.
	\end{definition}
	
	\begin{definition}[dissipator]\label{def:dissipator}
		Let $\group$ be a boundedly supported group acting on the set $X$. Fix $i \in I$. A \highlight{dissipator} for $\group_i$ is an element $\groupel_i \in \group$ such that the following two conditions hold:
			\begin{enumerate}[label=(\roman*), itemsep=0pt, nosep]
				\item $\groupel_i^{k}(X_i)\cap X_i = \emptyset \quad \forall k\geq 1$.
				\item For all $g\in\group_i$ the element
					\[\phi(g) := \begin{cases}
									\groupel_i^{k}g\groupel_i^{-k}  & \text{on } \groupel_i^{k}(X_i), \forall k\geq 1\\
									id & \text{ elsewhere}
								\end{cases}\]
					is in $\group$.
			\end{enumerate}
		Furthermore, the group $\group$ is called \highlight{dissipated} if there exists a dissipator for $\group_i$ for all $i$.
	\end{definition}

	\begin{figure}[htbp]%
		\centering
		\begin{minipage}[t]{.95\textwidth}
			\centering
 \begin{tikzpicture}[every node/.append style={circle, draw, inner sep=0pt, minimum size=10pt}]
	\node[minimum size=3cm] at (0,0) (X0) {$X_i$};
	\node[circle, minimum size=2.3cm] at (3.4,0) (X1) {$\groupel_i(X_i)$};
	\node[circle, minimum size=1.8cm] at (6.1,0) (X2) {$\groupel_i^{2}(X_i)$};
	\node[circle, minimum size=1.4cm] at (8.2,0) (X3) {$\scriptstyle \groupel_i^{3}(X_i)$};
	\node[circle, minimum size=1.1cm] at (9.9,0) (X4) {$\scriptscriptstyle \groupel_i^{4}(X_i)$};
	\node[circle, minimum size=0.8cm] at (11.35,0) (X5) {}; 
	\node[draw=none] at (12.4,0) (dots) {$\scriptstyle\dots$};
	
	\draw[->, shorten <= 2pt, shorten >= 2pt, thick] (X0) -- (X1) node[draw=none,midway,below] {$\groupel_i$}; 
	\draw[->, shorten <= 2pt, shorten >= 2pt, thick] (X1) -- (X2) node[draw=none,midway,below] {$\groupel_i$};
	\draw[->, shorten <= 2pt, shorten >= 2pt, thick] (X2) -- (X3) node[draw=none,midway,below] {$\scriptstyle\groupel_i$};
	\draw[->, shorten <= 2pt, shorten >= 2pt, thick] (X3) -- (X4) node[draw=none,midway,below] {$\scriptstyle\groupel_i$};
	\draw[->, shorten <= 2pt, shorten >= 2pt, thick] (X4) -- (X5) node[draw=none,midway,below] {$\scriptscriptstyle\groupel_i$};
	\draw[->, shorten <= 2pt, shorten >= 2pt, thick] (X5) -- (dots);
	
	\node[ellipse,
		draw,
		minimum width = 14.5cm, 
		minimum height = 5cm,
		anchor = west, label={[xshift=0.0cm, yshift=-4.5cm]{$X_j$}}] (e) at (-2,0) {};		
\end{tikzpicture}
		\end{minipage}
		\begin{minipage}[t]{.95\textwidth}
			\centering
 \begin{tikzpicture}[every node/.append style={circle, draw, inner sep=0pt, minimum size=10pt}]
	\node[minimum size=3cm] at (0,0) (X0) {$X_i$};
	\node[circle, minimum size=2.3cm] at (3.4,0) (X1) {$\groupel_i(X_i)$};
	\node[circle, minimum size=1.8cm] at (6.1,0) (X2) {$\groupel_i^{2}(X_i)$};
	\node[circle, minimum size=1.4cm] at (8.2,0) (X3) {$\scriptstyle \groupel_i^{3}(X_i)$};
	\node[circle, minimum size=1.1cm] at (9.9,0) (X4) {$\scriptscriptstyle \groupel_i^{4}(X_i)$};
	\node[circle, minimum size=0.8cm] at (11.35,0) (X5) {};
	\node[draw=none] at (12.4,0) (dots) {$\scriptstyle\dots$};
	
	\draw[thick, ->] (3.4,-0.85) arc (-90:260:0.85) node[draw=none,midway,below] {$g$};
	\draw[thick, ->] (6.1,-0.7) arc (-90:260:0.7) node[draw=none,midway,below,yshift=1pt] {$\scriptstyle g$};
	\draw[->] (8.2,-0.55) arc (-90:260:0.55) node[draw=none,midway,below, yshift=2pt] {$\scriptscriptstyle g$};
	\draw[->] (9.9,-0.45) arc (-90:260:0.45) node[draw=none,midway,below, yshift=2pt] {$\scriptscriptstyle g$};
	\draw[->] (11.35,-0.3) arc (-90:260:0.3) node[draw=none,midway,below, yshift=2pt] {$\scriptscriptstyle g$};

	\draw[->, shorten <= 2pt, shorten >= 2pt, thick] (X0) -- (X1) node[draw=none,midway,below] {$\groupel_i$};
	\draw[->, shorten <= 2pt, shorten >= 2pt, thick] (X1) -- (X2) node[draw=none,midway,below] {$\groupel_i$};
	\draw[->, shorten <= 2pt, shorten >= 2pt, thick] (X2) -- (X3) node[draw=none,midway,below] {$\scriptstyle\groupel_i$};
	\draw[->, shorten <= 2pt, shorten >= 2pt, thick] (X3) -- (X4) node[draw=none,midway,below] {$\scriptstyle\groupel_i$};
	\draw[->, shorten <= 2pt, shorten >= 2pt, thick] (X4) -- (X5) node[draw=none,midway,below] {$\scriptscriptstyle\groupel_i$};
	\draw[->, shorten <= 2pt, shorten >= 2pt, thick] (X5) -- (dots);
	
	\node[ellipse,
	draw,
	minimum width = 14.5cm, 
	minimum height = 5cm,
	anchor = west, label={[xshift=0.0cm, yshift=-4.5cm]{$X_j$}}] (e) at (-2,0) {};		
\end{tikzpicture}
		\end{minipage}
		\caption{Iterative action of the dissipator $\groupel_i$ on $X_i$ on the top;\\ as above, together with the action of $\phi(g)$ on the bottom}
		\label{fig:dissipator}%
	\end{figure}
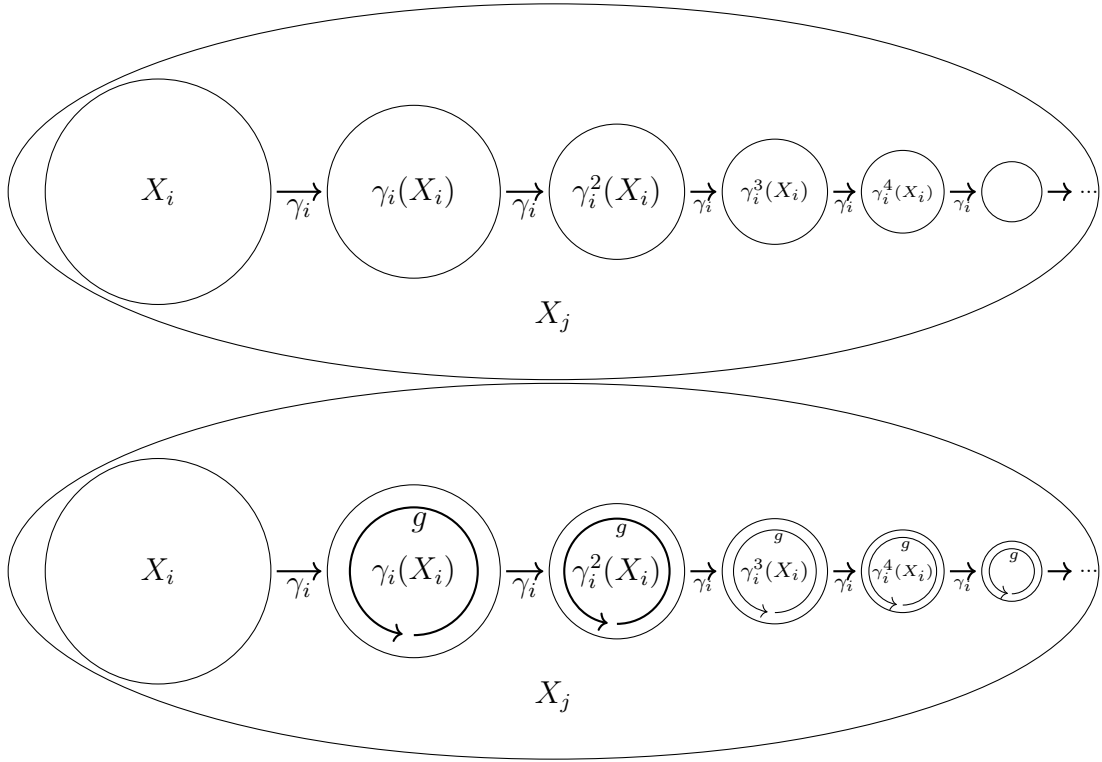
		
	Unfortunately, also the class of dissipated groups is not closed under taking subgroups. For illustrative purposes, we give some examples of dissipated groups, which have a non-dissipated subgroup.

	\begin{example}\label{exp:non-dissipated-subgroups}\hfill
		\begin{enumerate}[nosep]
			\item It is clear that a dissipator $\groupel_i$ as in \Cref{def:dissipator} must have infinite order, otherwise the first condition in \Cref{def:dissipator} cannot be satisfied. Hence, every \highlight{finite} subgroup $\subgroup$ of a dissipated group $\group$ is not dissipated.\\
			This is for example the case for $\CantorGrpStab$ (see \Cref{lem:stabilizer->dissipated->bAc}), some subgroup of all homeomorphisms of the Cantor set $\Cantor$, which is dissipated, but contains e.g. every finite cyclic subgroup.
			\item In \cite[Section 3.1.7]{berrick2002topologist} \citeauthor{berrick2002topologist} argues  that the fact that the general linear group of a cone over a ring is dissipated also yields that the group of permutations of $\NN\times\NN$, which fix all but finitely many copies of $\NN$, is dissipated. Via a bijection from $\NN\times\NN$ to $\NN$ it can then be seen that the infinite symmetric group $\sigma_\infty$, which is the group of permutations of $\NN$ with finite support, is a subgroup of a dissipated group. However, since $\sigma_\infty$ is locally finite, it cannot be dissipated.
			\item Moreover, by \Cref{prop:non-abelian-free-group-is-not-bAc} no free group can be dissipated, since it is not even \bAc. Therefore, if a dissipated group $\group$ contains a \highlight{free subgroup} $H$, this subgroup cannot be dissipated.\\
			This happens for example for $\group = \Homeo_c(\RR^n)$ which was shown to be dissipated in \cite{Gordon1960, berrick2002topologist} and any non-abelian free subgroup $\subgroup$, which can easily be constructed, for example by using the Ping-Pong lemma.
		\end{enumerate}
	\end{example}

	\begin{remark}
		Observe that by \Cref{cor:abelian-and-finite-groups-are-bAc} the non-dissipated subgroup considered in the first item in \Cref{exp:non-dissipated-subgroups} is itself \bAc, whereas the non-dissipated subgroup constructed in the third example is not \bAc (\Cref{prop:non-abelian-free-group-is-not-bAc}). Hence, these are examples of non-dissipated groups of different nature.
	\end{remark}

	In order to show that certain groups are \bAc, we are interested in a dynamical version of the following algebraic property, introduced by \citeauthor{Berrick1989UniversalGB} in \cite{Berrick1989UniversalGB}:

	\begin{definition}[binate group]\label{def:binate}
		A group $\group$ is \highlight{binate} if for every finitely generated subgroup $\subgroup\leq\group$ there exists a homomorphism $\phi:\subgroup\to\group$ and an element $\groupel\in\group$ such that
		\[ h = [\groupel, \phi(h)] = \groupel^{-1}\phi(h)^{-1}\groupel\phi(h) \quad \forall h \in H. \]
	\end{definition}

	\begin{lemma}[{\cite[Section 3.1.6]{berrick2002topologist}}]\label{lem:dissipated->binate}
		Every dissipated group is binate.
	\end{lemma}

	\begin{proof}
		Let $\group$ be a dissipated group and $(X_i,\group_i)_{i\in I}$ be as defined in \Cref{def:dissipator}. For every finitely generated $\subgroup\leq\group$ there is an $i_0\in I$ such that $\subgroup\leq\group_{i_0}$. This subgroup $\group_ {i_0}$ has a dissipator $\groupel_{i_0}$ in $\group$.
		\begin{claim*}
			$\groupel := \groupel_{i_0}$ together with the homomorphism $\psi:\subgroup\to\group$, $g\mapsto (\phi(g))^{-1}$, given in \Cref{def:dissipator}, satisfy
			\[ h = [\groupel, \psi(h)] = \groupel^{-1}\psi(h)^{-1}\groupel\psi(h) \]
		\end{claim*}
		\begin{claim-proof}
			Indeed, for all $h\in \subgroup$ on $\groupel_{i_0}^{k}(X_{i_0}), \forall k\geq 1$, we have
			\begin{align*}
				[\groupel, \psi(h)] &= \groupel^{-1}\psi(h)^{-1}\groupel\psi(h) = \groupel^{-1}\psi(h)^{-1}\groupel_{i_0}\groupel_{i_0}^{k}h^{-1}\groupel_{i_0}^{-k} = \\
				&=\groupel_{i_0}^{-1}\groupel_{i_0}^{k+1}h\groupel_{i_0}^{-(k+1)}\groupel_{i_0}^{k+1}h^{-1}\groupel_{i_0}^{-k} = \groupel_{i_0}^{k}hh^{-1}\groupel_{i_0}^{-k} = id
			\end{align*}
			where $\psi(h)^{-1}$ is replaced by $\groupel_{i_0}^{k+1}h\groupel_{i_0}^{-(k+1)}$, since $\groupel_{i_0}^{k+1}h^{-1}\groupel_{i_0}^{-k}$ sends $\groupel_{i_0}^{k}(X_{i_0})$ to $\groupel_{i_0}^{k+1}(X_{i_0})$ and on $X_{i_0}$ we have
			\[ [\groupel, \psi(h)] = \groupel^{-1}\psi(h)^{-1}\groupel\psi(h) = \groupel^{-1}\psi(h)^{-1}\groupel_{i_0} = \groupel_{i_0}^{-1}\groupel_{i_0}h\groupel_{i_0}^{-1}\groupel_{i_0} = h. \]
		\end{claim-proof}		
		Therefore, $\group$ is binate.
	\end{proof}

	However, although being dissipated and being binate seems to be equivalent, the reverse statement of the \namecref{lem:dissipated->binate} above is wrong in general. There are binate groups which are not dissipated. This happens for instance for binate groups which are torsion, since by definition every dissipator has infinite order. In \cite[Section 3.1.3]{berrick2002topologist} it is shown that Phillip Hall's countable universal locally finite group is binate, which is, since it is locally finite, in particular also torsion.\\
	
	Nevertheless, binate groups are of special interest for us. Binate groups give a large class of \bAc groups, as the concluding \nameCref{prop:binate->bAc} of this subsection shows.
	
	\begin{proposition}[\cite{fournierfacio2021binate}]\label{prop:binate->bAc}
		Every binate group is \bAc.
	\end{proposition}

\subsection{Generic Relations and Semi-simplicial Sets}

	\begin{definition}[generic relation]\label{def:generic_rel}
		Let $X$ be a set. We call a binary relation $\generic$ generic if, for every given finite set $Y\subseteq X$, there is an element $x\in X$ such that $y~\generic~x$ for all $y\in Y$.\\
		Moreover, any generic relation on $X$ gives rise to a semi-simplicial set $X^{\generic}_\bullet$ in the following way: we define $X^{\generic}_n$ to be the set of all $(n+1)$-tuples $(x_0,\dots,x_n)\in X^{n+1}$ for which $x_i~\generic~x_j$ holds for all $0\leq j<i\leq n$. The face maps $\diff[n]:X^{\generic}_n\to X^{\generic}_{n-1}$ are the usual simplex face maps, i.e.
		\[	\diff[n](x_0,\dots,x_n) = \sum_{i=0}^{n}(-1)^{i}(x_0,\dots,\widehat{x}_i,\dots,x_n) \]
		where $\widehat{x}_i$ means that $x_i$ is omitted.\\
		Note that if $X\neq\emptyset$ and $\generic$ is generic then $X^{\generic}_n\neq\emptyset$ for all $n\geq0$.
	\end{definition}

	In the given setting generic relations are especially useful due to the following \namecref{prop:generic-rel->LES} which states that the cohomology of the cochain complex associated to the corresponding semi-simplical set vanishes.
	
	\begin{proposition}\label{prop:generic-rel->LES}
		Let $\generic$ be a generic relation on a set $X$. Denote by $X^{\generic}_\bullet$ the associated semi-simplicial set. Then $X^{\generic}_\bullet$ is \bAc, i.e. the cohomology of the corresponding cochain complex $\ChainComplex[0]{\linf{X^{\generic}_1}, \linf{X^{\generic}_2}, \linf{X^{\generic}_3}}[][\delta^]$ vanishes in all degrees except $0$.
	\end{proposition}
	
	We postpone the proof of this proposition to the \Cref{APX_sec:Ultrafilters-and-ultralimits} (see \Cref{APX_prop:generic-rel->LES}) because it requires knowledge about \text{ultrafilters} and \text{ultralimits}. Since these concepts are nonstandard, \Cref{APX_sec:Ultrafilters-and-ultralimits} also contains a short (but by no means exhaustive) introduction into filters, ultrafilters and ultralimits.\\
	
	Summing up our results above we can now state a corollary which lets us calculate the bounded cohomology of certain groups.
	
	\begin{corollary}\label{cor:bounded-cohomology-through-action-and-generic-relation}
		Let $\group$ be a group acting transitively on a set $X$ on which there is a generic relation $\generic$ preserved by $\group$. If for every $n\geq 1$ the stabilizer of every point in $X^{\generic}_n$ is \bAc, we can calculate the bounded cohomology of $\group$ with (trivial) $\RR$-coefficients via
		\[ \HH[][\linf{X_{\bullet}^{\generic}}^\group] \xrightarrow{\cong} \HHbR. \]
	\end{corollary}
	\begin{proof}
		Of course we want to apply \Cref{prop:bounded-cohomology-through-acyclic-resolutions} to $\group$ together with the cochain complex $(\linf{X^{\generic}_{\bullet}},\codiff)$. By \Cref{prop:generic-rel->LES} the sequence 
		\[ \LES[0]{\linf{X^{\generic}_1}, \linf{X^{\generic}_2}, \linf{X^{\generic}_3}} \]
		is exact, since we have, by assumption, a generic relation on $X$.\\
		What is left to show is that the $\RR[\group]$-module $\linf{X^{\generic}_i}$ is \bAc for all $i\geq 1$. However, if $\group$ acts transitively on $X^{\generic}_i$ for all $i\geq1$ this follows from \Cref{lem:transitive-actions->stabilizer} and the fact that all stabilizers are \bAc. More generally, if the $\group$-action has more than one orbit in $X^{\generic}_i$, this follows from \Cref{rem:gen-of-bAc-stabilizer->bAc-module} since the fact that all stabilizers are \bAc implies the technical requirement mentioned in the \namecref{rem:gen-of-bAc-stabilizer->bAc-module}. This yields
		\[ \HH[][\linf{X_{\bullet}^{\generic}}^\group] \xrightarrow{\cong} \HHbR. \]
	\end{proof}

	\newpage
	
 \section{Homeomorphisms of the Cantor Set}\label{sec:homeomorphisms-on-K}
	
	In this section we will discuss the group of homeomorphisms of the Cantor set. We will use results from the previous section to prove \Cref{INT_thm:Homeo(K)-is-bAc} (\Cref{thm:Homeo(K)-is-bAc} below) which states that the group under consideration is \bAc.\\
	The result is essentially a consequence of the vast flexibility the group of homeomorphisms of the Cantor set offers.
	
\subsection{Introduction to the Cantor Set \texorpdfstring{$\Cantor$}{K}}
	
	There are many different possibilities to describe the Cantor set. Probably the most known description is often referred to as the \highlight{\enquote{middle-third Cantor set}} which is constructed as follows:\\
	Set $C_0 := [0,1]$. Set $C_{n+1} := \tfrac{1}{3}\cdot C_n \cup (\tfrac{1}{3}\cdot C_n+\tfrac{2}{3})$ for all $n\geq 0$. Then the middle-third Cantor set can be obtained by $C := \bigcap_{n\geq 0}C_n$.\\
	
	However, this is just one specific model of the Cantor set. Although all of the models are (as we will see) homeomorphic to one another for the sake of generalization we chose to mainly work with an abstract set, which carries the characterizing properties of a Cantor set. Nevertheless, to illustrate some of the following concepts we will use the specific \highlight{dyadic model} of the Cantor set, introduced below
	 
	\begin{definition}[Cantor set]
		A Cantor set is a metric space which is compact, totally disconnected and perfect (has no isolated points).\\
		In the following $\Cantor$ will always denote a Cantor set and $\CantorMetric$ its metric.
	\end{definition}
	
	Moreover, for a subset $U\subseteq\Cantor$ we will denote the diameter of $U$ by $\diam(U) := \sup_{x,y\in U}\CantorMetric(x,y)$.
	
	\begin{theorem}\label{thm:Cantor-sets-are-homeomorphic}
		Every two Cantor sets are homeomorphic.
	\end{theorem}
	The proof of this fact can be found in many topology textbooks. (e.g. \cite[Section 12, Theorem 8]{moise2013geometric}). However, for the interested reader we also give a detailed proof in \Cref{APX_thm:Cantor-sets-are-homeomorphic}. This proof uses some topological properties of Cantor sets so we chose to move it to \Cref{APX_sec:cantor-sets-and-maps-between-them}.
	
	\begin{example}[Dyadic model of the Cantor set]\label{exp:dyadic-Cantor-model}
		A further model of the Cantor set is the following. Let $\{0,1\}$ be equipped with the discrete topology. Then the set $\prod_{n=1}^{\infty}X_n$, where $X_n = \{0,1\}$ for all $n\geq1$, equipped with the product topology is a Cantor set.\\
		Recall that if we let $p_k:\prod_{n=1}^{\infty}X_n \to X_k$ be the projection for all $k\geq1$, then a subbase for the topology on $\prod_{n=1}^{\infty}X_n$ is given by the set 
		\[ \{p_k^{-1}(U_k)  \subseteq \prod_{n=1}^{\infty}X_n~|~ k\geq1 \text{ and } U_i\subseteq X_k \text{ is open in } X_k\}. \]
		We will henceforth denote this set by $\CantorModel$ and represent its elements with sequences $\seq$, where $\seqel\in\{0,1\}$ for all $k\geq1$.\\
		Although this example is also quite standard, for sake of clarity we nevertheless show that this is a Cantor set in detail:\\
		A metric $\CantorModelMetric:\CantorModel\times\CantorModel\to\RR$ is given by
		\[ \CantorModelMetric(\seq,\seq[y]) := 2^{-N}, \]
		where $N\geq1$ is the first index for which $\seqel[][N] \not=\seqel[y][N]$. This metric is an ultrametric, that is for all $\seq,\seq[y],\seq[z]\in\CantorModel$ it holds that
		\[ \CantorModelMetric(\seq,\seq[z]) \leq \max\{\CantorModelMetric(\seq,\seq[y]),\CantorModelMetric(\seq[y],\seq[z])\}. \]
		Note that the topology induced by $\CantorModelMetric$ coincides exactly with the topology on $\CantorModel$ introduced before. It is easy to see that $\CantorModel$ is compact, since it is closed and totally bounded (see also Tychonoff's Theorem). Regarding isolated points, for every $\epsilon>0$ and every $\seq[y]\in\CantorModel$, we can find an $N\geq1$ such that $2^{-N}<\epsilon$ and choose an $\seqel\in\CantorModel$, so that $\seqel = \seqel[y]$ for all $k\leq N$. Then
		 \[ \CantorModelMetric(\seq,\seq[y]) \leq 2^{-N} < \epsilon. \]
		Hence, there are no isolated points, so $\CantorModel$ is also perfect.\\
		Furthermore, we have to argue why $\CantorModel$ is totally disconnected. Assume that $U\subseteq\CantorModel$ is a set containing two distinct points $\seq[y]\not=\seq[z]$ in $U$. Then there is an index $N\geq1$ so that, say, $0=\seqel[y][N]\not=\seqel[z][N]=1$. However, the sets
		\[U_0 := \{\seq\in\CantorModel~|~\seqel[][N] = 0 \} \text{ and } U_1 := \{\seq\in\CantorModel~|~\seqel[][N] = 1 \} \]
		are both open in $\CantorModel$. Moreover, $\CantorModel = U_0\cup U_1$ and $\seq[y]\in U_0$, $\seq[z]\in U_1$. Hence, $U = (U\cap U_0)\cup (U\cap U_1)$ is a disjoint union of non-empty, open subsets, so $U$ is not connected. Therefore $\CantorModel$ is totally disconnected.\\
		This shows that $\CantorModel$ is indeed a Cantor set.\\
		
		Finally, the lexicographic order gives a natural linear order $\lorder$ on $\CantorModel$, i.e. $\seq\lorder\seq[y]$ if and only if $\seqel[][N] = 0$ and $\seqel[y][N] = 1$, where again $N$ is the first index for which $\seqel[][N] \not=\seqel[y][N]$.
	\end{example}
	
	This example already hints to one of the most interesting properties of Cantor sets, namely \enquote{self-similarity} properties. The next few results will delve further into this. 
	
\subsection{Properties of the Cantor Set \texorpdfstring{$\Cantor$}{K}}
	
	\begin{theorem}[{\cite[Section 12, Theorem 5]{moise2013geometric}}]\label{thm:clopen-subset->Cantor-set}
		Let $\Cantor$ be a Cantor set and $U$ a subset of $\Cantor$. If $U$ is open and closed, then $U$ is a Cantor set.\\
		In particular, there is a homeomorphism between $\Cantor$ and $U$.
	\end{theorem}
	\begin{proof}
		Clearly $U$ is also a metric space. Since $U$ is closed, $U$ is also compact. Moreover, $U$ is open, therefore $U$ cannot have isolated points. Finally, $U$ is also totally disconnected, because, every connected component of $U$ is also contained in a connected component of $\Cantor$, all of which are single points.\\
		Whence $U$ satisfies all defining properties of a Cantor set.
	\end{proof}

	\begin{lemma}[{\cite[Section 12, Theorem 3]{moise2013geometric}}]\label{lem:closed-subsets-of-compact-metric-space}
		Let $X$ be a compact metric space and $A$, $B$ be closed subsets of $X$. Then one of the following holds:
		\begin{itemize}
			\item There is a connected component of $X$ intersecting both $A$ and $B$.
			\item We can decompose $X$ into a union of two disjoint closed (thus open) subsets containing $A$ and $B$, respectively.
		\end{itemize}
	\end{lemma}
	
	\begin{corollary}[self-similarity of the Cantor set]\label{cor:self-similarity}		
		For every point $x\in\Cantor$ and every neighborhood $U\ni x$ there is an even smaller neighborhood $O\subsetneq U$ of $x$ which is homeomorphic to $\Cantor$.
	\end{corollary}
	\begin{proof}
		This follows easily from the two results above.\\
		Define $\delta := \inf\{\CantorMetric(x,y)~|~y\in\Cantor\setminus U\} > 0$, which is bigger than $0$, since $\Cantor\setminus U$ is closed. Now set $A := \overline{B(x,\delta/2)} = \{y\in\Cantor~|~\CantorMetric(x,y) \leq \delta/2\}$ the closed $\delta/2$-ball around $x$ and $B := \Cantor\setminus U$. Clearly, $A\subseteq U$, hence $A$ and $B$ are closed and disjoint. Since the only connected components of $X$ are singletons, we can, by the previous lemma, decompose $X$ into two disjoint closed subsets $C_A$ and $C_B$ such that $A\subseteq C_A$ and $B\subseteq C_B$.\\
		Finally, $\Cantor\setminus U = B \subseteq C_B$, hence $C_A\subseteq U$ and by \Cref{thm:clopen-subset->Cantor-set} $O := C_A$ is a neighborhood of $x$ homeomorphic to $\Cantor$.
	\end{proof}
	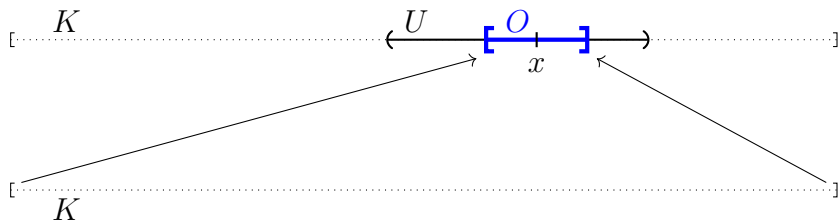
\begin{figure}[htbp]
		\begin{center}
 \begin{tikzpicture}
	\draw[{[-]}, dotted] (0,0) -- +(11,0); 
	\node at (0.75, -0.25) (Cantor1) {$\Cantor$};
	\draw[{[-]}, dotted] (0,2) -- +(11,0);
	\node at (0.75, 2.25)  (Cantor2) {$\Cantor$};
	\draw[{(-)}, thick] (5,2) -- (8.5,2);
	\node at (5.4, 2.25)  (U) {$U$};
	\draw[{[-]}, ultra thick, blue] (6.3,2) -- (7.7,2);
	\node[blue] at (6.75, 2.25)  (O) {$O$};
	\draw[thick] (7,2.1) -- (7,1.9) node[below] {$x$};
	\draw[->] (0.15,0.1) -- (6.2,1.75);
	\draw[->] (10.85,0.1) -- (7.8,1.75);
\end{tikzpicture}
			\caption{Schematic illustration of the homeomorphism from $\Cantor$ to $O$.}
		\end{center}
	\end{figure}

	\begin{corollary}\label{cor:homeomorphism_with_partial_id}
		Let $U,V \subset \Cantor$ be subsets of the Cantor set which are open and closed and such that their intersection has non-empty interior. Then there is a homeomorphism $\phi$ on $\Cantor$ which maps $U$ onto $V$ and is the identity on a non-empty open subset of $U\cap V$.  
	\end{corollary}
	
	\begin{remark}
		Note that the property of the previous \namecref{cor:homeomorphism_with_partial_id} is rather restrictive. For most homeomorphism groups which do not offer that much flexibilty this fails completely. For example, consider homeomorphisms of $[0,1]$ and the sets $U = [0,3/4]$, $V=[1/4,1]$.\\
		Although there exist homeomorphisms on $[0,1]$ which send $U$ to $V$ (e.g. $x\to(1-x)$), there cannot be one which equals the identity on a subset of $U\cap V$ with non-empty interior, since any such map would have to be orientation preserving. However, any orientation preserving homeomorphism on $[0,1]$ has to preserve the endpoints, hence $U$ (containing $0$) cannot be mapped to $V$ (not containing $0$.)
	\end{remark}
	
	\begin{proof}
		Let $O'\subset U\cap V$ be an open subset of the intersection. Pick $x\in O'$ and let $O\subset O'$ be an open neighborhood of $x$ which is homeomorphic to $\Cantor$ which exists by \Cref{cor:self-similarity}. Then we have the following properties:
		\begin{enumerate}[label=(\roman*), itemsep=0pt, nosep]
			\item $U$ and $V$ are closed and open in $\Cantor$,
			\item $O$ is closed and open in $\Cantor$, $U$ and $V$, and
			\item $U\setminus O$ and $V\setminus O$ are closed and open in $\Cantor$, since $O$ is closed and open.
		\end{enumerate}
		By \Cref{thm:clopen-subset->Cantor-set} $O$, $U\setminus O$, $V\setminus O$, $\Cantor\setminus U$ and $\Cantor\setminus V$ are all homeomorphic to the Cantor set. Hence, we can choose an element $g\in\Homeo(\Cantor)$ which satisfies:
		\begin{align*}
		g|_O &= id|_O\\
		g(U\setminus O) &= g(V\setminus O)\\
		g(K\setminus U) &= g(K\setminus V)
		\end{align*}
		This proves the \namecref{cor:homeomorphism_with_partial_id}.
	\end{proof}
	
	\begin{remark}
		Unfortunately, we cannot get rid of the assumption for the subsets $U$ and $V$ to be both closed and open in $\Cantor$. Indeed, there are at least two types of closed subsets of $\Cantor$ which are homeomorphic to $\Cantor$. One of those types consists of all the sets which are closed and open in $\Cantor$, the other ones are of course also closed, but need not be open in $\Cantor$. If, say $U$ is open and closed and $V$ is non-open but closed, then there cannot exist a homeomorphism $\Cantor\to\Cantor$ which sends $U$ to $V$.
	\end{remark}
	
	\begin{example}\label{exp:non-open-homeomorphic-subsets}
		To illustrate the situation where we have a subset of $\Cantor$ which is homeomorphic to $\Cantor$ but not open in $\Cantor$, we consider the dyadic model $\CantorModel$ of the Cantor set and give a concrete example:\\
		Let $p := \seq[p]\in\CantorModel$, where $\seqel[p] = k+1 \mod 2$. So $p = (0,1,0,1,0,1,0,\dots)$. Further define $h:\Cantor\to\Cantor$ as
		\[ h(\seq) := p + (0,\seqel[][1],\seqel[][2],\seqel[][3],\seqel[][4],\dots). \]
		To be more precise, define 
		\[ \operatorname{carry}_k(\seq[][l]) := \exists n\geq \lfloor\tfrac{k-1}{2}\rfloor (\forall 0<l\leq n:\seqel[][2l] = 1~\wedge~\seqel[][2n+1] = 1) \]
		and 
		\[ \operatorname{no-carry}_k(\seq[][l]) :=\neg(\operatorname{carry}_k(\seq[][l])) = \forall n\geq\lfloor\tfrac{k-1}{2}\rfloor(\exists 0<l\leq n:\seqel[][2l] =0~\vee~\seqel[][2n+1] = 0).\]
		If now $h(\seq) = \seq[y]$, then  
		\begin{align*}
		\seqel[y][1] = 1&\text{ if and only if }\operatorname{carry}_0(\seq[][l]),\\			
		\seqel[y][2k] = 1 &\text{ if and only if } 
		\begin{cases} 
		\seqel[][2k-1] = 1 \text{ and } \operatorname{carry}_{2k}((\seq[][l]))\\
		\seqel[][2k-1] = 0 \text{ and } \operatorname{no-carry}_{2k}((\seq[][l]))\\
		\end{cases} \text{ for all } k\geq 1,\\
		\seqel[y][2k+1] = 1 &\text{ if and only if } 
		\begin{cases} 
		\seqel[][2k] = 1 \text{ and } \operatorname{no-carry}_{2k}((\seq[][l]))\\
		\seqel[][2k] = 0 \text{ and } \operatorname{carry}_{2k}((\seq[][l]))\\
		\end{cases} \text{ for all } k\geq 1.\\
		\end{align*}
		First of all, note that $h:\CantorModel\to\CantorModel$ is injective. Moreover, $h$ is continuous, since $\CantorModelMetric(h(\seq),h(\seq[y])) = \tfrac{1}{2}\cdot\CantorModelMetric(\seq,\seq)$. Thus, $h$ is automatically a homeomorphism on its image $h(\CantorModel)$, since $\CantorModel$ is compact and Hausdorff. (e.g. see \cite[Theorem 7.8]{bredon_topology_1993}).\\
		To conclude the argument, we need to show that $h(\CantorModel)$ is indeed not open in $\CantorModel$.\\
		Clearly, $p\in h(\CantorModel)$. Assume there is an open set $U\in\CantorModel$, so that $p\in U\subseteq h(\CantorModel)$. Recall the subbase of open sets introduced in $\CantorModel$ (\Cref{exp:dyadic-Cantor-model}). Using the same notation we can find  a finite set of indices $I\subseteq\NN$ and open sets $U_i\in\{0,1\}$ for all $i\in I$, such that $p\in \bigcap_{i\in I} p_i^{-1}(U_i)\subseteq U$. However, if we now choose a number $M\geq1$ bigger than all the finitely many indices $i\in I$, then the element $\seq[z]\in\CantorModel$ with $\seqel[z][k] = k+1 \mod 2$ if $k\leq M$ and $\seqel[z][k] = 0$ if $k\geq M$ is inside $\bigcap_{i\in I}p_i^{-1}(U_i)\subseteq U$, but not in $h(\CantorModel)$. This is because $p$ is \enquote{the left most point} of $h(\CantorModel)$, i.e. $p\lorder h(\seq)$ for all $\seq\in\CantorModel$, but we have $\seq[z]\lorder p$ and $\seq[z] \neq p$. A contradiction.\\
		Hence, in particular $U\not\subseteq h(\Cantor)$, so $h(\Cantor)$ cannot be open.
	\end{example}
	
	\begin{remark}\label{rem:homeo-without-interior}
		In fact, there are even homeomorphisms $h:\Cantor\to\Cantor$ such that the interior of the image is empty, so the image can definitely not be open in $\Cantor$.\\
		Considering again the dyadic model $\CantorModel$ we could define such a map $h:\CantorModel\to\CantorModel$ as:
		\[ h(\seqel) := (0,\seqel[][1],0,\seqel[][2],0,\seqel[][3],0,\seqel[][4],0,\seqel[][5],\dots). \]
		Now it is not difficult to show that no point inside $h(\CantorModel)$ has a neighborhood which is fully contained in the image of $h$.\\
		The argument is very similar to the argument given in the previous example. We leave the details to the interested reader.
	\end{remark}

	Now we can state the first result about the group of homeomorphisms of the Cantor set. We define
	\[ \Homeo(\Cantor) := \{f:\Cantor\to\Cantor~|~f \text{ is a homeomorphism}\}. \]
	
	\begin{proposition}\label{prop:Homeo(Cantor)-acts-transitively}
		$\Homeo(\Cantor)$ acts transitively on $\Cantor$.
	\end{proposition}
	
	To prove this result just in terms of the abstract properties of a Cantor set some topological niceties are needed. Thus, we chose to put this proof in \Cref{APX_sec:cantor-sets-and-maps-between-them} (see \Cref{APX_cor:Homeo(Cantor)-acts-transitively}) and work here only through the specific dyadic model $\CantorModel$. Of course, having the transitivity of the homeomorphisms on a single model of the Cantor set is already enough to show \Cref{prop:Homeo(Cantor)-acts-transitively}, assuming the fact that all models of the Cantor set are homeomorphic to one another.
	
	\begin{example} [$\Homeo(\CantorModel)$ acts transitively on $\CantorModel$]
		Let $\seq,\seq[y]\in\CantorModel$ be two elements of the dyadic Cantor set model. We want to find a homeomorphism $g\in\Homeo(\CantorModel)$, such that $g(\seq) = \seq[y]$.\\
		One such homeomorphism could be defined as follows:
		\[ \seqel[{(g(\seq[z]))}][n] := \seqel[][n] + \seqel[y][n] + \seqel[z][n] \mod 2 \quad \forall n\geq 1.\]
		First, observe that $\seqel[{(g(\seq))}][n] = \seqel[][n] + \seqel[y][n] + \seqel[][n] \mod 2 = \seqel[y][n]$ for all $n\geq1$, i.e. $g(\seq) = \seq[y]$. What is left to argue is why $g$ is actually a homeomorphism. We have $g\circ g = id$, therefore $g$ is bijective and we only have to verify that $g$ is continuous (e.g. see \cite[Theorem 7.8]{bredon_topology_1993}). Since $g$ is the inverse of $g$, saying that $g$ is continuous is equivalent to saying that $g$ is an open mapping, i.e. $g$ sends open sets to open sets. Let $U\subseteq\CantorModel$ be an open set and $\seq[z]\in U$. By the definition of the subbase of the topology on $\CantorModel$ in \Cref{exp:dyadic-Cantor-model}, there is an index $N\geq1$ such that $U' := \{ \seq[z']\in\CantorModel~|~\forall n\leq N:~ \seqel[z'][n] = \seqel[z][n]\} \subseteq U$. However, then $g(\seq[z]) \in g(U') = \{ \seqel[z']\in\CantorModel~|~\forall n\leq N:~ \seqel[z'][n] =  \seqel[][n] + \seqel[y][n] +\seqel[z][n]\} \subseteq g(U)$ is, by definition of the product topology on $\CantorModel$, an open neighborhood around $g(\seq[z])$. Hence, $g$ is an open mapping and therefore $g\in\Homeo(\CantorModel)$.\\
		Another way to see that $g$ is continuous and hence a homeomorphisms is by using the universal property of products: The map $g$ is defined coordinate-wise and it is easy to see that in every coordinate the preimages of all open sets $\emptyset$, $\{0\}$, $\{1\}$ and $\{0,1\}$ are open in $\CantorModel$. Hence $g$ is continuous in every coordinate, and thus $g:\CantorModel\to\CantorModel$ is continuous.
	\end{example}
	
	\begin{corollary}\label{cor:Homeo(Cantor)-acts-highly-transitively}
		$\Homeo(\Cantor)$ acts highly transitively on tuples of distinct points in $\Cantor$.
	\end{corollary}
	\begin{proof}
		Let $(x_1,\dots,x_k)$ and $(\bar{x}_1,\dots,\bar{x}_k)$ be tuples of pairwise distinct points in $\Cantor$. For each $1\leq i\leq k$  we can choose closed and open neighborhoods $U_i$, $\bar{U}_i$ of $x_i$ and $\bar{x}_i$, respectively, such that all the $U_i$ and all the $\bar{U}_i$ are pairwise disjoint. That is $U_i\cap U_j=\emptyset$ and $\bar{U}_i\cap \bar{U}_j=\emptyset$ for all $1\leq i,j\leq k$ with $i\neq j$. Moreover, define $U := \Cantor\setminus\bigcup_{i=1}^k U_i$, $\bar{U} := \Cantor\setminus\bigcup_{i=1}^k \bar{U}_i$ and observe that both of these sets are homeomorphic to $\Cantor$, since the complements are closed and open subsets of $\Cantor$.\\
		Now we construct a map $g$ satisfying the following properties:
		\begin{align*}
			g(x_i) &= \bar{x}_i &&\text{for all } 1\leq i\leq k\\
			g(U_i) &= \bar{U}_i &&\text{for all } 1\leq i\leq k\\
			g(U) &= \bar{U}
		\end{align*}
		First, for all $1\leq i\leq k$, $U_i$ and $\bar{U}_i$ are both homeomorphic to $\Cantor$. Therefore the existence of a homeomorphism $g|_{U_i}$ on $U_i$ satisfying the first two properties is given by \Cref{thm:Cantor-sets-are-homeomorphic} and \Cref{prop:Homeo(Cantor)-acts-transitively}. Second, we can choose $g|_U$ to be the composition of any homeomorphisms between $U$ and $\Cantor$, and $\Cantor$ and $\bar{U}$, whose existence is assured by \Cref{thm:Cantor-sets-are-homeomorphic}.\\
		Finally, observe that the resulting map $g$ is indeed a homeomorphism on all of $\Cantor$, since all the sets $U_i$ and $U$ are open. Hence $g$ is an element of $\Homeo(\Cantor)$ and sends the tuple $(x_1,\dots,x_k)$ to $(\bar{x}_1,\dots,\bar{x}_k)$. This proves the \namecref{cor:Homeo(Cantor)-acts-highly-transitively}.
	\end{proof}

	In order to simplify the notation, we will give the following (non-standard) definition
	
	\begin{definition}[embedding]\label{def:embedding}
		Let $X$, $Y$ be topological spaces. We call a function $f:X\to Y$ an \highlight{embedding}, if $f$ is a homeomorphism onto its image $f(X)$ and moreover $f(X)$ is open in $Y$.
	\end{definition}
	
	\begin{remark}
		Usually an embedding between topological spaces $X,Y$ is just a map $f:X\to Y$, which is a homeomorphism onto its image $f(X)$.\\
		A more common phrase for maps as in \Cref{def:embedding} would be \highlight{open embedding}. However, since we only consider this type of embeddings, we chose to drop the word \enquote{open}.
	\end{remark}

	In particular, any embedding $f:\Cantor\to\Cantor$ will lead to a closed and open set $f(\Cantor)$ inside $\Cantor$.
	
	\begin{definition}[shrinking]\label{def:shrinking-expansion}
		We call an embedding $\tau:\Cantor\to\Cantor$  \highlight{shrinking} if\linebreak $\lim_{n\to\infty}\diam(\tau^n(K)) = 0$.
	\end{definition}

	\begin{remark}
		Note that since $\tau$ is an embedding, $\tau(K)$ is closed in $\Cantor$. Hence, by induction, also $\tau^n(\Cantor)$ is closed in $\Cantor$ for all $n\geq1$. Moreover, $\tau^{n-1}(K)\supseteq\tau^n(\Cantor)$ holds for all $n\geq1$. Therefore, by Cantor's Intersection Theorem $\bigcap_{n\geq 0}\tau^n(K)\neq\emptyset$ and since the diameter of the sets converges to $0$, this intersection consists of exactly one point.
	\end{remark}
	
	The following lemma ensures that there are plenty of shrinking maps with image in a given proper non-empty closed and open subset of $\Cantor$:

	\begin{lemma}\label{lem:shrinking-map}
		Let $U\subsetneq\Cantor$ be a non-empty closed and open subset. Then there is a shrinking $\tau:\Cantor\to\Cantor$, such that $\tau(\Cantor) = U$.
	\end{lemma}
	\begin{remark}
		Note that we are not requiring $\diam(U)<\diam(\Cantor)$ we only require the set $U$ to be a proper closed and open subet of $\Cantor$. It could well be that both sets have the same diameter. Further, we do not impose any requirements on the metric properties of the image of the map $\tau$. Hence this result is of a different nature than Banach's Fixed Point Theorem.
	\end{remark}
	\begin{remark}
		To prove this \namecref{lem:shrinking-map} one may want to use a similar approach as in the proof of \Cref{thm:Cantor-sets-are-homeomorphic} (see \Cref{APX_sec:cantor-sets-and-maps-between-them}) and just keep track on which elements in the partition of the Cantor set contain the image of $\Cantor$, $\tau(\Cantor),\tau^2(\Cantor),\dots$ under $\tau$. However, since in this situation the sets $\Cantor$ and $\Cantor'$ are not separated, but $\Cantor'$ is in fact a subset of $\Cantor$, it is very hard to do this.
	\end{remark}
	\begin{proof}
		We cannot give an explicit description of $\tau$, but we will define it piece for piece on exhausting subsets of $\Cantor$. To do this we use the result of \Cref{thm:Cantor-sets-are-homeomorphic}.\\
		Set $U_1 := U$ and define $V_1 := \Cantor\setminus U_1$. Note that $V_1\neq\emptyset$, which is crucial, as otherwise the \namecref{lem:shrinking-map} clearly cannot hold. By \Cref{APX_cor:exact-cover-of-Cantor-sets} we can decompose $U_1$ into two non-empty disjoint closed subsets $U_1 = U_2 \cup V_2$ such that $\diam(U_2) < 2^{-2}$. Since both sets are closed and open in $U_1 = U$ (hence also in $\Cantor$), they are also Cantor sets. Thus, we can find some homeomorphisms $\tau_1:V_1\to V_2$ by \Cref{APX_thm:Cantor-sets-are-homeomorphic}. The map $\tau_1$ is well-defined, since $V_1$ is not empty.\\
		Suppose now that we have already constructed non-empty disjoint closed and open subsets $U_n,V_n,\dots V_2\subseteq U$, such that $U = U_n\cup \bigcup_{k=2}^{n} V_k$ and $\diam(U_n) < 2^{-n}$, as well as homeomorphisms $\tau_k:V_k\to V_{k+1}$ for all $1\leq k < {n-1}$. We now want to define $U_{n+1}$, $V_{n+1}$ and $\tau_n:V_n\to V_{n+1}$ as follows:\\
		By \Cref{APX_cor:exact-cover-of-Cantor-sets} we can again decompose $U_n$ into two non-empty disjoint closed subsets $U_n = U_{n+1} \cup V_{n+1}$ such that $\diam(U_{n+1}) < 2^{-{n+1}}$, and we pick an arbitrary homeomorphism $\tau_n:V_n\to V_{n+1}$.\\
		
		Since the diameter of the sets $U_n$ converges to $0$ if $n$ goes to infinity and $U_1\supseteq U_2\supseteq U_3\supseteq\dots$, we have $\bigcap_{n\geq 1}U_n =\{p\}$ for some $p$ in $\Cantor$. Moreover, we also have $\Cantor = \{p\}\cup\bigcup_{n\geq 1}V_n$. We can now define the map $\tau:\Cantor\to\Cantor\setminus(A\cup\{z\})$ as 
		\begin{align*}
			\tau(p) &:= p\\
			\tau(x) &:= \tau_n(x) \text{ if } x\in V_n.
		\end{align*}
		What is left to show is that $\tau$ is a homeomoprhism and an embedding. Once we showed that $\tau$ is a homeomorphism, the fact that $\tau$ is an embedding is clear, since $\tau(\Cantor) = \tau(\{p\}\cup\bigcup_{n\geq 1}V_n) = \{p\}\cup\bigcup_{n\geq 2}V_n = U$ which is open in $\Cantor$. Furthermore, it is also not difficult to show that $\tau$ is continuous, hence a homeomorphism. For a point $x\neq p$ continuity is obvious, since then $x$ is in the closed and open set $V_N$ for some $N\geq1$ and $\tau$ restricts to the homeomorphism $\tau_N$ on $V_N$. Moreover, for every neighborhood $O$ of $p$ there is some index $N\geq1$ such that $U_N\subseteq O$, hence
		\[ \tau^{-1}(O)\supseteq\tau^{-1}(U_N) = \tau^{-1}(\{p\}\cup\bigcup_{n\geq N+1}V_n) = \{p\}\cup\bigcup_{n\geq N}V_n = U_{N-1}, \] 
		which is an open neighborhood of $p$, so $\tau$ is also continuous at $p$.\\
		Finally, we have $\tau(\Cantor) = U = U_1$ and $\tau(U_n) = U_{n+1}$ for all $n\geq1$. Hence, $\bigcap_{n\geq1}\tau^n(\Cantor) = \bigcap_{n\geq1}U_n = \{p\}$, so $\tau$ is indeed a shrinking.\\
		This concludes the proof of the existence of shrinking maps.
	\end{proof}		
	
	\begin{example}[shrinking map on $\CantorModel$]\label{exp:shrinking-map-in-dyadic-Cantor-Model}
		Although it is quite abstract to describe a shrinking map in the general topological setting, once we work in the dyadic model of the Cantor set, for closed and open sets $U$ which are \enquote{nice enough}, a shrinking map can be depicted more easily.\\
		Consider the dyadic model $\CantorModel$ of the Cantor set and pick a point $\seq\in\CantorModel$. For an arbitrary $0<\delta<1$ and we can find a closed and open neighborhood $U$ of $\seq$, which can be described as
		\[ U = \{\seq[y]\in\CantorModel~|~\forall k\leq M:~ \seqel[y][k] = \seqel[][k]\} \]
		for some fixed integer $M\geq1$, such that $2^{-M} < \delta$.\\
		Now we can define $\tau:\CantorModel\to U$ as
		\[ \tau(\seq[y]) = (\seqel[][1],\seqel[][2],\dots,\seqel[][M],\seqel[y][1],\seqel[y][2]\dots). \]
		It is easy to see that $\tau(\CantorModel) = U$ and that $\tau$ is a homeomorphism. Hence it is also an embedding, since $U$ is open. Further 
		\[ \CantorModelMetric(\tau(\seq[y^1]),\tau(\seq[y^2]))= 2^{-M}\CantorModelMetric(\seq[y^1],\seq[y^2]) < \delta\CantorModelMetric(\seq[y^1],\seq[y^2]), \]
	 	thus $\tau$ is a shrinking.\\
	 	The unique fix point $p$ of $\tau$, is exactly the unique point inside $\bigcap_{n\geq1}\tau^n(\Cantor)$. In this case this will be $p = (x_1,\dots,x_M,x_1,\dots,x_M,\dots)$.
	\end{example}

	\begin{remark}
		Observe that this easy description of the map $\tau$ is only possible to achieve since $U$ is the intersection of finitely elements of the subbase of the topology described in \Cref{exp:dyadic-Cantor-model}. Further, note that since the map $\tau$ in the previous example is obtained by shifting the coordinates and adding a fixed finite sequence at the beginning we obtain a fixed point $p$ of $\tau$ which is described by a periodic sequence. However, this must not be the case in general.\\
		Indeed by \Cref{lem:shrinking-map}, it is possible to start with \highlight{any} proper closed and open subset $U$, and then require that a specific point $p\in U$ is the unique fixed point of $\tau$. This can be done by making sure that $p$ is always contained in $U_n$ and not in $V_n$ for all $n\geq1$.\\
		In particular, if we again look at $\CantorModel$ and pick a point $p$ which is not described by a periodic sequence, we can not define $\tau$ in the same way as we did in \Cref{exp:shrinking-map-in-dyadic-Cantor-Model}.
	\end{remark}
	
\subsection{Bounded Cohomology of Subgroups of \texorpdfstring{$\Homeo(\Cantor)$}{Homeo(K)}}
	
	In the following section we will discuss the bounded cohomology of the group of homeomorphisms of the Cantor set and deduce that it is \bAc.\\
	Throughout the section we use the following notation:\\
	$\Cantor$ will denote an arbitrary model of the Cantor set. In some examples and illustrations we will work with a concrete model, namely the product of countably many copies of $\{0,~1\}$ equipped with the product topology, and we write $\CantorModel$ in that case. Furthermore, we set
	\[ \CantorGrp := \Homeo(\Cantor) = \{g:\Cantor \to \Cantor~|~g \text{ is a homeomorphism}\}. \]
	
\subsubsection{Bounded Acyclicity of Rigid Stabilizer Subgroups}
	
	This subsection is devoted to the proofs of \Cref{INT_lem:stabilizer->dissipated->bAc} and \Cref{INT_cor:tuple-stabilizer->dissipated} (\Cref{lem:stabilizer->dissipated->bAc} and \Cref{cor:tuple-stabilizer->dissipated} below, respectively).\\
	In order to calculate the bounded cohomology of $\CantorGrp$, we would like to analyze the bounded cohomology of certain subgroups of $\CantorGrp$.
	
	\begin{definition}
		Let $z\in\Cantor$. We denote by  
		\[ \CantorGrpStab := \{g \in\CantorGrp~|~g \text{ fixes a neighborhood of } z \text{ pointwise}\} \]
		the subgroup of elements which fix a neighborhood of $z$ pointwise.
	\end{definition}

	\begin{lemma}[\Cref{INT_lem:stabilizer->dissipated->bAc}]\label{lem:stabilizer->dissipated->bAc}
		For every $z\in\Cantor$, $\CantorGrpStab$ is dissipated. In particular, $\CantorGrpStab$ is \bAc.
	\end{lemma}
	\begin{remark}
		By \Cref{lem:dissipated->binate} the fact that $\CantorGrpStab$ is dissipated implies that $\CantorGrpStab$ is binate and hence proves \Cref{INT_lem:stabilizer->dissipated->bAc}.
	\end{remark}
	\begin{remark}
		In \cite[Section 3.1.6]{berrick2002topologist} a similar statement, namely that the subgroup of $\CantorGrp$ fixing some non-empty open subset is dissipated, is given. However, the statement is mentioned without a proof.\\
		Note that this is similar to our situation. In fact, we have $\{g\in\CantorGrp~|~ g|_U = id \} \subseteq \CantorGrpStab$ for all open neighborhoods $U$ of $z$. Moreover, $\CantorGrpStab$ can be written as increasing union of sets of the form $\{g\in\CantorGrp~|~ g|_U = id \}$, where the corresponding union of the sets $U$ is equal to $\Cantor\setminus\{z\}$. Since increasing unions of dissipated groups are dissipated, this yields that $\CantorGrpStab$ is dissipated.
	\end{remark}
	\begin{proof}
		Consider the set $\Cantor$ and let $\mathcal{U}$ be the set of closed and open neighborhoods of $z$. Define $X_U := \Cantor\setminus U$ for all $U \in \mathcal{U}$. This gives a family of compact subsets of the Cantor set, such that $\Cantor\setminus\{z\}$ is the directed union of the $(X_U)_{U\in\mathcal{U}}$ equipped with the obvious inclusion. Further $\CantorGrpStab$ acts faithfully on $\Cantor\setminus\{z\}$. As in \Cref{def:boundedly_supported} let $\group_U$ be the subgroup of elements of $\CantorGrpStab$ with support in $X_U$, then it is easy to see that $\CantorGrpStab$ is the directed union of the $\group_U$, so $\CantorGrpStab$ is boundedly supported.\\
				
		Now fix $U\in \mathcal{U}$. We have to find a dissipator $\groupel_U \in\CantorGrpStab$ for $\group_U$, that is a homeomorphism on $\Cantor$, which fixes a neighborhood of $z$ pointwise and satisfies
		\begin{enumerate}[label=(\roman*), itemsep=0pt, nosep]
			\item $\groupel_U^{k}(X_U)\cap X_U = \emptyset$ for all $k\geq 1$, and
			\item for all $g\in\group_U$ the element
			\begin{align}
				\phi(g) := \begin{cases}
					\groupel_U^{k}g\groupel_U^{-k}  & \text{on } \groupel_U^{k}(X_U), \forall k\geq 1\\
					id & \text{ elsewhere}
				\end{cases}
			\end{align}
			is in $\CantorGrpStab$.
		\end{enumerate}
		Since $X_U$ is closed, we can fix two disjoint closed and open subsets $O_+$ and $O_-$, which do not intersect $X_U\cup\{z\}$ and apply \Cref{lem:shrinking-map} to both of them to obtain two shrinkings $\tau_+:\Cantor\to O_+$ and  $\tau_-:\Cantor\to O_-$. Using these maps iteratively we can now define the following subsets of $\Cantor$:
		\begin{align*}
			P_{0} &:= X_U,\\
			P_{n} &:= \tau_+(P_{n-1}) \quad\text{ for } n\geq 1,\\
			P_{-n} &:= \tau_-(P_{-n+1}) \quad\text{ for } n\geq 1,\\
			P_{\infty} &:= \{x_+^*\},\\
			P_{-\infty} &:= \{x_-^*\},
		\end{align*} 
		where $x_+^*$ (resp. $x_-^*$) denotes the unique fixed-point of the contraction $\tau_+$ (resp. $\tau_-$), which exists due to Cantor's intersection Theorem.
		
		Note that $P_0 \cap\tau_+(\Cantor) = \emptyset$ as well as $P_0 \cap\tau_-(\Cantor) = \emptyset$, hence by an inductive argument we see that $P_{n}\cap P_0=\emptyset$ for all $n\in(\ZZ\setminus\{0\})\cup\{\infty,-\infty\}$. Actually, we even have $P_{n}\cap P_{m}=\emptyset$ for $n,m\in\ZZ\cup\{\infty,-\infty\}$, if $n\neq m$.\\
		Indeed, assume that $0\leq n< m$ and $x \in P_{n}\cap P_{m}$. Then by the iterative definition of the sets $P_k$ there are points $x_0,y_0\in P_0$ such that $\tau_+^n(x_0) = x = \tau_+^m(y_0)$. Since $\tau_+$ is bijective, this implies $x_0 = \tau_+^{m-n}(y_0)$. However, this is clearly a contradiction to the fact that $\tau_+(\Cantor) = O_+$ and $O_+\cap P_0 = \emptyset$. The same argument works for $\leq m< n \leq 0$, if we replace $\tau_+$ by $\tau_-$. We will define the dissipator $\groupel_U$ as follows:
		\begin{align*}
			\groupel_U(x) &:= \tau_+(x) &&\forall x\in P_{\infty}\cup\textstyle\bigcup_{n\geq 0}P_n,\\
			\groupel_U(x) &:= \tau_-^{-1}(x) &&\forall x\in P_{-\infty}\cup\textstyle\bigcup_{n\leq-1}P_n,\\
			\groupel_U(x) &:= x &&\forall x \in\Cantor\setminus P,
		\end{align*}
		where $P := P_{-\infty}\cup\textstyle\bigcup_{n\in\ZZ}P_n\cup P_{\infty}$.
		By construction we get that $\groupel_U(P_n) = P_{n+1}$ for all $n\in\ZZ$ as well as $\groupel_U(P_{\pm\infty})= P_{\pm\infty}$.\\
		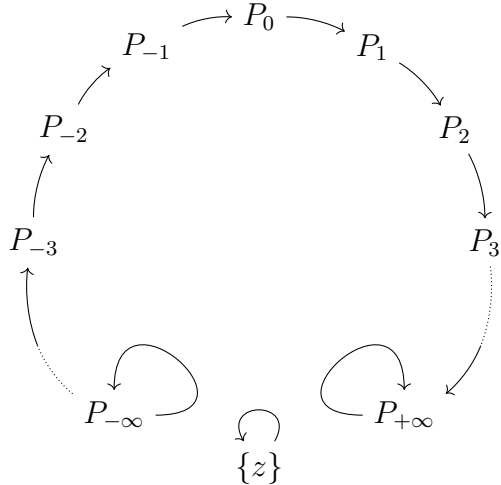
\begin{figure}[htbp]
			\begin{center}
 \begin{tikzpicture}[
	point/.style={
		circle}
	]					
	\node (n0) at (270:3) {$\{z\}$};
	\node (n1) at (230:3)  {$P_{-\infty}$};
	\node (n2) at (180:3)  {$P_{-3}$};
	\node (n3) at (150:3)  {$P_{-2}$};
	\node (n4) at (120:3)  {$P_{-1}$};
	\node (n5) at (90:3)  {$P_{0}$};
	\node (n6) at (60:3)  {$P_{1}$};
	\node (n7) at (30:3)  {$P_{2}$};
	\node (n8) at (0:3)  {$P_{3}$};
	\node (n9) at (310:3)  {$P_{+\infty}$};
	
	\draw[densely dotted] (n1) edge[bend left=15,looseness=0.8] (205:3.25);
	\path (205:3.25) edge [->,bend left=15,looseness=0.8] (n2);
	\path (n2) edge [->,bend left=15,looseness=0.8] (n3);
	\path (n3) edge [->,bend left=15,looseness=0.8] (n4);
	\path (n4) edge [->,bend left=15,looseness=0.8] (n5);
	\path (n5) edge [->,bend left=15,looseness=0.8] (n6);
	\path (n6) edge [->,bend left=15,looseness=0.8] (n7);
	\path (n7) edge [->,bend left=15,looseness=0.8] (n8);
	\draw[densely dotted] (n8) edge [bend left=15,looseness=0.8] (335:3.25);
	\path (335:3.25) edge [->,bend left=15,looseness=0.8] (n9);

	\draw[->] (n0) to [out=60,in=120,loop,looseness=4] (n0);
	\draw[->] (n1) to [out=0,in=90,loop,looseness=6] (n1);
	\draw[->] (n9) to [out=180,in=90,loop,looseness=6] (n9);
\end{tikzpicture}
				\caption{Illustration of the action of $\groupel_U$ on the sets $\{z\}$, $P_{\pm\infty}$ and $P_n$, $n\in\ZZ$.}
			\end{center}
		\end{figure}
		The dissipator  $\groupel_U$ has to be an element of $\CantorGrpStab$, so we have to verify that $\groupel_U$ is a homeomorphism and that it fixes a neighborhood of $z$ pointwise.\\ 
		The second point is easy to see: 
		\[ \supp(\groupel_U) \subseteq  P \subseteq X_U\cup O_+\cup O_- \]
		and hence $\groupel_U(z) = z$. Moreover, $X_U\cup O_+\cup O_-$ is closed, so $\Cantor\setminus X_U\cup O_+\cup O_-$ is an open neighborhood of $z$, where $\groupel_U$ acts trivially.\\
		Since $\groupel_U$ is bijective, it is enough to show that it is also continuous (because a bijective continuous map between a compact space and a Hausdorff space is automatically a homeomorphism.\\
		First of all, note that $\tau_+$ and $\tau_-$ are embeddings and $P_0$ is closed and open in $\Cantor$, so by an inductive argument we can see that also $P_n$ is closed and open for all $n\in\ZZ$. Therefore, for $n\in\ZZ$ the continuity of $\groupel_U$ on $P_n$ is a consequence of the continuity of $\tau_+$ and $\tau_-$.\\
		Furthermore, the set $P$ is closed in $\Cantor$. Indeed, although $\textstyle\bigcup_{n\in\ZZ}P_n$ is a union of countably many closed subsets, which in general will not be closed, if we add the points $x_+^*$ and $x_-^*$ to this union, we recover closedness. In fact, any converging sequence of elements in this union has either infinitely many elements inside $P_n$ for a (unique) $n$, or the limit point of the sequence will be inside $\{x_+^*,x_-^*\}$. Either way, the limit belongs to $P$.\\
		What is left to show is the continuity of $\groupel_U$ at the points $x_+^*$ and $x_-^*$. Since the arguments for both points are exactly the same, we only consider the point $x_+^*$. Let $V$ be a neighborhood of $\groupel_U(x_+^*) = x_+^*$. By the contracting property of $\tau_+$ we can find an index $N\in\ZZ$, such that $P_n\subseteq V$ for all $n\geq N$. Otherwise there would be a sequence $(x_{i_n})_{n\geq1}\subset \Cantor\setminus V$, with $x_{i_n}\in P_{i_n}\setminus V$, such that $i_n\xrightarrow{n\to\infty}\infty$. Then by compactness of $\Cantor$ we can pass to a converging subsequence, which can definitely not converge to $x_+^*\in V$, since $V$ is open. However, this contradicts the contracting property of $\tau_+$.\\
		Therefore, $W := V\setminus(P_{-\infty}\cup\textstyle\bigcup_{n\leq(N-1)}P_n)$ is an open neighborhood of $x_+^*$, which satisfies $\groupel_U(W) \subseteq V$. This shows the continuity of $\groupel_U$ in $x_+^*$ and hence we conclude that $\groupel_U\in\CantorGrpStab$.
		Finally, for all $g\in \group_U$, we can describe the element $\phi(g)$ as follows
			\begin{align*}
				\phi(g) &:= \begin{cases}
								\groupel_U^{k}g\groupel_U^{-k}  & \text{on } \groupel_U^{k}(X_U), \forall k\geq 1\\
								id & \text{ elsewhere}
							\end{cases}\\
						&:= \begin{cases}
								\groupel_U^{k}g\groupel_U^{-k}  & \text{on } P_k, \forall k\geq 1\\
								id & \text{ elsewhere.}
							\end{cases}\\	
			\end{align*}
		In particular, $\supp(\phi(g)) \subseteq P_{\infty}\cup\bigcup_{n\geq0}P_{n}\subseteq X_U\cup O_+$, so since $z\not\in X_U\cup O_+$, we get that $\phi(g)\in\CantorGrpStab$. The fact that $\phi(g)$ is a homeomorphism can be shown similarly to the continuity of $\groupel_U$. This shows that $\CantorGrpStab$ is indeed dissipated.\\
		Lastly, \Cref{lem:dissipated->binate} and \Cref{prop:binate->bAc} imply that $\CantorGrpStab$ is \bAc.
	\end{proof}

	\begin{example}		
		To make things more concrete we will describe the dissipator created in the proof above in the model $\CantorModel$ of the Cantor set. Recall that $\CantorModel$ is the set of sequences $\seq \in 2^\NN$ equipped with the product topology.
		Fix a neighborhood $\widetilde{U}$ of $\zero$. By definition of the product topology there is an integer $M \geq 3$ such that the set $U:=\{\seq\in\CantorModel~|~~\forall k\leq M-2: \seqel = 0\}$ is contained in $\widetilde{U}$.\\
		Now let $\alpha := (\alpha_k)_k$ be the element which satisfies $\alpha_k = 0$ if $k\leq M$ and $\alpha_k = 1$ if $k>M$. One possible partition of $\CantorModel$ using the same notation as in the proof of \Cref{lem:stabilizer->dissipated->bAc} could be obtained as follows. Define the sets
		\begin{align*}
			X_U &:=  \CantorModel\setminus U = \{\seq\in\CantorModel~|~~\exists k\leq M-2: \seqel = 1\},\\
			O_{+} &:= \{\seq \in\CantorModel~|~ \forall k\leq M-2:\seqel=0;~\seqel[][M-1]=\seqel[][M]=1\},\\
			O_{-} &:= \{\seq \in\CantorModel~|~ \forall k\leq M-2:\seqel=0;~\seqel[][M-1]=1;~\seqel[][M]=0\},
		\end{align*}
		and observe that they are all closed and open as well as pairwise disjoint.\\
		An explicit definition of the shrinkings $\tau_+:\Cantor\to O_+$ and  $\tau_-:\Cantor\to O_-$ is given by
		\begin{align*}
			\tau_+(\seq) &= (\overbrace{0,\dots,0}^{M-2\text{ zeros}}\!,1,1,\seqel[][1],\seqel[][2],\dots)\in O_{+},\\
			\tau_-(\seq) &= (\underbrace{0,\dots,0}_{M-2\text{ zeros}}\!,1,0,\seqel[][1],\seqel[][2],\dots)\in O_{-}.\\
		\end{align*}
		Observe that $d(\tau_+(\seq),\tau_+(\seq[y])) = 2^{-M}d(\seq,\seq[y])$ and also $d(\tau_-(\seq),\tau_-(\seq[y])) = 2^{-M}d(\seq,\seq[y])$, hence $\tau_+$ and $\tau_-$ are indeed shrinkings $(2^{-M} < 1)$.\\
		Finally, the partition of $\CantorModel$ is  
		\begin{align*}
			P_{0} &:= X_U = \{\seq \in\CantorModel~|~\exists k\leq M-2:\seqel=1\},\\
			P_{n} &:= \{\tau_+(\seq) \in\CantorModel~|~\seq \in P_{n-1}\} \quad\forall n\geq 1,\\
			P_{\infty} &:= \{(\seq \in\CantorModel~|~\seqel = 1 \Leftrightarrow (k=0 \vee k=-1 ) \mod M \},\\
			P_{-n} &:= \{\tau_-(\seq) \in\CantorModel~|~\seqel \in P_{-n+1}\} \quad\forall n\geq 1,\\
			P_{-\infty} &:= \{\seq \in\CantorModel~|~\seqel = 1 \Leftrightarrow k=-1\mod M\},\\
			I &:= \CantorModel\setminus(P_{-\infty}\cup(\textstyle\bigcup_{n\in\ZZ}P_n)\cup P_{\infty}).
		\end{align*}
		
		Note that $P_{\infty}$ and $P_{-\infty}$ are again just one point each and they can also be seen as the fixed point of the shrinkings $\tau_+$ and $\tau_-$, respectively. That is  
		\begin{align*}
			P_{\infty} = \{\lim_{n\to\infty} \tau_+^n(\seq)~|~\text{for some }\seq\in O_+\},\\ P_{-\infty} = \{\lim_{n\to\infty} \tau_-^n(\seq)~|~\text{for some }\seq\in O_-\}.
		\end{align*}
		The dissipator $\groupel_U$ satisfies
		\begin{align*}
			&\groupel_U|_{I} = id_{I},\\
			&\groupel_U(P_{n}) = P_{n+1} \text{ for } n\in \mathbb{Z},\\
			&\groupel_U(P_{\infty}) = P_{\infty},\\
			&\groupel_U(P_{-\infty}) = P_{-\infty}.
		\end{align*}
		More explicitly, $\groupel_U$ is defined as follows:
		\begin{align*}
			\groupel_U(\seq) =
			\begin{cases}
			\seq & \text{if } \seq \in I\cup P_{\infty}\cup P_{-\infty}\\
			\tau_+(\seq) = (\overbrace{0,\dots 0}^{M-2\text{ zeros}}\!\!\!,1,1,\seqel[][1],\seqel[][2]\dots) & \text{if } \seq \in P_{n} \text{ for } n\geq 0\\
			\tau_-^{-1}(\seq) =(\seqel[][M+1],\seqel[][M+2]\dots) & \text{if } \seq \in P_{-n} \text{ for } n\geq 1.
			\end{cases}
		\end{align*}
		So on $P_n$, $n\geq 0$, $\groupel_U$ shifts the coordinates by $M$ places to the right and inserts a sequence of $(M-2)$ zeroes followed by two ones. For negative $n$ $\groupel_U$ just  \enquote{forgets} the first $M$ coordinates.\\
		Note that the general case $z\neq0$ follows by conjugating this dissipator with a homeomorphism of $\CantorModel$ which maps $z$ to $0$, which exists by \Cref{prop:Homeo(Cantor)-acts-transitively}.
	\end{example}
	
	Analyzing the proof of \Cref{lem:stabilizer->dissipated->bAc} yields an even stronger result:
	
	\begin{corollary}[\Cref{INT_cor:tuple-stabilizer->dissipated}]\label{cor:tuple-stabilizer->dissipated}
		Let $k\in\NN$. For any $k$-tuple $(x_1,\dots,x_k)\in\Cantor^k$ the subgroup
		\[ \CantorGrpStab[(x_1,\dots,x_k)] = \{g\in\CantorGrp~|~g \text{ fixes a neighborhood of } x_i \text{ pointwise}, 1\leq i\leq k\} \]
		is dissipated, so in particular \bAc.
	\end{corollary}
	
	Note that $\CantorGrpStab[(x_1,\dots,x_k)] = \bigcap_{i=1}^k\CantorGrpStab[x_i]$, so $\CantorGrpStab[(x_1,\dots,x_k)]$ is a subgroup of a dissipated group. However, dissipated groups are not closed with respect to taking subgroups in general (see \Cref{exp:non-dissipated-subgroups}). In fact, not even bounded acyclicity is preserved under taking subgroups (see \Cref{exp:non-bAc-subgroups}). Therefore, the corollary above does not follow immediately from the fact that $\CantorGrpStab$ is dissipated for every $z\in\Cantor$.
	
	\begin{proof}
		The proof is essentially the same as the one for just a single point in the Cantor set. The only difference is that we have to be careful about how to choose the shrinkings $\tau_+$ and $\tau_-$ or rather their images $O_+$ and $O_-$.\\
		Here the family $\mathcal{U}$ of subsets of $\Cantor$ consists of the sets $U$, which are closed and open neihgborhoods of \highlight{all} the points $x_1,\dots,x_k$ simultaneously. Moreover, the sets $O_+$ and $O_-$ have to be chosen in such a way that they do not intersect the set $X_U\cup\{x_1,\dots,x_k\}$. However, since $U/\{x_1,\dots,x_k\}$ is definitely not finite, we can certainly find two points $x_+$ and $x_-$ inside $U$ which both have closed and open neighborhoods $O_+$ and $O_-$, respectively, which do not intersect.\\
		The rest of the proof follows along the same lines.
	\end{proof}
	
	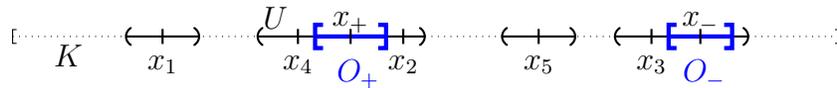
\begin{figure}[htbp]
		\begin{center}		
\begin{tikzpicture}
	\draw[{[-]}, dotted] (0,0) -- (11,0); 
	\node at (0.75, -0.25) {$\Cantor$};
	\draw[thick] (2,0.1) -- (2,-0.1) node[below] {$x_1$};
	\draw[thick] (5.2,0.1) -- (5.2,-0.1) node[below] {$x_2$};
	\draw[thick] (8.5,0.1) -- (8.5,-0.1) node[below] {$x_3$};
	\draw[thick] (3.8,0.1) -- (3.8,-0.1) node[below] {$x_4$};
	\draw[thick] (7,0.1) -- (7,-0.1) node[below] {$x_5$};
	\draw[{(-)}, thick] (1.5,0) -- (2.5,0);
	\draw[{(-)}, thick] (3.25,0) -- (5.5,0);
	\draw[{(-)}, thick] (6.5,0) -- (7.5,0);
	\draw[{(-)}, thick] (8,0) -- (9.8,0);
	\node at (3.5, 0.25)  (U) {$U$};
	\draw[{[-]}, ultra thick, blue] (4,0) -- (5,0);
	\node[blue] at (4.6, -0.5)  (O+) {$O_+$};
	\draw[thick] (4.5,0.1) -- (4.5,-0.1) node[above] {$x_+$};
	\draw[{[-]}, ultra thick, blue] (8.7,0) -- (9.6,0);
	\node[blue] at (9.2, -0.5)  (O-) {$O_-$};
	\draw[thick] (9.15,0.1) -- (9.15,-0.1) node[above] {$x_-$};
\end{tikzpicture}
			\caption{Schematic illustration of a possible situation in the proof of \Cref{cor:tuple-stabilizer->dissipated} for $k=5$.}
		\end{center}
	\end{figure}

\subsubsection{Fat Points on the Cantor Set}\label{subsec:fat-points-on-K}
	
	To make use of \Cref{lem:stabilizer->dissipated->bAc} we would like to find an action of $\CantorGrp$ on some set such that the stabilizer of an element coincides with a subgroup of the form $\CantorGrpStab$ as defined in the previous section. We follow an approach similar to \cite[Sections 3 and 4]{monod2021bounded}.
	
	\begin{definition}[Fat point]\label{def:fat-points}
		For an element $x$ of the Cantor set $\Cantor$ and an embedding $h:\Cantor\to\Cantor$ with $x = h(0)$ we call the tuple $(x,h)$ a \highlight{fat point} with \highlight{core} $x$ and \highlight{tissue} $h(\Cantor)$.\\
		We denote the set of all fat points on $K$ by $\AllFatPoints$.\\
	\end{definition}
	
	Recall that by our convention of embeddings (\Cref{def:embedding}) the image of embeddings is open. In particular, for all fat points $(x,h)$ the tissue $h(\Cantor)$ will be a closed and open subset of $\Cantor$. 

	However, we are not interested in all possible fat points on $\Cantor$ but rather on how the tissues behave in a small neighborhood around the core. Therefore we will consider the following relation on $\AllFatPoints$:

	Let $(x,h)$, $(\bar{x},\bar{h})$ be fat points. We define
	\begin{align*}
		(x,h) \sim (\bar{x}, \bar{h}) \Leftrightarrow& x = \bar{x}\\
		&\text{and  there is a closed and open set } U \ni 0 \text{ s.t. } h|_U \equiv \bar{h}|_U.
	\end{align*}

	\begin{lemma}
		The relation $\sim$ described above is an equivalence relation on $\AllFatPoints$.
	\end{lemma}
	\begin{proof}
		\textbf{reflexivity}: We have $(x,h) \sim (x,h)$ by choosing $U = K$ in the characterization above.
		
		\textbf{symmetry}: Follows easily by the symmetry of the relation condition.
		
		\textbf{transitivity}: Let $(x,h) \sim (x',h')$ and $(x',h') \sim (\bar{x},\bar{h})$.\\
		Then we know that $x = x' = \bar{x}$. Moreover, if $U$ and $V$ are closed and open neighborhoods of $0$ satisfying $h|_U \equiv h'|_U$ and $h'|_V \equiv \bar{h}|_V$, then $U\cap V$ is a closed and open neighborhood of $0$ for which $h|_{U\cap V} \equiv h'|_{U\cap V} \equiv \bar{h}|_{U\cap V}$. Thus $(x,h) \sim (\bar{x},\bar{h})$.		
	\end{proof}
	
	\begin{remark}\label{rem:small_nbhd}
		Observe that any fat point $(x,h)$ depends on $h$ only in a very small neighborhood around $x$. Therefore, we can require the tissue to be contained in an arbitrarily small neighborhood around $x$.\\
		More precisely, for a given neighborhood $V\not\supseteq h(\Cantor)$ of $x$ we can precompose $h$ with an embedding $\iota:\Cantor\to U$, such that:
		\begin{enumerate}[label=(\roman*), itemsep=0pt, nosep]
			\item $U$ is a small neighborhood around $0$ isomorphic to $\Cantor$,
			\item $\iota$ is the identity on a small neighborhood $U'$ of $0$, which is also  isomorphic to $\Cantor$ and
			\item $\iota$ \enquote{squeezes} the remaining part of the Cantor set, in such a way that the image of $h\circ\iota$ is contained in $V$.
		\end{enumerate}
		Indeed, all we have to do is to first choose $U \subseteq h^{-1}(V)$ accordingly, which is possible because $h^{-1}(V)$ is a neighborhood of $0$, since $h$ is an embedding. Then take $U'\subsetneq U$ closed, open and small enough and map the remaining part $\Cantor\setminus U'$ which is homeomoprhic to $\Cantor$, into $U\setminus U'$ which is also homeomorphic to $\Cantor$.\\
		To sum up, for any neighborhood $V$ of $x$ we can choose a fat point $(x,\bar{h})$ which is equivalent to $(x,h)$ and has tissue contained in $V$.
	\end{remark}

	\begin{example}
		Again, for the sake of completeness, we give a concrete example in the model $\CantorModel$ of such a \enquote{squeezing} map.\\
		Suppose that $(x,h)$ is a fat point and $V$ is any neighborhood of $x$, such that $h(\CantorModel)$ is not already contained in $V$. By definition of the product topology there is an integer $M$ such that the closed and open set $O := \{\seq[y]\in2^\NN~|~\forall k\leq M: \seqel[y]=\seqel\}$ is in $V$. Now let $O' := \{\seq[y]\in2^\NN~|~\forall k\leq (M+1): \seqel[y]=\seqel\}$ and define the embedding $p:\CantorModel\to O\subset V$ by
		\begin{align*}
			p|_{O'} &= id_{O'},\\
			p(\seq[y]) &= (\seqel[][1],\seqel[][2],\dots,\seqel[][M], (\seqel[][M+1]+1 \mod 2),\seqel[y][1],\seqel[y][2],\dots) \in O/O'\quad \forall \seq[y]\in\Cantor/O'.\\
		\end{align*}
		Now this map $hp:\CantorModel \to \CantorModel$ is an embedding which satisfies that $(\seq,h) \sim (\seq, hp)$, since $h|_{O'} = (hp)|_{O'}$ and additionally achieves $hp(\CantorModel) \subset V$ as required.
	\end{example}
	
	\begin{notation}
		To simplify the notation, in the following we will refer to an equivalence class of fat points also just as fat point.\\
		Moreover, we define $\FatPoints$ to be the set of equivalence classes of fat points, i.e. $\FatPoints := \AllFatPoints/_{\sim}$.
	\end{notation}
	
	Finally, we can focus on constructing an action of $\CantorGrp$ on some set such that the point-stabilizers have the required from $\CantorGrpStab$. This set will of course be $\FatPoints$.\\
	Let us first define a map $\CantorGrp\times\AllFatPoints \to \AllFatPoints$ by $g\cdot(x,h) = (g(x), gh)$. This map is well-defined because if $h:\Cantor\to\Cantor$ is an embedding with $h(0) = x$, then $gh:\Cantor\to\Cantor$ is also an embedding satisfying $gh(0) = g(x)$. Moreover, we have
	\begin{align*}
		& id_{\Cantor}\cdot(x, h) = (x, h) \text{ and}\\
		&(g_1g_2)\cdot(x,h) = (g_1g_2(x), g_1g_2h) = g_1\cdot(g_2(x), g_2h) = g_1\cdot(g_2\cdot(x, h)).
	\end{align*}
	
	Hence, this map is indeed an action of $\CantorGrp$ on $\AllFatPoints$. In order to show that this induces an action of $\CantorGrp$ on $\FatPoints$, we need to show that $\CantorGrp$ preserves equivalences.\\
	
	Assume that $(x,h) \sim (x, \bar{h})$, i.e. there is a neighborhood $U$ of $0$ such that $h|_U \equiv \bar{h}|_U$. Then for the two fat points $g\cdot(x,h)$ and $g\cdot(x,\bar{h})$ the first condition of the equivalence relation is obviously satisfied. Moreover, we have that $gh|_U \equiv g\circ h|_U \equiv g\circ \bar{h}|_U \equiv g\bar{h}|_U$.\\
	Therefore, $g\cdot(x, h) \sim g\cdot(x, \bar{h})$ and we get a well-defined action of $\CantorGrp$ on $\FatPoints$.\\	
	
	What is left to verify to achieve the goal mentioned in the beginning of this section is to check that the stabilizer of a fat point has indeed the required form. Assume that $g\cdot(x,h) \sim (x,h)$. Then by the definition of the action and the equivalence relation we must have $g(x) = x$. Therefore, $x$ is a fixed point of $g$ and $gh$ and $h$ have the same preimage of $x$, namely 0. Further, there needs to be a closed and open neighborhood $U$ of $0$ on which $gh$ and $h$ coincide. However, since $h$ is an embedding, it sends a closed and open neighborhood $U$ of $0$ onto a closed and open neighborhood $V := h(U)$ of $x$. Thus, $gh_U = h_U \Leftrightarrow g $ fixes $V$ pointwise. In summary, for every fat point $(x,h)$ we get
	\[ \{g\in\CantorGrp~|~g\cdot(x,h) = (x,h)\} = \{g\in\CantorGrp~|~g \text{ fixes a neighborhood of } x \text{ pointwise} \} = \CantorGrpStab[x]. \]
	
	\begin{remark}
		Observe that precisely in the last argument we need the function $h$ to be an embedding, and not just a homeomorphism onto the image, to ensure that open neighborhoods of points are mapped to open (with regards to the topology on $\Cantor$ and not to the induced topology on the image) neighborhoods of the corresponding image point. For general homeomorphisms this is not the case (see \Cref{exp:non-open-homeomorphic-subsets}).
	\end{remark}
	
	We can summarize the previous discussion in the following \namecref{lem:action-of-CantorGrp-on-FatPoints}:
	\begin{lemma}\label{lem:action-of-CantorGrp-on-FatPoints}
		The action of $\CantorGrp$ on fat points $\AllFatPoints$ induces an action on the equivalence classes of fat points $\FatPoints$ and the stabilizer of an equivalence class $[(x,h)]$ of fat points is precisely 
		\[ \CantorGrpStab[x] = \{g\in\CantorGrp~|~g\cdot(x,h) = (x,h)\} = \{g\in\CantorGrp~|~g \text{ fixes a neighborhood of } x \text{ pointwise} \}. \]
	\end{lemma}	
	
\subsubsection{Generic Relation on Fat Points}
	
	\begin{definition}\label{def:generic-rel-on-FatPoints}
		Consider the set $\FatPoints$ and define $(x,h)~\generic~(y,\bar{h}) :\Leftrightarrow x \neq y \text{ and } \im(h)\cap \im(\bar{h}) = \emptyset$ for some $h$, $\bar{h}$ in the equivalence classes.\\
		We call fat points which satisfy this relation \highlight{disjoint}.
	\end{definition}

	\begin{lemma}\label{lem:generic-rel-on-FatPoints}
		The binary relation $\generic$ on $\FatPoints$ is generic.
	\end{lemma}
	\begin{proof}
		Assume that $\{(y^0,h^0),\dots (y^n,h^n)\}\subseteq \FatPoints$ is a finite set of fat points. By \Cref{rem:small_nbhd} we can without loss of generality assume that $\im(h^i)\cap\im(h^j) = \emptyset$ whenever $y^i \neq y^j$ and that the union of all images does not cover $\Cantor$. Now for any point $x\in\Cantor\setminus\{y^0,\dots,y^k\}$ choose a small enough neighborhood $U$ of $x$ such that $U\cap \im(h^i) = \emptyset$ for all $q\leq i \leq k$ and an arbitrary embedding $h:\Cantor\to U$. Then we have $(x,h)~\generic~(y^i,h^i)$ for all $q\leq i \leq k$, hence $\generic$ is indeed generic.
	\end{proof}

\subsubsection{Transitivity of the \texorpdfstring{$\CantorGrp$}{G}-action}
	
	Recall that the generic relation $\generic$ on $\FatPoints$ gives a semi-simplicial set $\FatPoints_{\bullet}^{\generic}$, where $\FatPoints_k^{\generic}$ is the set of $(k+1)$-tuples $((x_0,h_0),\dots (x_k,h_k))$ for which $(x_i,h_i)~\generic~(y_j,h_j)$ for all $0\leq i<j\leq k$ holds.
	
	\begin{lemma}\label{lem:transitive-action-on-fat-points}
		The action of $\CantorGrp$ on $\FatPoints$ described above is highly transitive on generic tuples, that is, it is transitive on $\FatPoints_k^{\generic}$ for any integer $k\geq0$.
	\end{lemma}
	\begin{proof}
		Let $((x_0,h_0),\dots (x_k,h_k))$ and $((\bar{x}_0,\bar{h}_0),\dots (\bar{x}_k,\bar{h}_k))$ be two elements of $\FatPoints_k^{\generic}$. To prove the lemma we want to find an element $g\in\CantorGrp$ which satisfies $g\cdot(x_i,h_i) = (\bar{x}_i,\bar{h}_i)$ for all $1\leq i\leq k$.\\
		First of all, by \Cref{rem:small_nbhd} we can assume that all the tissues $O_i := h_i(\Cantor)$ ($\bar{O}_j := \bar{h}_j(\Cantor)$) for $0\leq k$ are disjoint. We only require $O_i \cap O_j = \emptyset$ and $\bar{O}_i \cap \bar{O}_j = \emptyset$ for $i\neq j$ and not $O_i \cap\bar{O}_j=\emptyset$, which might not be possible if $x_i = x_j$. Moreover, since the $h_i$ and $\bar{h}_j$ are embeddings, all the sets $O_i$ and $O_j$, respectively, are closed and open inside $\Cantor$. Define $O := \Cantor\setminus(\bigcup_{i=0}^kO_i)$ and $\bar{O} := \Cantor\setminus(\bigcup_{j=0}^k\bar{O}_j)$ without loss of generality we can assume $O\neq\emptyset$ and $\bar{O}\neq\emptyset$; otherwise we make the image of some tissues smaller.\\
		Then we have two decompositions of $\Cantor$ into $(k+2)$ non-empty disjoint closed and open sets. We define the desired element $g\in\CantorGrp$ as follows:
		\begin{align*}
			g|_O &: O\to\bar{O},\\
			g|_{O_i} := \bar{h}_i\circ h_i^{-1} &: O_i\to\bar{O}_i \quad\forall 0\leq i\leq k,
		\end{align*}
		where $g|_O$ is an arbitrary homeomorphism given by \Cref{thm:Cantor-sets-are-homeomorphic}. The argument that $g$ is indeed a homeomorphism of $\Cantor$ is the same as in the proof of \Cref{cor:Homeo(Cantor)-acts-highly-transitively}.\\
		Clearly, we have
		\begin{align*}
			g\cdot(x_i,h_i) = (g(x_i),gh_i)	&= (\bar{h}_i\circ h_i^{-1}(x_i), \bar{h}_i\circ h_i^{-1}h_i)\\
			& = (\bar{h}_i(0), \bar{h}_i) = (\bar{x}_i, \bar{h}_i) \quad \text{for all } 0\leq i\leq k,
		\end{align*}
		hence $g$ sends the tuple $((x_0,h_0),\dots (x_k,h_k))$ to the tuple $((\bar{x}_0,\bar{h}_0),\dots (\bar{x}_k,\bar{h}_k))$.\\
		This shows that $\CantorGrp$ acts transitively on $\FatPoints_k^{\generic}$ for any integer $k\geq0$.
	\end{proof}
	\begin{figure}[htbp]
		\begin{center}
 \begin{tikzpicture}
	\draw[{[-]}, dotted] (0,2) -- (11,2); 
	\node at (0.75, 2.25) {$\Cantor$};
	\draw[{[-]}, ultra thick, blue] (1.5,2) -- (2.5,2);
	\node[blue] at (2, 2.25)  (O1) {$O_0$};
	\draw[thick] (2,2.1) -- (2,1.9) node[below] {$x_0$};
	\draw[{[-]}, ultra thick, blue] (8,2) -- (9.5,2);
	\node[blue] at (8.75, 2.25)  (O2) {$O_1$};
	\draw[thick] (8.5,2.1) -- (8.5,1.9) node[below] {$x_1$};
	\draw[{[-]}, ultra thick, blue] (4,2) -- (5.5,2);
	\node[blue] at (4.75, 2.25)  (O3) {$O_2$};
	\draw[thick] (5.2,2.1) -- (5.2,1.9) node[below] {$x_2$};
	\node at (10.5, 2.25) {$O$};
	
	\draw[{[-]}, dotted] (0,0) -- (11,0); 
	\node at (0.75, -0.35) {$\Cantor_{}$};
	\draw[{[-]}, ultra thick, blue] (8.8,0) -- (10,0);
	\node[blue] at (9.4, -0.35)  (O1) {$\bar{O}_0$};
	\draw[thick] (9.2,0.1) -- (9.2,-0.1) node[above] {$\bar{x}_0$};
	\draw[{[-]}, ultra thick, blue] (5,0) -- (8,0);
	\node[blue] at (6.5, -0.35)  (O2) {$\bar{O}_1$};
	\draw[thick] (6,0.1) -- (6,-0.1) node[above] {$\bar{x}_1$};
	\draw[{[-]}, ultra thick, blue] (1.2,0) -- (3,0);
	\node[blue] at (2.1, -0.35)  (O3) {$\bar{O}_2$};
	\draw[thick] (1.7,0.1) -- (1.7,-0.1) node[above] {$\bar{x}_2$};
	\node at (10.5, -0.35) {$\bar{O}_{}$};
	
	\draw[->] (2.1,1.8) to node[midway,fill=white, scale=0.7] {$\bar{h}_0\circ h_0^{-1}$} (8.9,0.3);
	\draw[->] (8.2,1.65) to node[midway,pos= 0.4,fill=white, scale=0.7] {$\bar{h}_1\circ h_1^{-1}$} (6.7,0.2);
	\draw[->] (4.75,1.8) to node[midway,pos= 0.6, fill=white, scale=0.7] {$\bar{h}_2\circ h_2^{-1}$} (2.3,0.2);
\end{tikzpicture}
			\caption{Schematic illustration of the homeomorphism $g\in\CantorGrp$ constructed in the proof of \Cref{lem:transitive-action-on-fat-points} for $k=2$.}
		\end{center}
	\end{figure}

\subsection{Bounded Cohomology of \texorpdfstring{$\Homeo(\Cantor)$}{Homeo(K)}}
	
	We end this section by recalling and proving our main theorem concerning the group of homeomorphisms of the Cantor set. Due to our exhaustive preparatory work, the proof is easy since it is essentially a summary, respectively consolidation, of our previously obtained results.
	
	\begin{theorem}[\Cref{INT_thm:Homeo(K)-is-bAc}]\label{thm:Homeo(K)-is-bAc}
		Let $\Cantor$ be the Cantor set, then $\Homeo(\Cantor)$ is \bAc, i.e. $\HHbR[n][\Homeo(\Cantor)] \cong 0$ for all $n\geq1$.
	\end{theorem}
	\begin{proof}
		As usual, to simplify the notation we set $\CantorGrp := \Homeo(\Cantor)$.\\
		We want to calculate the bounded cohomology of $\CantorGrp$ using \Cref{prop:bounded-cohomology-through-acyclic-resolutions}. By \Cref{prop:generic-rel->LES} the following sequence
		\[ \LES[0]{\RR, \linf{\FatPoints_1^{\generic}}, \linf{\FatPoints_2^{\generic}}, \linf{\FatPoints_3^{\generic}}} \]
		is a resolution of normed $\RR$-modules, i.e. a long exact sequence.		
		Hence, in order to apply \Cref{prop:bounded-cohomology-through-acyclic-resolutions} it would be enough to show that $\HHbc[][\CantorGrp][\linf{\FatPoints_n^{\generic}}]=0$ for all $n\geq 1$. However, by \Cref{lem:transitive-action-on-fat-points} $\CantorGrp$ acts transitively on $\FatPoints_n^{\generic}$, and by \Cref{cor:tuple-stabilizer->dissipated} the stabilizer of every tuple $(x_0,\dots,x_n)$ in $\FatPoints_n^{\generic}$ is \bAc.\\
 		
 		Therefore, by \Cref{cor:bounded-cohomology-through-action-and-generic-relation} we get that
 		
 		\[ \HHbR[i][\CantorGrp] \cong \HH[i][\linf{\FatPoints_{\bullet}^{\generic}}^\CantorGrp] \quad\text{for all } i\geq 0.\]
 		
 		Now since $\CantorGrp$ acts also transitively on $\FatPoints_{n}^{\generic}$ we have $\linf{\FatPoints_{n}^{\generic}}^\CantorGrp \simeq \RR$ for all $n\geq 1$, so every function in $g\in\linf{\FatPoints_{n}^{\generic}}^\CantorGrp$ is constant on $\FatPoints_{n}^{\generic}$. Looking at the definition of the differential map for the cochain complex $(\linf{\FatPoints_{\bullet}^{\generic}}^\CantorGrp,\codiff)$, we get that
 		\begin{align*}
 			\codiff[n]: \linf{\FatPoints_{n}^{\generic}}^\CantorGrp &\to \linf{\FatPoints_{n+1}^{\generic}}^\CantorGrp, \\
 			\codiff[n](f)(x_0,\dots,x_{n+1}) &=\sum_{i=0}^{n+1}(-1)^{i} f(x_0,\dots,\widehat{x}_i\dots,x_{n+1}).
 		\end{align*}
 		The last expression is equal to $0$ if $n$ is even, and $f(x_0,\dots,x_n)$ if $n$ is odd (since $f$ is constant on $\FatPoints_{n}^{\generic}$, it does not matter at which of the $(n+2)$ points we evaluate $f$). Hence, $\codiff[2n] \equiv 0$ and $\codiff[2n+1] \simeq id_{2n+1}$, where the identity is up to identifying $\linf{\FatPoints_{2n+1}^{\generic}}^\CantorGrp$ and $\linf{\FatPoints_{2n+2}^{\generic}}^\CantorGrp$ with $\RR$, for all $n\geq 1$.\\
 		It easily follows that $\HH[i][\linf{\FatPoints_{\bullet}^{\generic}}^\CantorGrp]$ vanishes, so $\CantorGrp$ is \bAc.
	\end{proof}

	\newpage
	
 \section{Thompson's Group \texorpdfstring{$V$}{V}}\label{sec:thompsons-group-V}
	
	In the following chapter we consider another group which arises from a specific type of action on the model $\CantorModel$ of the Cantor set, namely \highlight{Thompson's group} $\V$. This group (as well as the \highlight{Thompson's groups} $F$ and $T$) were introduced by Richard Thompson in 1965 in some unpublished work.\\
	Our goal is to prove \Cref{INT_thm:V-is-bAc}, that is, that $\V$ is \bAc and consequently also deduce \Cref{INT_cor:V-is-universally-bAc}. We will take a very similar approach as in the previous chapter. However, before we dive into the proof of bounded acyclicity we will give a short introduction into $\V$ and discuss some properties.\\
	
	\begin{remark}
		Until now $V$ always represented a module over some ring $\ring$. As already indicated in the preliminary words this gives a notational confliction with Thompson's group $V$. However, throughout this section $V$ will always denote Thompson's group and never a module.
	\end{remark}
		
\subsection{Introduction to Thompson's Groups \texorpdfstring{$F,T,V$}{F,T,V}}
	
	In what follows we will denote by $\dyadics\subseteq\CantorModel$ the subset of sequences $\seq$ which are eventually constant to $0$, i.e. $\exists N\geq1$ such that $\seqel[][n] = 0~\forall n\geq N$. Moreover, for a finite sequence $(r_1,r_2,\dots,r_N)\in\{0,1\}^N$ we define
	\[  \dyadicNgbh{(r_1,r_2\dots,r_N)} := \{\seq[]\in\CantorModel~|~\forall 1\leq k\leq N:~ \seqel[r][k] = \seqel[][k]\}. \]
	
	\begin{notation}
		\begin{itemize}
			\item We call a finite sequence $(r_1,r_2,\dots,r_N)$ a \highlight{subroot (of level $N$)}. Mostly we will shorten the notation and write $r_1r_2\cdots r_N$ instead of $(r_1,r_2\dots,r_N)$ and switch freely between these two notations. 
			\item Moreover, we can link together two subroots $r_1r_2\cdots r_N$ and $s_1s_2\cdots s_M$ of levels $N$ and $M$, respectively, to get a subroot $r_1r_2\cdots r_N\link s_1s_2\cdots s_M$ of level $N+M$. This linking operation works too if $r_1r_2\cdots r_N$ is a subroot of level $N$ and $(s_1,s_2,\dots)$ any element of $\CantorModel$. 
			\item Further, for an element $\seq\in\CantorModel$ we denote by $\seq|_N$ the subroot $(x_1,x_2\dots,x_N)$ of level $N$. In particular, if $N=0$ then $\seq|_N$ denotes an empty sequence.
			\item Finally, recall that we identify $\zero$ with the element $(0,0,0,\dots)$ and $\one$ with the element $(1,1,1,\dots)$. Therefore, combining this with the previous notation, we can write $\sroot{N}$ for the subroot $0\cdots0$ of level $N$ and similarly $\sroot{N}[\one]$ for the subroot $1\cdots1$ of level $N$. 
		\end{itemize}
	\end{notation}

	Using this notation convention we can define Thompson's group $\V$ as follows:
	
	\begin{definition}[Thompson's group $\V$]
		Thompson's group $\V\leq\Homeo(\CantorModel)$ consists of all those elements $g:\CantorModel\to\CantorModel$ for which there exist finitely many pairs of subroots $(r^1,s^1),\dots,(r^n,s^n)$, such that:
		\begin{itemize}[nosep]
			\item $\CantorModel = \bigcup_{i=1}^n {\dyadicNgbh{r^i}}$ is a (finite) disjoint union \label{item1:def-of-V}
			\item $\CantorModel = \bigcup_{i=1}^n \dyadicNgbh{s^i}$ is a (finite) disjoint union \label{item2:def-of-V}
			\item $g\vert_{\dyadicNgbh{r^i}}(\seq[]) = (s^i_1,s^i_2,\dots,s^i_{N_i},\seqel[][M_i+1],\seqel[][M_i+2],\dots) = s^i\link (x_k)_{k>M_i}$, where $M_i$ is the level of $r^i$ and $N_i$ is the level of $s^i$.
		\end{itemize}
		In the following, we will often use the shorter notation $\Vel$ for such an element in $\V$. Moreover, we also write $g(\dyadicNgbh{r^i}) = \dyadicNgbh{s^i}$, if $g$ sends $\dyadicNgbh{r^i}$ to $\dyadicNgbh{s^i}$ via the formula given above.
	\end{definition}

	Note that in the definition above neither $r^i$ and $s^i$, nor $r^i$ and $r^j$ must have the same level for any $i,j$.
	
	\begin{remark}
		Thompson's groups $F$ and $T$ are defined very similarly.\\
		However, for elements $g\in T$ we additionally require that $g$ preserves the \highlight{cyclic order} with respect to $\lorder$ of $\CantorModel$, that is if $r^1\lorder r^2\lorder\dots\lorder r^n$, then $s^j\lorder s^{j+1}\lorder\dots\lorder s^n\lorder s^1\lorder s^2\lorder\dots\lorder s^{j-1}$ for some $1\leq j\leq n$.\\
		Further, for elements $g\in F$ we even require that $g$ preserves the \highlight{total order} with respect to $\lorder$ of $\CantorModel$, that is if $r^1\lorder r^2\lorder\cdots\lorder r^n$, then $s^1\lorder s^2\lorder\cdots\lorder s^n$.\\
		Clearly, we have $F\subseteq T\subseteq V$.\\
		In particular, we can attach to every element $\Vel$ in $\V$ some permutation $\sigma:\{1,\dots,n\}\to\{1,\dots,n\}$ such that $s^{\sigma(1)}\lorder\dots\lorder s^{\sigma(n)}$ is in the correct order. For elements $g\in T$ this permutation is a cyclic permutation and for elements $g\in F$ this permutation is the identity.\\
		For the sake of completeness we point out that Thompson's groups $F$ and $T$ also arise as homeomorphisms of the unit interval and the unit circle, respectively.\\
		We refer the reader to \cite{cannon1996introductory} for a more thorough introduction.
	\end{remark}
	
	\begin{definition}
		Two subroots $r$ and $s$ are called \highlight{disjoint}, if $\dyadicNgbh{r}\cap\dyadicNgbh{s}=\emptyset$.\\
		Further, we call a tuple of finitely many pairwise disjoint subroots $(r^1,\dots,r^n)$ a partition of $\CantorModel$, if $\CantorModel = \bigcup_{i=1}^n \dyadicNgbh{r^i}$. If moreover $r^1\lorder\cdots\lorder r^n$ holds, we call it an ordered partition.
	\end{definition}

	\begin{lemma}\label{lem:existence-of-partition}
		Two subroots $r$ and $s$ are either disjoint or one of the sets $\dyadicNgbh{r}$ and $\dyadicNgbh{s}$ contains the other one.\\
		Consequently, finitely many pairwise disjoint subroots $r^1,\dots,r^n$ can always be extended to a partition of $\CantorModel$.
	\end{lemma}
	\begin{proof}
		The first part is easy. Let us denote $r=(r_1,\dots,r_N)$ and $s=(s_1,\dots,d_M)$ and assume without loss of generality that $M\leq N$. If $r_i = s_i$ for all $1\leq i\leq M$, then $\dyadicNgbh{s}$ is contained in $\dyadicNgbh{r}$; otherwise the sets are disjoint.\\
		For the second part let us denote $r^i = (r^i_1,\dots,r^i_{N_i})$ and set $N :=1 + \max_{1\leq i\leq n}N_i$. Now consider the $2^N$ subroots $s^1 \lorder\cdots\lorder s^{2^N}$ of all possible subroots of level $N$. Note that the collection 
		\[ \{\dyadicNgbh{r^1},\dots,\dyadicNgbh{r^n},\dyadicNgbh{s^1},\dots,\dyadicNgbh{s^{2^N}}\} \]
		forms a cover of $\CantorModel$. From the corresponding $2^N + n$ subroots we remove all the subroots $s$, which are not disjoint with some subroot $r$ and such that $\dyadicNgbh{s}\subseteq\dyadicNgbh{r}$. By our assumption of the pairwise disjointness of the subroots $r^1,\dots,r^n$ and the construction of the subroots $s^1,\dots,s^{2^N}$ we only remove subroots $s$ of the form $s^j$ for some $1\leq j\leq 2^N$, since the level of those is strictly bigger than the level of the subroots $r^i$, $1\leq i\leq n$. The set of remaining subroots is now pairwise disjoint and the collection of neighborhoods $\dyadicNgbh{\cdot}$ still forms a cover of $\CantorModel$, as is easy to see.
	\end{proof}

	We call such a finite sequence \enquote{subroot} because of the identification of $\CantorModel$ as the \highlight{boundary} of an \highlight{infinite rooted binary tree}. We delve further into this identification in the sequel.
	
	\begin{definition}[infinite rooted binary tree and its boundary]
		A \highlight{rooted binary tree} is an acyclic connected graph $\tree$ with an distinguished vertex $r$, called root, such that:
		\begin{itemize}[nosep]
			\item If $\tree$ has more vertices than $r$, then $r$ has valence precisely $2$.
			\item If a vertex $v\in\tree$ has valence more than $1$, then there are exactly two edges $e_v^L$ and $e_v^R$ incident with $v$ but not contained in the geodesic from $r$ to $v$.
		\end{itemize}
		Moreover, the vertices incident to $e_v^L$ and $e_v^R$ which are not $v$ are designated as \enquote{left} and \enquote{right} child of $v$, respectively.
		An \highlight{infinite rooted binary tree} is a rooted binary tree $\tree$, which has no leaves.\\
		The \enquote{boundary} $\bdryTree$ of an infinite rooted binary tree $\tree$ is the set of all infinite directed paths (no backtracking) starting at the root.
	\end{definition}
	Henceforth, $\tree$ stands for an infinite rooted binary tree and $\bdryTree$ for its boundary.\\
	Using the convention of denoting left children by $0$ and right children by $1$, we can write down an element of $\bdryTree$ by an infinite sequence of zeros and ones, e.g. $010100111010\dots$.
	
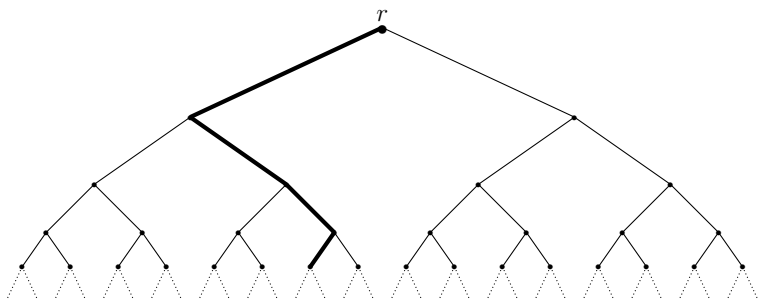
	\includeElement{binarytree}{Illustration of an infinite rooted binary tree, where the path $0110$ starting from the root $r$ is emphasized.}{binarytree}[1]
	
	
	This notation already implies the equivalence (as sets) between $\CantorModel$ and $\bdryTree$ and we can naturally equip $\bdryTree$ with the topology induced by $\CantorModel$. Then it is evident that two elements in $\bdryTree$ are close, if the common beginning of the corresponding infinite paths is large.\\
	The interpretation of $\CantorModel$ as the boundary of an infinite rooted binary tree also gives a good visualization of the action of $\V$ on $\CantorModel$. Observe that every ordered partition $(r^1,\dots,r^n)$ of $\CantorModel$ corresponds exactly to a unique \highlight{finite} subtree of $\tree$. In this bijection the subroot $r^i=(r^i_1,\dots,r^i_{n_i})$ is precisely the end vertex $v$ of the path  $r^i_1\cdots r^i_{n_i}$ starting at the root $r$. Moreover, the set $\dyadicNgbh{r^i}\subseteq\CantorModel$ corresponds to the set of infinite directed paths which traverse the vertex $v$.

	\begin{corollary}[$\V$ acts on the binary tree]
		There is a bijection between $\V$ and triples $(\tree_1,\tree_2,\sigma)$, where $\tree_1$ and $\tree_2$ are finite subtrees of $\tree$ with the same number $n\geq1$ of leaves, and $\sigma:\{1,\dots,n\}\to\{1,\dots,n\}$ is a bijection.
	\end{corollary}
	\begin{proof}
		The proof is straightforward and follows from the identification given in the paragraph above.\\
		More precisely, let $\tree_1$ and $\tree_2$ be two finite subtrees of $\tree$ where $v^1_1,v^1_2,\dots,v^1_n$ and $v^2_1,v^2_2,\dots,v^2_n$ are the leaves, respectively. Moreover, let $\sigma:\{1,\dots,n\}\to\{1,\dots,n\}$ be a bijection. Then we can construct an element $g\in\V$ as follows:
		\begin{itemize}[nosep]
			\item $\CantorModel = \bigcup_{i=1}^n {\dyadicNgbh{v^1_i}}$ is a disjoint union
			\item $\CantorModel = \bigcup_{i=1}^n \dyadicNgbh{v^2_i}$ is a disjoint union
			\item $g(\dyadicNgbh{v^1_i}) = \dyadicNgbh{v^2_i}$ for all $1\leq i\leq n$.
		\end{itemize}
		The construction can be easily seen to be reversible.
	\end{proof}
	
	Hence, $\V$ is just a subset of all homeomorphims from $\bdryTree$ to itself. In fact, $\V$ is closely related to \highlight{Neretin's group $N_2$} which is the group of all \highlight{spheromorphisms} of $\bdryTree$ (see \cite{GarncarekLazarovich}).
	
	Using the identification between $\CantorModel$ and $\bdryTree$ we can depict elements of $\V$ far clearer. In all following figures illustrating an element $\Vel$ of $\V$ we use the following convention. There will be two finite subtrees of $\tree$. The left hand tree illustrates the partition $(r^1,\dots,r^n)$ of $\bdryTree$ where the subtree with root $r^i$ is depicted by $\fBTriangle[i]$. The right hand tree illustrates the partition $(s^1,\dots,s^n)$ of $\bdryTree$ where the subtree which root is $s^i$, hence the image of $r_i$ under $g$, is also depicted by $\fBTriangle[i]$. Additionaly, there will be an arrow connecting the two trees which has a label above with the name of the element of $\V$ under consideration. 
	
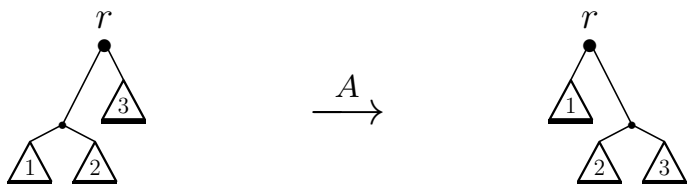
	\includeElement{binarytreeA}{Illustration of $\Vel[A][00,01,1][0,10,11]$.}{A-in-V}
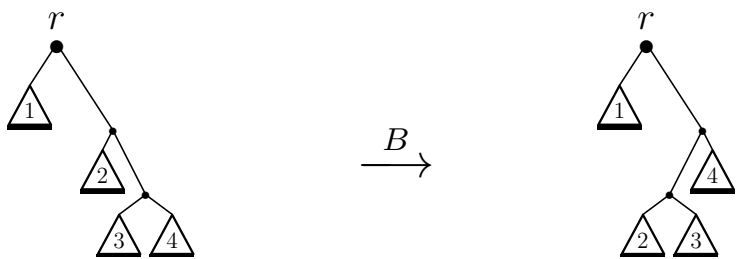
	\includeElement{binarytreeB}{Illustration of $\Vel[B][0,10,110,111][0,100,101,11]$.}{B-in-V}
%
\subsection{Properties of Thompson's Group \texorpdfstring{$V$}{V}}

	Having found this second description of Thompson's group $\V$, we can now prove some introductory properties:
	 
	\begin{lemma}\label{lem:V-acts-transitively}
		$\V$ acts transitively on $\dyadics\subseteq\CantorModel$.
	\end{lemma}
	\begin{proof}
		Let $\seq,\seq[y]\in\dyadics$. If $\seq=\seq[y]$ we can just take the identity mapping $id\in\V$. Hence, assume that $\seq\neq\seq[y]$. Then we can find an integer $M\geq1$, such that $x_n = y_n = 0$ for all $n\geq M$. One possible element $g\in\V$ which sends $\seq$ to $\seq[y]$ is given by \begin{align*}
			g(\dyadicNgbh{\sroot{M}[\seq]}) = \dyadicNgbh{\sroot{M}[{\seq[y]}]}\\
			g(\dyadicNgbh{\sroot{M}[{\seq[y]}]}) =\dyadicNgbh{\sroot{M}[\seq]}
		\end{align*}
		and $g(\seq[z]) = \seq[z]$ for all $\seq[z]\in\CantorModel\setminus(\dyadicNgbh{\sroot{M}[\seq]}\cup\dyadicNgbh{\sroot{M}[{\seq[y]}]})$. The corresponding partition of $\CantorModel$ can be obtained via \Cref{lem:existence-of-partition}, since $\sroot{M}[\seq]$ and $\sroot{M}[{\seq[y]}]$ are disjoint.
	\end{proof}

	\begin{lemma}\label{lem:fixpoints-of-commutators}
		If $g,h\in\V$ both fix a point $\seq[z]\in\Cantor$, then $[g,h]$ not only fixes $z$ but also a neighborhood of $z$ pointwise.
	\end{lemma}
	\begin{proof}
		Let $g$ and $h$ be represented by $\Vel[g][r^1,\dots,r^m][s^1,\dots,s^m]$ and $\Vel[h][\overline{r}^1,\dots,\bar{r}^n][\bar{s}^1,\dots,\bar{s}^n]$, respectively. Moreover, let $i$ and $j$ be the unique indices such that $\seq[z]\in\dyadicNgbh{r^i}$ and $\seq[z]\in\dyadicNgbh{\bar{r}^j}$. For ease of exposition, we drop the indices, i.e. $r = r_i$,$s = s_i$, $\bar{r} = \bar{r}_j$ and $\bar{s} = \bar{s}_j$.\\
		By assumption we have 
		\begin{align}
			g(\seq[z]) = \seq[z] &\Leftrightarrow (s_1,\dots,s_{N},\seqel[z][M+1],\seqel[z][M+2],\dots) = (\seqel[z][1],\dots,\seqel[z][N],\seqel[z][N+1],\dots)\label{eq:1}\\
			h(\seq[z]) = \seq[z] &\Leftrightarrow (\bar{s}_1,\dots,\bar{s}_{\bar{N}},\seqel[z][\bar{M}+1],\seqel[z][\bar{M}+2],\dots)= (\seqel[z][1],\dots,\seqel[z][\bar{N}],\seqel[z][\bar{N}+1],\dots)\label{eq:2}
		\end{align}
		where $M,N,\bar{M},\bar{N}$ denote the levels of $r,s,\bar{r},\bar{s}$, respectively. Note that this gives an equality between $z_k$ and $r_k, s_k, \bar{r}_k, \bar{s}_k$ up to the corresponding levels.\\
		Now let $L := N+\bar{N} +1$. Then $[g,h]$ fixes $\dyadicNgbh{\sroot{L}[z]}$ pointwise. Indeed, for $\seq[z']\in\dyadicNgbh{\sroot{L}[z]}$ we have
		\begingroup
		\allowdisplaybreaks
		\begin{align*}
			[g,h]&(\seq[z'])\\
			&= ghg^{-1}h^{-1}(z_1,\dots,\seqel[z][L],\seqel[z'][L+1],\dots)\\
			&= ghg^{-1}h^{-1}(\bar{s}_1,\dots,\bar{s}_{\bar{N}},\seqel[z][\bar{N}+1],\dots,\seqel[z][L],\seqel[z'][L+1],\dots) &(z\in\dyadicNgbh{\bar{s}})\\
			&= ghg^{-1}(\bar{r}_1,\dots,\bar{r}_{\bar{M}},\seqel[z][\bar{N}+1],\dots,\seqel[z][L],\seqel[z'][L+1],\dots) &(\text{definition of }h)\\	
			&= ghg^{-1}(z_1,\dots,z_{\bar{M}},\seqel[z][\bar{N}+1],\dots,\seqel[z][L]\seqel[z'][L+1],\dots) &(z\in\dyadicNgbh{r})\\
			&= ghg^{-1}(z_1,\dots,z_{\bar{M}},\seqel[z][\bar{M}+1],\dots,\seqel[z][\bar{M}+L-\bar{N}],\seqel[z'][L+1],\dots) &(\text{use }\ref{eq:2})\\	
			&= ghg^{-1}(s_1,\dots,s_N,\seqel[z][N+1],\dots,\seqel[z][\bar{M}+L-\bar{N}],\seqel[z'][L+1],\dots) &(z\in\dyadicNgbh{s} \wedge N +\bar{N}<L)\\
			&= gh(r_1,\dots,r_M,\seqel[z][N+1],\dots,\seqel[z][\bar{M}+L-\bar{N}],\seqel[z'][L+1],\dots) &(\text{definition of }g)\\
			&= gh(z_1,\dots,z_M,\seqel[z][N+1],\dots,\seqel[z][\bar{M}+L-\bar{N}],\seqel[z'][L+1],\dots) &(z\in\dyadicNgbh{r})\\
			&= gh(z_1,\dots,z_M,\seqel[z][M+1],\dots,\seqel[z][M +\bar{M}+L-\bar{N}-N],\seqel[z'][L+1],\dots) &(\text{use }\ref{eq:1})\\
			&= gh(\bar{r}_1,\dots,\bar{r}_{\bar{M}},\seqel[z][\bar{M}+1],\dots,\seqel[z][M+\bar{M}+L-\bar{N}-N],\seqel[z'][L+1],\dots) &(z\in\dyadicNgbh{\bar{r}} \wedge N +\bar{N} < L)\\
			&= g(\bar{s}_1,\dots,\bar{s}_{\bar{N}},\seqel[z][\bar{M}+1],\dots,\seqel[z][M+\bar{M}+L-\bar{N}-N],\seqel[z'][L+1],\dots) &(\text{definition of }h)\\
			&= g(z_1,\dots,z_{\bar{N}},\seqel[z][\bar{M}+1],\dots,\seqel[z][M+\bar{M}+L-\bar{N}-N],\seqel[z'][L+1],\dots) &(z\in\dyadicNgbh{\bar{s}})\\
			&=
			g(z_1,\dots,z_{\bar{N}},\seqel[z][\bar{N}+1],\dots,\seqel[z][M+L-N],\seqel[z'][L+1],\dots) &(\text{use }\ref{eq:2})\\
			&= g(r_1,\dots,r_{M},\seqel[z][M+1],\dots,\seqel[z][M+L-N],\seqel[z'][L+1],\dots) &(z\in\dyadicNgbh{r})\\
			&= (s_1,\dots,s_{N},\seqel[z][M+1],\dots,\seqel[z][M+L-N],\seqel[z'][L+1],\dots) &(\text{definition of }g)\\
			&= (z_1,\dots,z_{N},\seqel[z][M+1],\dots,\seqel[z][M+L-N],\seqel[z'][L+1],\dots) &(z\in\dyadicNgbh{s})\\
			&= (z_1,\dots,z_{N},\seqel[z][N+1],\dots,\seqel[z][L],\seqel[z'][L+1],\dots) = \seq[z'] &(\text{use }\ref{eq:1}).\\
		\end{align*}
		\endgroup
		This concludes the proof.
	\end{proof}
	
	\begin{proposition}[{\cite[§6. Thompson's Group $V$]{cannon1996introductory}}]\label{prop:V-is-finitely-generated}
		Thompson's group $\V$ is finitely generated.\\
		More precisely, the four elements $A,B,C$ and $D$ (see \Cref{fig:A-in-V,,fig:B-in-V,,fig:C-in-V,,fig:D-in-V}, respectively) generate $\V$.
	\end{proposition}

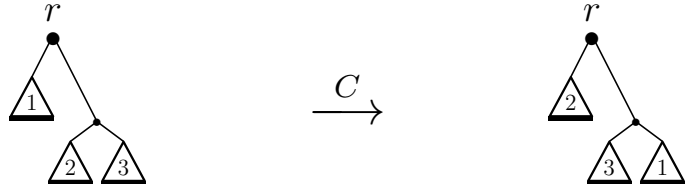
	\includeElement{binarytreeC}{Illustration of $\Vel[C][0,10,11][11,0,10]$.}{C-in-V}
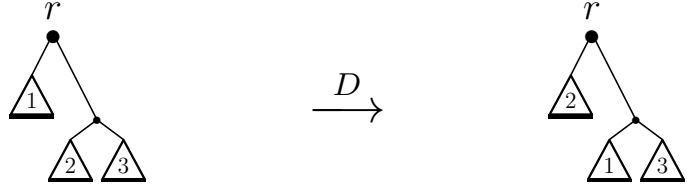
	\includeElement{binarytreeD}{Illustration of $\Vel[D][0,10,11][10,0,11]$.}{D-in-V}	

	\begin{remark}
		In fact, Thompson's group $F$ is generated by $A$ and $B$ and Thompson's group $T$ is generated by $A,B$ and $C$.\\
		Observe that $A$ and $B$ preserve the order of the subroots, $C$ permutes them cyclically and $D$ is an element which permutes two of the three corresponding subroots. Hence, this single permutation together with the $3$-cycle given by the element $C$ is enough to generate \highlight{all} possible non-order-preserving bijections between subtrees $\tree_1$, $\tree_2$ of $\tree$ with the same number of roots.
	\end{remark}
	
	Moreover, in \cite[Theorem 6.4]{SzyminWahl2019} it was proved that $\V$ is $\ZZ$-acyclic, that is, the homology with coefficients in $\ZZ$ vanishes in every positive degree. The fact that $\V$ is rationally acyclic was earlier proven by \citeauthor{Brown1992} in \cite[§7, Theorem 4, p. 134]{Brown1992} and the author already conjectured that $\V$ is also $\ZZ$-acyclic. However, the proof of $\ZZ$-acyclicity is much more intricated. Using $\QQ$-acyclicity of $\V$ together with the fact that $\V$ is uniformly simple \cite{2017SwiatoslawGismatullin}, the authors in \cite{fournierfacio2021binate} were able to deduce that $\V$ is \bAc[2]. In the following chapter we will give a positive answer to the question \cite[Question 6.3]{fournierfacio2021binate} showing that $\V$ is in fact even \bAc.
	
\subsection{Bounded Cohomology of Subgroups of \texorpdfstring{$V$}{V}}\label{sec:BC-of-thompsons-group-V}
	
	Since $\V$ is much more accessible than the group of all homeomorphisms of a Cantor set $\Cantor$, one might be tempted to show that $\V$ is \bAc via a direct argument. For example by showing that $\V$ is binate or dissipated.\\
	However, at least for those two properties, $\V$ is in some sense too small. Indeed, $\V$ cannot be binate and therefore also not dissipated since it is finitely generated (see \Cref{prop:V-is-finitely-generated}) and no binate group can be finitely generated:\\
	Assume that $\group$ is a (non-trivial) finitely generated binate group. Then by definition there exists a homomorphism $\phi:\group\to\group$ and an element $\groupel\in\group$ such that
	\[ h = [\groupel, \phi(h)] = \groupel^{-1}\phi(h)^{-1}\groupel\phi(h) \quad \forall h \in \group. \]
	Using this equality for $h = \groupel\in\group$ we get
	\begin{align*}
		\groupel = [\groupel, \phi(\groupel)] = \groupel^{-1}\phi(\groupel)^{-1}\groupel\phi(\groupel) \text{ and}\\
		\groupel^{-1} = [\groupel, \phi(\groupel^{-1})] = \groupel^{-1}\phi(\groupel)\groupel\phi(\groupel)^{-1},\\
	\end{align*}
	where the second statement follows using that $\phi$ is a homomorphism. Now substituting $\groupel$ on the right-hand side of the second expression with the right-hand side of the first expression we get
	\begin{align*}
		\groupel^{-1} = \groupel^{-1}\phi(\groupel)(\groupel^{-1}\phi(\groupel)^{-1}\groupel\phi(\groupel))\phi(\groupel)^{-1} = \groupel^{-1}\phi(\groupel)\groupel^{-1}\phi(\groupel)^{-1}\groupel.
	\end{align*}
	Hence, $\phi(\groupel)\groupel^{-1}\phi(\groupel)^{-1}\groupel$ is the identity. However, then $id = \phi(\groupel)\groupel^{-1}\phi(\groupel)^{-1}\groupel = [\groupel, \phi(\groupel)]^{-1} = \groupel^{-1}$. This clearly cannot be the case.
	
\subsubsection{Bounded Acyclicity of Point Stabilizers}\label{subsec:bAc-of-point-stabilizers}

	Denote by $\Vstab := \{g\in\V~|~ g\cdot\zero= \zero \}$ the point-stabilizer of $\zero$.\\
	Similarly as we did for fatpoint stabilizers for the full homeomorphism group $\Homeo(\Cantor)$ we would like to prove that $\Vstab$ is \bAc and then use the action of $\V$ on $\CantorModel$ to deduce bounded acyclicity of $\V$. Regarding our approach to handle the fat point stabilizers in \Cref{sec:homeomorphisms-on-K} the first thing which might come to mind is trying to prove that $\Vstab$ is binate, or even dissipated and then again use the results from \Cref{sec:bounded-cohomology} to compute the bounded cohomology of $\V$. However, this will not work for instance due to the following reasons.\\
	Most apparently, $\Vstab$ has elements whose support is $\CantorModel\setminus\{\zero\}$. To give an example consider the element $E\in\V$ depicted in \Cref{fig:element-of-almost-full-support}.
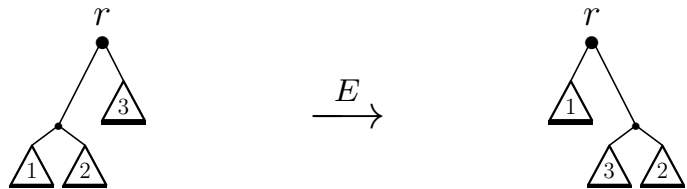
	\includeElement{binarytreeE}{Illustration of $\Vel[E][00,01,1][0,11,10]$. We have $E = C^2\circ D\circ C\circ A$.}{element-of-almost-full-support}

	Therefore, $\Vstab$ cannot be boundedly supported (note the incompatibility with \Cref{def:boundedly_supported}) and thus not dissipated.\\
	However, even if we replace the point stabilizer $\Vstab$ with the rigid stabilizer $\V_{\stab[\zero]} := \{g\in\V~|~ g \text{ fixes a neighborhood of } \zero \text{ pointwise} \}$ in order to avoid this problem, although the resulting subgroup is now boundedly supported, it will not be dissipated either. Indeed, assume that $\V_{\stab[\zero]}$ is the directed union of all $\subgroup_{i}$, $i\in I$ for some (infinite) index set $I$. Using the same notation as in \Cref{def:boundedly_supported} recall that a dissipator $\groupel_i$ of a subgroup $\subgroup_{i}$ must satisfy 
	\begin{enumerate}[label=(\roman*), itemsep=0pt, nosep]
		\item $\groupel_i^{k}(X_i)\cap X_i = \emptyset \quad \forall k\geq 1$,
		\item for all $g\in\group_i$ the element
				\[\phi(g) := \begin{cases}
								\groupel_i^{k}g\groupel_i^{-k}  & \text{on } \groupel_i^{k}(X_i), \forall k\geq 1\\
								id & \text{ elsewhere}
							\end{cases}\]
			  is in $\V_{\stab[\zero]}$.
	\end{enumerate}
	Combining both of those requirements, we see that if $g\in\subgroup_{i}$ is not the identity, $\phi(g)$ would have to be an element in $\V_{\stab[\zero]}\leq \V$ whose corresponding partition of $\CantorModel$ has infinitely many elements. Each subset $X_i$ would have to be partitioned in a non-trivial way. This is a contradiction to the finiteness condition \Cref{item1:def-of-V} in the definition of $\V$.
	
	To sum up, we need a different way to show that $\Vstab$ is \bAc. To approach the bounded cohomology of $\Vstab$ we can use the next \namecref{thm:bAc-through-action}. This result was established by \citeauthor{monod2021lamplighters} as byproduct in the proof of bounded acyclicity of lamplighter groups which in turn was a key step in proving that Thompson's group $F$ is \bAc \cite[Theorem 1]{monod2021lamplighters}. 

	\begin{theorem}[{\cite[Corollary 6]{monod2021lamplighters}}]\label{thm:bAc-through-action}
		Let $\group$ act faithfully on a set $X$. Suppose we have a subset $X_0\subseteq X$ and an element $\groupel\in\group$ such that the following two properties are satisfied:
		\begin{enumerate}[label=(\roman*), itemsep=0pt, nosep]
			\item every finite subset of $\group$ can be conjugated in such a way, that the conjugates are supported in $X_0$, \label{itm:finite-set-conjugation}
			\item $\groupel^k(X_0)\cap X_0 = \emptyset$ for all $k\geq 1$. \label{itm:displacement}
		\end{enumerate}
		Then $\HHbc[n][][W] = 0$, for all separable normed dual $\RR$-modules $W$ and $n\geq 1$. In particular, $\group$ is boundedly acyclic.
	\end{theorem}
	\begin{proof}[Proof sketch:]
		Denote by $\group_0<\group$ the subgroup of elements with support in $X_0$. The conjugation property \ref{itm:finite-set-conjugation} ensures that the subgroup $\group_0$ is coamenable in $\group$. Indeed, for a given group action of $\group$ on a locally convex compact space $Y$ with an $\subgroup$-fixed point the directed set of finite subsets of $\group$ together with the conjugation property give rise to a $\group$-fixed point in $Y$. Thus, using \Cref{rem:coamenable-characterization} we deduce that $\group_0$ is coamenable in $\group$. Moreover, the subgroup $H_0\leq\group$ generated by $\group_0$ together with the element $\groupel$ given in property \Cref{itm:displacement} is also coamenable in $\group$, since it obviously contains $\group_0$. In particular, by \Cref{rem:coamenable->vanishing-BC} there is an injective map $\HHbc\to\HHbc[][H_0]$.\\
		By the dislpacement property \Cref{itm:displacement} the sets $\groupel^n(X_0)$ and $\groupel^m(X_0)$ are disjoint for all $n,m\in\ZZ$ with $n\neq m$. Since we have a faithful group action of $\group$ on $X$, this implies that the conjugates of $\group_0$ by $\groupel^n$ and $\groupel^m$ commute for all $n,m\in\ZZ$ with $n\neq m$, because these two groups act on disjoint sets.\\
		Further, this commutativity of subgroups gives a well-defined surjective group homomorphism $\psi$ from the wreath product $H := \group_0\wr\ZZ := \left(\bigoplus_{\ZZ}\group_0\right)\rtimes\ZZ$ to $H_0$ by mapping $1$ to $\groupel$. This surjective map $\psi$ turns $\module$ into an $H$-module and gives rise to a natural map $\HHbf[\psi]:\HHbc[][H_0]\to\HHbc[][H]$. By \cite[Theorem 3]{monod2021lamplighters} the bounded cohomology $\HHbc[n][H]$ of the wreath product $H$ vanishes for all $n\geq1$ and all separable normed dual $\RR$-modules $\module$. What is hence left to show is that the map $\HHbf[\psi]:\HHbc[][H_0]\to\HHbc[][H]$ is injective. Then we can deduce that $\HHbc[n][H_0]$ and thus also $\HHbc[n]$ vanishes for $n\geq1$.\\
		However, it can be shown that the kernel $\ker(\psi)$ of the homomorphism $\psi$ from $H$ to $H_0$ is metabelian, that is $[\ker(\psi),\ker(\psi)]$ is abelian. In particular, by \Cref{exp:abelian-and-finite-groups-are-amenable} $[\ker(\psi),\ker(\psi)]$ is amenable and for the same reason $\ker(\psi)/[\ker(\psi),\ker(\psi)]$ is amenable. Hence, by \Cref{prop:extensions-amenability} $\ker(\psi)$ is a normal and amenable subgroup of $H$, which forces the map
		\[ \HHbc[n][H/\ker(\psi)]\to\HHbc[n][H] \]
		induced by the projection $\psi:H\to h/\ker(\psi)$ to be an isometric isomorphism for all $n\geq1$ (see \cite[Corollary 8.5.3, Remark 8.5.4]{monod2001continuous}). Since $H_0 \cong H/\ker(\psi)$ this concludes the proof.
	\end{proof}
	
	\begin{remark}
		Note that the advantage in this \namecref{thm:bAc-through-action} compared to the dissipated case is that although we also need an element $\groupel$ which satisfies $\groupel^k(X_0)\cap X_0 = \emptyset$ for all $k\geq 1$, we do not have an assumption about all subsets $\groupel^k(X_0)$, for $k\geq1$ at once. So we do not require an element whose support is the disjoint union of infinitely many subsets as we did in the dissipated case.\\
		This allows to apply this result in more general situations (\cite{monod2021lamplighters, monod2021bounded}).
	\end{remark}
	
	We need one further ingredient to prove the bounded acyclicity of $\Vstab$ given by the following \namecref{lem:element-in-derived-subgroup}. Roughly speaking the \namecref{lem:element-in-derived-subgroup} reveals that commutators of rigid point stabilizers can mimic the action of those elements of $\V$ on almost all of $\CantorModel$. This reflects the fact that the derived subgroups of rigid point stabilizers are very flexible. 
	
	\begin{lemma}\label{lem:element-in-derived-subgroup}
		Let $h\in\V$ be an element which fixes a non-empty open set $U$ pointwise and $\seq[z]$ some point in $U$. Then there is a subroot $r$ and an element $h'\in\V$ such that:
		\begin{enumerate}[label=(\roman*), itemsep=0pt, nosep]
			\item $\seq[z]\in\dyadicNgbh{r}\subseteq U$,
			\item $h|_{\CantorModel\setminus U} = [h',h]|_{\CantorModel\setminus U}$, and
			\item $h'$ fixes $\dyadicNgbh{r}$ pointwise.
		\end{enumerate}
		In particular, $h,h'$ and $[h',h]$ fix every element in $\dyadicNgbh{r}$ pointwise.
	\end{lemma}
	\begin{proof}
		Let $h\in\V$ be represented by $\Vel[h][r^1,\dots,r^m][s^1,\dots,s^m]$.	Since $U$ is open, we can choose $M\geq1$ bigger than all the levels of the subroots $r^i$, $1\leq i\leq m$, such that the subroot $r' = (z_1,\dots,z_M)$ satisfies $\seq[z]\in\dyadicNgbh{r'}\subseteq U$. Then, using \Cref{lem:existence-of-partition}, we have that $\dyadicNgbh{r'}\subseteq\dyadicNgbh{r^i}$ for some $1\leq i\leq m$. In particular $r^i = (z_1,\dots,z_l)$ for some $l< M$. Since $h$ fixes $U$ pointwise, it must be the case that $s^i = r^i$. Indeed, $\dyadicNgbh{r'}\subseteq U$ and thus
		\[ r^i\link z_{l+1}\cdots z_M\link\seq = r'\link\seq = g(r'\link\seq) = s^i\link z_{l+1}\cdots z_M\link\seq \]
		for all sequences $\seq\in\CantorModel$. By choosing $\seq$ to be the constant sequence not equal to $z_M$ we can first deduce that the levels of $r^i$ and $s^i$ must be the same and hence that even $r^i = s^i$ holds. Therefore, $g$ even fixes $\dyadicNgbh{r^i}$ pointwise.\\
		Now we will construct an element $h'\in\V$ which \enquote{exchanges} the set $\cup_{j\neq i}\dyadicNgbh{r^j}$ with a certain subset of $\dyadicNgbh{r'}$. First, by \Cref{lem:existence-of-partition} we can extend $\{r'\}$ to a partition $r',r'^1,\dots r'^n$ of $\CantorModel$. Second, define $r=r'\link z_{M+1}$, $s'^{1} := r'\link \eta\link\sroot{n-1}$ and $s'^j := s\link\eta\link\sroot{n-j}\link 1$ for all $2\leq j\leq n$, where $\eta:=z_{M+1} +1\mod 2$. Then we can define $h'$ on the (in general not ordered) partition 
		\[ (r'^1,\dots,\dots,r'^n,r,s'^1,\dots,s'^{n}), \]
		as follows:
		\begin{align*}
			h'(\dyadicNgbh{r'^j}) &= \dyadicNgbh{s'^j}, \qquad \text{for all } 1\leq j\leq n,\\
			h'(\dyadicNgbh{s'^j}) &= \dyadicNgbh{r'^j}, \qquad \text{for all } 1\leq j\leq n,\\
			h'(\dyadicNgbh{r}) &= \dyadicNgbh{r}.
		\end{align*}
		Thus, $h'$ exchanges $\bigcup_{j=1}^n\dyadicNgbh{r'^j}$ with $\bigcup_{j=1}^n\dyadicNgbh{s'^j}$ and fixes $\dyadicNgbh{r}$.	Further, observe that $\dyadicNgbh{r}\cup\bigcup_{j=1}^n\dyadicNgbh{s'^j}$ is a partition of $\dyadicNgbh{r'}$ which is fixed by $h$ (and hence also $h^{-1}$) pointwise. Therefore, we get
		\[ [h',h](\dyadicNgbh{r}) = h'^{-1}h^{-1}h'h(\dyadicNgbh{r}) = \dyadicNgbh{r}. \]
		Moreover, for a point $\seq$ in $\CantorModel\setminus U$ it holds that $h(\seq)$ is inside $\CantorModel\setminus U \subseteq \bigcup_{j=1}^n\dyadicNgbh{r'^j}$, hence
		\begin{align*}
			&h'h(\seq) \in\bigcup_{j=1}^n\dyadicNgbh{s'^j},
		\end{align*}
		and thus
		\begin{align*}
			&[h',h](\seq) = h'^{-1}h^{-1}h'h(\seq) = h'^{-1}h'h(\seq) = h(\seq).
		\end{align*}
		The last line follows from the fact that $h^{-1}$ fixes $\bigcup_{j=1}^n\dyadicNgbh{s'^j}\subseteq\dyadicNgbh{r'}$ pointwise.\\
		Hence, $h'$ and $r$ fulfill all claimed properties.
	\end{proof}

	\begin{proposition}\label{prop:V0-is-bAc}
		$\Vstab$ is \bAc.
	\end{proposition}
	\begin{proof}
		\begin{claim*}
			$\Vstab' := [\Vstab,\Vstab]$ is \bAc.\\
		\end{claim*} 
		Assuming the claim we can deduce the statement easily. Since the quotient $\Vstab/\Vstab'$ is abelian, so in particular amenable (see \Cref{exp:abelian-and-finite-groups-are-amenable}), $\Vstab'$ is coamenable in $\Vstab$. Hence, $\Vstab$ is \bAc by \Cref{prop:coamenable->bAc}.\\
	
		\begin{claim-proof}
			$\Vstab'$ acts faithfully on $X := \CantorModel$. Choose $X_0 := \dyadicNgbh{10}$. In order to be able to apply \Cref{thm:bAc-through-action} we need an element $\groupel\in\Vstab$ which \enquote{moves $X_0$ away from itself}. More concretely, define the following sets
			\begin{align*}
				O_{-} := \dyadicNgbh{01},\\
				O_{+} := \dyadicNgbh{110}.
			\end{align*}
			Then we can define $\groupel\in\Vstab$ as $\Vel[\groupel][00,010,011,10,110,111][00,01,10,1100,1101,111]$. Note that $g$ is even in $F$ as can be seen in \Cref{fig:groupel}.
			
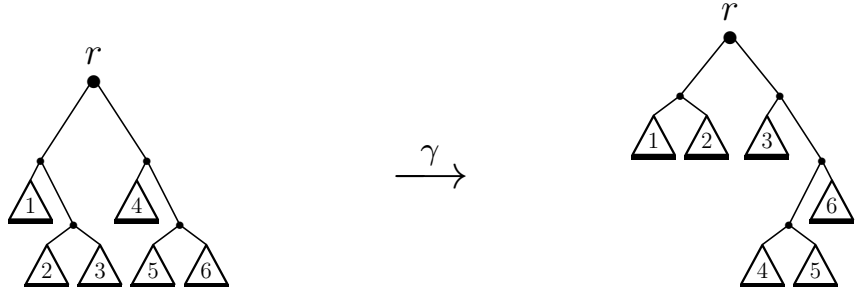
			\includeElement{binarytreeg}{Illustration of $\Vel[\groupel][00,010,011,10,110,111][00,01,10,1100,1101,111]$.}{groupel}[1.5]

			It is easy to see that $\groupel$ fixes $\zero$ and that $\groupel^k(X_0) \subseteq O_{-}$ for all $k\geq1$, so in particular $\groupel^k(X_0)\cap X_0 = \emptyset$. Moreover, we need to verify that $\groupel$ is not only an element of $\Vstab$ but also of $\Vstab'$. However, we have $\groupel = [A,B]$ where $\Vel[A][00,01,1][0,10,11]$ and $\Vel[B][0,10,110,111][0,100,101,11]$ are in $\Vstab$. Therefore, the second assumption of \Cref{thm:bAc-through-action} holds.\\
			
			It remains to verify the first one, i.e. that for given finitely many elements $g_1,\dots,g_n\in\Vstab'$ , we can always conjugate them to be supported in $X_0$.\\
			First of all, note that by \Cref{lem:fixpoints-of-commutators} every element in $\Vstab'$ fixes not only $\zero$, but also a neighborhood of $\zero$ pointwise. Therefore, we can find a (large enough) integer $M\geq 2$ such that the neighborhood $\dyadicNgbh{\sroot{M}}$ is fixed pointwise by all elements $g_1\dots g_n$. Let $t_0 = \sroot{M}$ and $t_k = \sroot{M-k}\link 1$ be the subroot of level $(M-k+1)$, for $1\leq k\leq M$. Observe that $\dyadicNgbh{\sroot{M}} = \dyadicNgbh{\sroot{M+1}}\cup\left(\bigcup_{k=0}^M \dyadicNgbh{t_0\link 1\link t_k}\right)$.\\
			Using this notation we can define $h\in\Vstab$ as 
			\[ \Vel[h][0\link t_0,t_0\link 1\link t_0,\dots,t_0\link 1\link t_M,t_1,\dots,t_M][0\link t_0,\dots,0\link t_M,11,1\link t_0,\dots,1\link t_{M-1}]. \]
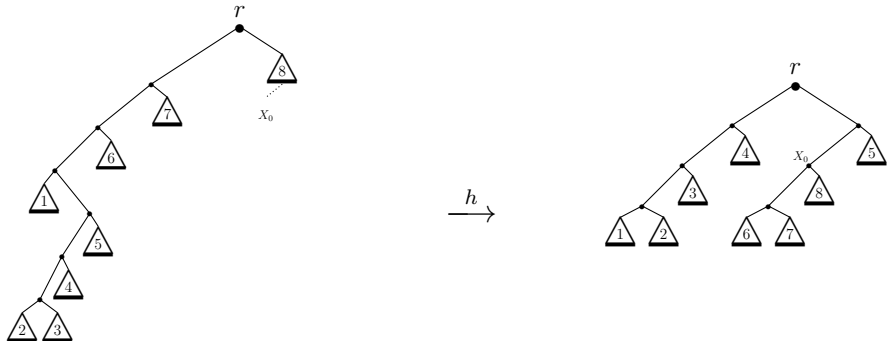
			\includeElement{binarytreeh}{Illustration of $h\in\V$ in the case $m=3$.}{h}[1.0]
			Note that $h$ maps $\CantorModel\setminus\dyadicNgbh{\sroot{M}}=\bigcup_{k=1}^M \dyadicNgbh{t_k}$ into $X_0 = \dyadicNgbh{10}$ and so if $x\not\in X_0$, then $h^{-1}(x)\in\dyadicNgbh{\sroot{M}}$. In particular, we have $\supp(hg_ih^{-1}) \subseteq X_0$ for all $1\leq i\leq n$ as required.\\
			However, we need to conjugate the elements $g_i$ not by a general element of $\Vstab$ but by one of the derived subgroup $\Vstab'$. This is the place where we make use of \Cref{lem:element-in-derived-subgroup}. Indeed, we can apply \Cref{lem:element-in-derived-subgroup} to the element $h\in\V$ and the open set $U = \dyadicNgbh{\sroot{M+1}}$ to obtain an element $h'\in\V$ such that $[h',h]|_{\CantorModel\setminus\dyadicNgbh{\sroot{M+1}}} = h|_{\CantorModel\setminus\dyadicNgbh{\sroot{M+1}}}$. Moreover, $[h',h]\in\Vstab'$ is an element of the derived subgroup. Finally, we get that $\supp([h',h]g_i[h',h]^{-1}) \subseteq X_0$, since $[h',h]^{-1}$ acts on $\CantorModel\setminus\dyadicNgbh{\sroot{M}}\subseteq\CantorModel\setminus\dyadicNgbh{\sroot{M+1}}$ in the same way as $h$.\\
			Therefore, by applying \Cref{thm:bAc-through-action} we deduce that $\Vstab'$ is \bAc, whence it follows that $\Vstab$ is \bAc as well.
		\end{claim-proof}	
	
	\end{proof}
	
	Although, as already mentioned, the map $g\in\V$ in the proof above is an element of $F$, $h\in\V$ does not even preserve the cyclic order of the partition, as can be easily seen in \Cref{fig:h}: the subroot $s_0$, which is not the biggest subroot of the partition $(\sroot{M+1},s_M\dots,s_0,r_M,\dots,r_0)$ is mapped to the biggest subroot in $(\sroot{M+1},r_M,\dots,r_1,11,10\link s_M,\dots,10\link s_0)$.

	\begin{remark}\label{rem:fat-points-for-V}
		In order to show that $\V$ is \bAc, we could have also worked with \enquote{fat points} similar to the case of $\Homeo(\Cantor)$. However, working with point stabilizers simplifies the arguments below.
	\end{remark}
	
	Again, the choice of the point $\zero$ is not particularly relevant. By conjugation, we could replace $\zero$ by any other point in $\dyadics$.
	
	\begin{corollary}\label{cor:stab-of-tuple-is-bAc}
		For a $k$-tuple $(x_1,\dots,x_k)\subseteq \dyadics$, where $x_i \neq x_j$ if $i\neq j$, the stabilizer $\V_{(x_1,\dots,x_k)}$ is \bAc.
	\end{corollary}
	\begin{proof}
		Essentially the same argument as for the stabilizer of a single point works here too. It is merely more cumbersome to describe the elements corresponding to $g$ and $h$.\\
		First, it can be shown that the derived subgroup $\V_{(x_1,\dots,x_k)}' := [\V_{(x_1,\dots,x_k)},\V_{(x_1,\dots,x_k)}]$ is \bAc.\\
		Choose $N\geq1$ large enough, such that all sequences corresponding to the finitely many points $x_1,\dots,x_k$ are constantly equal to $0$ for indices larger than $N$. We can define the set $X_0$ as well as the map $\tilde{g}$ to be \enquote{locally} the same as the map $g$ in the proof of \Cref{prop:V0-is-bAc}. More precisely, set $X_0  := \bigcup_{i=1}^k \dyadicNgbh{\sroot{N}[x_i]\link 10}$, $O_-  := \bigcup_{i=1}^k \dyadicNgbh{\sroot{N}[x_i]\link 01}$, $O_+  := \bigcup_{i=1}^k \dyadicNgbh{\sroot{N}[x_i]\link 110}$. Now define $\tilde{g}\in\V$ for  $\seq[z]\in\dyadicNgbh{\sroot{N}[x_i]}$ as
		\[ \tilde{g}(\seq[z]) := \sroot{N}[x_i]\link g(z_{N+1},z_{N+2},z_{N+3},\dots), \]
		for all $1\leq i\leq k$ as well as the identity on $\CantorModel\setminus(O_-\cup X_0\cup O_+)$.\\
		Then $\tilde{g}$ fixes all points $x_1,\dots,x_k$, and $\tilde{g}\in\V_{(x_1,\dots,x_k)}'$ is a commutator of maps $\tilde{A}$ and $\tilde{B}$ which are analogously locally equal to $A$ and $B$ (see the proof of \Cref{prop:V0-is-bAc}).\\
		
		Now let $g_1,\dots,g_n\in\V_{(x_1,\dots,x_k)}'$. By \Cref{lem:fixpoints-of-commutators} every one of those elements $g_1,\dots,g_n$ fixes a neighborhood of $x_i$ pointwise for $1\leq i\leq k$. Hence, we can find a (large enough) integer $M\geq1$ such that $g_j$ fixes $\dyadicNgbh{\sroot{M}[x_i]}$ for all $1\leq j\leq n$, $1\leq i \leq k$. Now we define $\tilde{h},\tilde{h'}\in\V_{(x_1,\dots,x_k)}$ by
		\begin{align*}
			\tilde{h}(\seq[z]) := \sroot{M}[x_i]\link h(z_{M+1},z_{M+2},z_{M+3},\dots),\\
			\tilde{h'}(\seq[z]) := \sroot{M}[x_i]\link h'(z_{M+1},z_{M+2},z_{M+3},\dots).
		\end{align*}
		for $\seq[z]\in\dyadicNgbh{\sroot{M}[x_i]}$ for all $1\leq i\leq k$ and $\tilde{h}(\seq[z]) := \seq[z] =: \tilde{h'}(\seq[z])$ if $\seq[z]$ is not in any of the sets $\dyadicNgbh{\sroot{M}[x_1]},\dots,\dyadicNgbh{\sroot{M}[x_k]}$. (Here we use again the same definitions for $h$ and $h'$ as in the proof of \Cref{prop:V0-is-bAc}.)\\
		Then we have that $\supp([\tilde{h'},\tilde{h}]g_i[\tilde{h'},\tilde{h}]^{-1}) \subseteq X_0$ and we can use \Cref{thm:bAc-through-action} to deduce that $\V_{(x_1,\dots,x_k)}'$ is \bAc.\\
		The bounded acyclicity of $V_{(x_1,\dots,x_k)}$ follows since $\V_{(x_1,\dots,x_k)}'$ is coamenable in $\V_{(x_1,\dots,x_k)}$.
	\end{proof}

\subsubsection{Generic Relation on \texorpdfstring{$\dyadics$}{Z[1/2]}}

	\begin{lemma}
		The relation $\neq$ on the set $\dyadics$ is a generic relation.
	\end{lemma}
	\begin{proof}
		This follows immediately, since $\dyadics$ is infinite, so for every finite subset $Y\subset \dyadics$ there is another element $x\in \dyadics$ such that $y\neq x$ for all $y\in Y$.
	\end{proof}

	\begin{remark}
		Note that considering the cochain complex $(\linf{\dyadics_\bullet^{\neq}}, \codiff)$ is closely related to the more standard approach of using the alternating cochain complex $(\linf[alt]{\dyadics^\bullet}, \codiff)$, see \cite[Section 4.10]{2017Frigerio}. For example in \cite{fournierfacio2021binate} they were used to calculate the bounded cohomology of Thompson's group $T$.\\
		Indeed, in $\linf{\dyadics_k^{\neq}}$ we consider functions from $\dyadics_k^{\neq}$ to $\RR$. These could be extended by $0$ to obtain functions on $\dyadics^k$. So, in particular, all these extended functions evaluate to $0$ for tuples $(x_0,\dots,x_k)\in\dyadics_k$, where $x_i = x_j$ for some $i\neq j$.\\
		The alternating cochain complex $\linf[alt]{\dyadics^\bullet}$ is defined as the subcomplex of alternating cochains of the cochain complex $\linf{\dyadics^\bullet}$, i.e.  $f\in\linf{\dyadics^k}$ is in $\linf[alt]{\dyadics^k}$ if and only if
		\[ f(x_{\sigma(0)},\dots,x_{\sigma(k)}) = \operatorname{sgn}(\sigma)\cdot f(x_0,\dots,x_k) \]
		for all $(x_0,\dots,x_k)\in\dyadics^{k+1}$ and permutations $\sigma:\{0,\dots,k\}\to\{0,\dots,k\}$, for all $k\geq0$.\\
		Thus, it is easy to see that all $f\in\linf[alt]{\dyadics^k}$ evaluate to $0$ for tuples $(x_0,\dots,x_k)\in\dyadics_k$, where $x_i = x_j$ for some $i\neq j$, as well. However, in this case there are certainly further conditions on the map $f$.\\
		Nevertheless we chose to use generic relations again, since these have a closer relation to the previous section where we worked on $\Homeo(\Cantor)$.
	\end{remark}

\subsubsection{Transitivity of the \texorpdfstring{$\V$}{V}-action}
	
	Recall that every generic relation $\generic$ on a set $X$, gives rise to a semi-simplicial set $X^{\generic}_\bullet$, where $X^{\generic}_n$ is the set of all $(n+1)$-tuples $(x_0,\dots,x_n)\in X^{n+1}$ for which $x_i~\generic~x_j$ holds for all $0\leq j<i\leq n$.
	
	\begin{lemma}\label{lem:V-acts-transitively-on-generic-tuples}
		$V$ acts highly transitively on generic tuples, that is it acts transitive on $\dyadics_k^{\neq}$ for all $k\geq0$.
	\end{lemma}
	\begin{proof}
		The case $k=0$ is handled in \Cref{lem:V-acts-transitively}. The proof for arbitrary $k$ uses the same idea.\\
		Let $(x_1,\dots,x_k),(y_1,\dots,y_k)\in\dyadics_k^{\neq}$ be two generic tuples. Since those are only finitely many points, there is an integer $M\geq1$ such that the sequences corresponding to all of those $2k$ points are constantly equal to $0$ for indices larger than $M$.\\
		Consider the ordered partition $(r^1,\dots,r^{2^M})$ of all subroots of level $M$. There are injections $\sigma_x,\sigma_y:\{1,\dots,k\}\hookrightarrow\{1,\dots,2^M\}$ such that $\sroot{M}[x_i] = r^{\sigma_x(i)}$ and $\sroot{M}[y_i] = r^{\sigma_y(i)}$ for all $1\leq i\leq k$. Using these injections we can define $g\in\V$ on this ordered partition as follows:
		\begin{align*}
			g(\dyadicNgbh{\sroot{M}[x_i]}) = g(\dyadicNgbh{r^{\sigma_x(i)}}) &= \dyadicNgbh{r^{\sigma_y(i)}} = \dyadicNgbh{\sroot{M}[y_i]} &&1\leq i\leq k,\\
			g(\dyadicNgbh{\sroot{M}[y_i]}) = g(\dyadicNgbh{r^{\sigma_y(i)}}) &= \dyadicNgbh{r^{\sigma_x(i)}} = \dyadicNgbh{\sroot{M}[x_i]} &&1\leq i\leq k,\\
			g(\dyadicNgbh{r^i}) &= \dyadicNgbh{r^i} &&\text{ for all } i\not\in\sigma_x(\{1,\dots,k\})\cup\sigma_y(\{1,\dots,k\}).
		\end{align*}
		Thus, $g$ exchanges the neighborhoods $\dyadicNgbh{\sroot{M}[x_i]}$ and $\dyadicNgbh{\sroot{M}[y_i]}$ for all $1\leq i\leq k$ and acts as the identity on all neighborhoods $\dyadicNgbh{r}$, $r\in\{r^1,\dots,r^{2^M}\}$, which do not contain any of the points in the generic tuples $(x_1,\dots,x_k)$ and $(y_1,\dots,y_k)$.
	\end{proof}

\subsection{Bounded Cohomology of \texorpdfstring{$\V$}{V}}

	In the remaining lines we can finally give a proof of \Cref{INT_thm:V-is-bAc} and \Cref{INT_cor:V-is-universally-bAc}. Let us recall the statement of the \namecref{INT_thm:V-is-bAc}:

	\begin{theorem}[\Cref{INT_thm:V-is-bAc}]\label{thm:V-is-bAc}
		Thompson's group $\V$ is \bAc.
	\end{theorem}

	For the sake of completeness we go through the proof in detail, although bounded acyclicity follows from all the results we have collected so far together with the strategy of the proof of \Cref{thm:Homeo(K)-is-bAc}.
	
	\begin{proof}
		By \Cref{prop:generic-rel->LES} the sequence
		\[ \LES[0]{\linf{\dyadics^{\neq}_1}, \linf{\dyadics^{\neq}_2}, \linf{\dyadics^{\neq}_3}} \]
		is exact.\\
		Moreover, $\V$ acts transitively on $\dyadics^{\neq}_n$ for all $n\geq1$ (\Cref{lem:V-acts-transitively-on-generic-tuples}) and the stabilizer of a generic tuple $(x_0,\dots,x_n)\in \dyadics^{\neq}_n$ is \bAc by \Cref{cor:stab-of-tuple-is-bAc}. Hence, we are in place to apply \Cref{lem:transitive-actions->stabilizer} and deduce that $\HHbc[][\V][\linf{\dyadics^{\neq}_n}] = 0$ for all $n\geq1$.\\
		Now by \Cref{prop:bounded-cohomology-through-acyclic-resolutions} we get that 
		\[ \HHbR[i][\V] \cong \HH[i][\linf{\dyadics_{\bullet}^{\neq}}^\V] \quad\text{for all } i\geq 0.\]
		Since $\V$ acts transitively on $\dyadics_{\bullet}^{\neq}$ we can argue in the same way as in the end of the proof of \Cref{thm:Homeo(K)-is-bAc} to show that $\HH[i][\linf{\dyadics_{\bullet}^{\neq}}^\V] = 0$ for all $i\geq1$, hence $\V$ is \bAc.
	\end{proof}

	\begin{remark}\label{rem:BC-of-Thompsons-group-F-and-T}
		It was recently shown \cite[Theorems 1 and 2]{monod2021lamplighters}, that also Thompson's group $F$ is \bAc. This also reveals the bounded cohomology of Thompson's group $T$ \cite[Corollary 6.12]{fournierfacio2021binate} with real coefficients. However, it turns out that $T$ is not \bAc. Instead, $\HHbR[][T]$ is isomorphic to the polynomial ring $\RR[x]$ with $\abs{x}=2$ and the bounded Euler class of $T$ is a polynomial generator of $\HHbR[][T]$. See \cite[Section 10.2]{2017Frigerio} for a detailed discussion on the bounded Euler class.
	\end{remark}
	
	Recall that the homology of $\V$ with coefficients in $\ZZ$ vanishes \cite[Theorem 6.4]{SzyminWahl2019}. Thus, 

	\begin{corollary}[\Cref{INT_cor:V-is-universally-bAc}]\label{cor:V-is-universally-bAc}
		$\V$ is \highlight{universally boundedly acyclic}, that is $\HHbc[n][\V][\mathbb{K}] = 0$ for all $n\geq1$ and all complete valued fields $\mathbb{K}$.
	\end{corollary}
	\begin{proof}
		Indeed, by \cite[Thereom 5.2]{fournierfacio2021binate} this follows from $\ZZ$-acyclicity (see paragraph above \Cref{sec:BC-of-thompsons-group-V}) and $\RR$-bounded acyclicity of $\V$, proven in \Cref{thm:V-is-bAc}.
	\end{proof}

	Hence \Cref{thm:V-is-bAc} makes Thompson's group $\V$ the first example of a finitely generated universally \bAc group, as well as the first example of a universally \bAc group of type $F_\infty$.

	\newpage
	
 \section{Conclusion, Outlook and Further Questions}\label{sec:outlook-and-open-questions}

	The two approaches to calculate the bounded cohomology of the full group of homeomorphisms of a Cantor set $\Homeo(\Cantor)$ and of Thompson's group $\V$ which is a subgroup of $\Homeo(\Cantor)$ given in this thesis share a lot of similarities. In both cases we deduced the vanishing of bounded cohomology by using actions on certain sets. We needed one structural property on the sets, namely that they carry a generic relation (\Cref{def:generic_rel}) which is preserved by the corresponding group action. We further exploited that the group actions were highly-transitive on specific tuples of points and more importantly, that the stabilizers of those tuples where themselves boundedly acylic subgroups.\\
	
	Here we want to point out two things:\\
	First, in our cases the group actions were rich enough to deduce high transitivity not only on the underlying set but also on tuples (restricted by the generic relation) of arbitrary size. However, as suggested by \Cref{cor:bounded-cohomology-through-action-and-generic-relation} and \Cref{rem:gen-of-bAc-stabilizer->bAc-module} such a strong condition on the group action is in general not necessary.\\
	
	Second, although the overall strategy of the proofs of \Cref{INT_thm:Homeo(K)-is-bAc} and \Cref{INT_thm:V-is-bAc} is the same the underlying arguments were quite different. On the one hand, in the case of $\Homeo(\Cantor)$ we showed that the stabilizers of tuples/points under the corresponding action are boundedly acyclic by proving that they are dissipated groups (\Cref{lem:stabilizer->dissipated->bAc,cor:tuple-stabilizer->dissipated}). As discussed in the preamble of \Cref{subsec:bAc-of-point-stabilizers}, the stabilizers of tuples/points under the associated action of $\V$ are neither dissipated nor binate due to the fact that Thompson's group $\V$ is in some sense \enquote{too small}. Therefore, we needed a different result (\Cref{thm:bAc-through-action}) to show that the stabilizer subgroups are boundedly acyclic. On the other hand, in the case of Thompson's group $\V$ the underlying set whereon $\V$ acted was the set of dyadic points $\dyadics$ in the binary Cantor model, which is somehow a more natural set than the artificially created set of fat points (\Cref{def:fat-points}) for the $\Homeo(K)$ action. For the letter mentioned set we additionally had to argue that the stabilizer of a fat point is indeed the rigid stabilizer subgroup (see \Cref{subsec:fat-points-on-K} for details), whereas the stabilizer subgroup of a dyadic point is evident.\\
	
	These differences arose from looking at the situation from two points of view. Regarding $\Homeo(\Cantor)$, we first observed that rigid stabilizers of the natural action on the set $\Cantor$ are boundedly acyclic and looked for an action of $\Homeo(\Cantor)$ such that the point stabilizers coincide with these subgroups. Regarding $\V$, we only considered the natural group action on $\CantorModel$ and found a (dense) subset $\dyadics$ for which we proved the bounded acyclicity of the point-stabilizers directly.\\
	
	
	Nevertheless, we want to stress that one could match these two strategies even more. This could for example be achieved by working with rigid stabilizers for Thompson's group $\V$ (\Cref{rem:fat-points-for-V}) or by applying \Cref{thm:bAc-through-action} to rigid stabilizers with respect to the natural action $\Homeo(\Cantor)\curvearrowright\Cantor$. Therefore, we expect that there is a more general subtle reason for the vanishing of bounded cohomology of the groups under consideration and hence pose the following question:
	
	\begin{question}
		Let $\group$ be a topological full group of a Cantor set $\Cantor$ and $S\subseteq\Cantor$ a dense subset. Suppose $\group$ acts highly transitively on $S$. Is $\group$ boundedly acyclic?
	\end{question}
	
	Here we use the definition of topological full group of a Cantor set $\Cantor$ given in \cite{Giordano1999}:
	
	\begin{definition}[{\cite[Definition 2.2]{Giordano1999}}]
		Let $\group$ be a subgroup of $\Homeo(\Cantor)$.\\
		To any finite family $(U_k,\groupel_k)_{k=1,\dots,n}$, where $U_k\subseteq\Cantor$ are closed and open subsets, $\groupel_k\in\group$ and
		\[ \Cantor = \bigsqcup_{k=1}^n U_k = \bigsqcup_{k=1}^n \groupel_k(U_k), \]
		we associate the homeomorphism $g$ of $\Cantor$ defined by
		\[ g(x) = \groupel_k(x), \text{for all } x\in U_k. \]
		Then the \highlight{topological full group} $[[\group]]$ of $\group$ is the group of all such homeomorphisms.
	\end{definition}
	
	Apart from the whole group $\Homeo(\Cantor)$, the groups of all quasi-isometries of $\Cantor$ and of all almost-isometries of $\Cantor$ are both examples of topological full groups, see \cite[Section 7, Lemma 7.2]{2017SwiatoslawGismatullin} for the relevant definitions and the result.\\
	
	A positive answer to the question above would handle the particular cases of $\Homeo(K)$ and Thompson's group $V$ simultaneously. However, there are many more groups to which such a statement would be applicable.\\
	Anyway, this question definitely needs further investigation and could be part of future work. Moreover, so far we only considered groups acting on Cantor sets. However, maybe the overall strategy could be applied to more general situations.
	
	\begin{question}
		To which other sets $X$ (and groups $\group$) is the strategy of the proofs of \Cref{INT_thm:Homeo(K)-is-bAc} and \Cref{INT_thm:V-is-bAc} also applicable? More precisely, what are sufficient assumptions on the set $X$ and/or the group $\group$ such that we have a generic relation $\generic$ on $X$ (preserved by the group action) for which the point stabilizers of the associated semi-simplicial set are (uniformly) \bAc?
	\end{question}

	In this setting it might be interesting to look at groups acting on a \highlight{dendrite}, which is a locally connected, connected, compact metrizable space, without simple closed curves. See \cite{Duchesne2018Groups, Duchesne2018Structural} for further information on dendrites.

	\newpage
	
	\appendix
	\addcontentsline{toc}{part}{Appendix}
 \section{Cantor sets and maps between them}\label{APX_sec:cantor-sets-and-maps-between-them}
	
	\subsection{Topological properties of the Cantor set}
	As usual, we denote by $\Cantor$ a Cantor set. Our goal is to show that $\Homeo(\Cantor)$ acts transitively on the Cantor set $\Cantor$. In order to prove this in \Cref{APX_cor:Homeo(Cantor)-acts-transitively} we need some auxiliary results.
	
	\begin{lemma}[{\cite[Section 12, Theorem 4]{moise2013geometric}}]\label{APX_lem:epsilon-cover-of-Cantor-sets}
		Let $\Cantor$ be a Cantor set and $\epsilon>0$. There is a finite collection $\mathcal{W} := \{W_1,\dots,W_{n(\epsilon)}\}$ of non-empty disjoint closed and open subsets of $\Cantor$, such that the diameter $\diam(W_i) = \sup_{x,y\in W_i} \CantorMetric(x,y)$ is smaller than $\epsilon$ for all $1\leq i\leq n$, and that $\Cantor = \bigcup_{i=1}^n W_i$.\\
	\end{lemma}
	We call such a collection an \highlight{\CantorPartition of $\Cantor$}.
	\begin{proof}
		The sets
		\[ B(x,\epsilon/2)\quad x\in\Cantor \]
		form an open cover of $\Cantor$. By compactness, there are finitely many points $x_1,\dots,x_m\in\Cantor$ such that $\Cantor = \bigcup_{i=1}^m B(x_i,\epsilon/2)$.\\
		For all $1\leq i\leq m$ define the two disjoint closed sets $A_i := \overline{B(x_i,\epsilon/2)}$ and $B_i := \Cantor\setminus B(x_i,\epsilon)$. Fix an index $i$. Since $\Cantor$ is totally disconnected, by \Cref{lem:closed-subsets-of-compact-metric-space}, we can decompose $\Cantor$ into two closed, disjoint sets $C_{A_i}$ and $C_{B_i}$ which contain $A_i$ and $B_i$ respectively. Observe that $B_i\subseteq C_{B_i}$ and $C_{A_i}\cap C_{B_i} = \emptyset$ imply that $C_{A_i}\subseteq B(x_i,\epsilon)$, in particular $\diam(C_{A_i}) < \epsilon$. Moreover note that $C_{A_i}$ is closed and open, since $C_{B_i}$ is closed.\\
		We can now define $W_i$ inductively:\\
		First, set $W_1 := C_{A_1}\supseteq B(x_1,\epsilon/2)$ and then define
		\[ W_i := C_{A_i}\setminus (\bigcup_{j=1}^{i-1} W_j) \]
		for $i\geq2$.\\
		$W_1$ is closed and open, hence, by induction also $W_i$ is closed and open since it arises from a closed set $C_{A_i}$ intersected with finitely many closed sets $\Cantor\setminus W_1,\dots,\Cantor\setminus W_{i-1}$. Furthermore we have
		\[ \bigcup_{i=1}^m W_i = \bigcup_{i=1}^m C_{A_i} \supseteq \bigcup_{i=1}^m A_i \supseteq \bigcup_{i=1}^m B(x_i,\epsilon/2) = \Cantor \]
		hence the collection $\{W_1,\dots,W_m \}$ covers $\Cantor$. In order to finish the argument, we just have to drop all the empty sets $W_i$, which might occur for $i\geq 2$, if $C_{A_i}\subseteq\bigcup_{j=1}^{i-1} W_j$. Then (after re-indexing accordingly) we get the desired collection 
		\[ \mathcal{W} := \{W_1,\dots,W_{n(\epsilon)}\}. \]
		Note, since the sets $W_1,\dots,W_n(\epsilon)$ are all closed, disjoint and cover $\Cantor$, they are in particular also open sets. In particular, by \Cref{thm:clopen-subset->Cantor-set} the set $W_i$ is a Cantor set for all $1\leq i\leq n(\epsilon)$.
	\end{proof}

	\begin{corollary}\label{APX_cor:exact-cover-of-Cantor-sets}
		Let $\Cantor$ be a Cantor set, $\epsilon>0$ and $m\geq1$. Then there is a collection $\mathcal{W} := \{W_1,\dots,W_{m}\}$ of exactly $m$ non-empty disjoint closed subsets of $\Cantor$, such that $\Cantor = \bigcup_{i=1}^m W_i$ and that $\diam(W_i)<\epsilon$ for all $1\leq i< m$.\\
		Additionally, all the sets $W_1,\dots,W_{m}$ are also Cantor sets.
	\end{corollary}
	We call such a collection an \highlight{\exactCantorPartition[$m$] of $\Cantor$}.
	\begin{proof}
		This is an immediate consequence of \Cref{APX_lem:epsilon-cover-of-Cantor-sets}. Indeed, since every Cantor set has infinitely many points, we can choose $\epsilon'$, with $\epsilon>\epsilon'>0$, such that the \CantorPartition[$\epsilon'$] $\mathcal{\widetilde{W}} = \{W_1,\dots,W_{n(\epsilon')}\}$ has at least $n$ elements. Now we can just amalgamate the last $n(\epsilon') - (m-1)$ sets to obtain a collection of exactly $m$ disjoint non-empty closed subsets 
		\[ \mathcal{W} := \{W_1,\dots,W_{m-1},\bigcup_{i=m}^{n(\epsilon)}W_i\} \]
		which partition the Cantor set. Moreover, since we did not change the first $(m-1)$ sets, the statement about the diameter is clear.\\
		The last statement follows again from \Cref{thm:clopen-subset->Cantor-set}.
	\end{proof}

	\subsection{Cantor sets are homeomorphic and \texorpdfstring{$\Homeo(\Cantor)$}{Homeo(K)} acts transitively}

	\begin{theorem}[\Cref{thm:Cantor-sets-are-homeomorphic}]\label{APX_thm:Cantor-sets-are-homeomorphic}
		Every two Cantor sets are homeomorphic.
	\end{theorem}
	\begin{proof}
		Let $\Cantor$ and $\Cantor'$ be two Cantor sets. We want to find a homeomorphism $f:\Cantor\to\Cantor'$.	However, since we have no explicit description of the two given Cantor sets, we have also no chance in giving an explicit homeomorphism. In fact, the homeomorphism $f$ we are going to construct is given very implicitly. The main tools we are going to use over and over again are \Cref{APX_lem:epsilon-cover-of-Cantor-sets} and \Cref{APX_cor:exact-cover-of-Cantor-sets}. This will eventually allow us to define a bijective map $f:\Cantor\to\Cantor'$, for which we have to show continuity in order to deduce that it is also a homeomorphism (\cite[Theorem 7.8]{bredon_topology_1993}).\\

		We will now define a sequence of maps $\seq[\sigma][n]$ between finite collections of certain subsets of $\Cantor$ and $\Cantor'$ of diameter converging to $0$, which will lead to the construction of $f:\Cantor\to\Cantor'$.\\
		First, let $\seq[\epsilon][n]$ be a sequence with $\seqel[\epsilon][n] = 2^{-n}$. Then define the sequences  $(\mathcal{W}^n)_{n}$, $(\mathcal{W'}^n)_{n}$ of collections of non-empty disjoint closed subsets of $\Cantor$ and $\Cantor'$ respectively, inductively as follows:\\
		Let $\mathcal{W}^1 := \{W^1_1,\dots,W^1_{n(\epsilon_1)}\}$ be an \CantorPartition[$\epsilon_1$] and $\mathcal{W'}^1 := \{W'^1_1,\dots,W'^1_{n(\epsilon_1)}\}$ be an \exactCantorPartition[$n(\epsilon_1)$]. Observe that the diameter of every set $W_i$ in $\mathcal{W}^1$ is smaller than $\epsilon_1$, which might not be the case for $W'_{n(\epsilon_1)}$ in $\mathcal{W'}^1$. However, since the cardinality of both collections is the same, we can choose an arbitrary bijection $\sigma_1:\mathcal{W}^1\to\mathcal{W'}^1$  between the sets.\\
		Assume now that, for $1\leq k<n$, $\mathcal{W}^k$ and $\mathcal{W'}^k$ as well as $\sigma_k:\mathcal{W}^k\to\mathcal{W'}^k$ are already constructed. We will define $\mathcal{W}^n$, $\mathcal{W}^n$ and $f_n$ depending on the parity of $n$:\\
		
		If $n$ is even, we first construct an \CantorPartition[$\epsilon_n$] $\mathcal{W}'^n_j = \{W'^n_{j,1},\dots,W'^n_{j,n_j(\epsilon_n)}\}$ for every set $W'^{n-1}_j\in\mathcal{W}'^{n-1}$ and then define 
		\[ \mathcal{W}'^n := \bigcup_{W'^{n-1}_j\in\mathcal{W}'^{n-1}} \mathcal{W}'^n_{j} = \{W'^n_1,W'^n_2,W'^n_3,\dots\} \]
		Further, using \Cref{APX_cor:exact-cover-of-Cantor-sets}, for every set $W'^{n-1}_j\in\mathcal{W}'^{n-1}$ we partition the set $W^{n-1}_i = \sigma_{n-1}^{-1}(W'^{n-1}_j)\in\mathcal{W}^{n-1}$ into exactly $n_j(\epsilon_n)$ many disjoint non-empty closed subsets $\mathcal{W}^n_i := \{W^n_{i,1},\dots,W^n_{i,n_j(\epsilon_n)}\}$ and similarly define
		\[ \mathcal{W}^n := \bigcup_{W^{n-1}_i\in\mathcal{W}^{n-1}} \mathcal{W}^n_i  = \{W^n_1,W^n_2,W^n_3,\dots\} \]
		Note that $\abs{\mathcal{W}^n} = \abs{\mathcal{W}'^n}<\infty$. Finally, we choose a bijection $\sigma_n:\mathcal{W}^n\to\mathcal{W}'^n$, which satisfies $\sigma_n(W^n_{i,l}) \in\mathcal{W}'^n_j$ if and only if $\sigma_{n-1}(W^{n-1}_i) = W^{n-1}_j$ for all $W^n_{i,l}\in\mathcal{W}^n$. This is possible, because $\abs{\mathcal{W}^n_i} = \abs{\mathcal{W}'^n_j}$ holds by construction.
		
		If $n$ is odd we do essentially the same thing but we switch the roles of $\mathcal{W}^{n-1}$ and $\mathcal{W}'^{n-1}$. That is, first we partition the sets in $\mathcal{W}^{n-1}$ into sets of diameter less than $\epsilon_n$, and set
		\[ \mathcal{W}^n := \bigcup_{W^{n-1}_i\in\mathcal{W}^{n-1}} \mathcal{W}^n_i \]
		Afterwards we decompose the set $W'^{n-1}_j = \sigma_{n-1}(W^{n-1}_i)\in\mathcal{W}'^{n-1}$ into exactly $n_i(\epsilon_n)$ many sets using \Cref{APX_cor:exact-cover-of-Cantor-sets} and define
		\[ \mathcal{W}'^n := \bigcup_{W'^{n-1}_j\in\mathcal{W}'^{n-1}} \mathcal{W}'^n_j \]
		The bijection $\sigma_n:\mathcal{W}^n\to\mathcal{W'}^n$ can then be chosen to have the same properties as in the case of $n$ being even.\\
		
		Not that it was necessary to distinguish the case where $n$ is even or odd since we need that the maximal diameter of a set in both sequences of collections $(\mathcal{W}^n)_{n}$, $(\mathcal{W}'^n)_{n}$ converges to $0$. If we would apply \Cref{APX_lem:epsilon-cover-of-Cantor-sets} to the sets in $\mathcal{W}^n$ and \Cref{APX_cor:exact-cover-of-Cantor-sets} to the sets in $\mathcal{W'}^n$,for all $n\geq1$, we would not have a control on the maximal diameter of a set in $\mathcal{W'}^n$.\\
		Moreover, recall that if $W\in\mathcal{W}^{n+1}_i\subseteq\mathcal{W}^{n+1}$ is a subset of $W^n_i\in\mathcal{W}_{n}$, then $\sigma_{n+1}(W)\in\mathcal{W}'^{n+1}_j\subseteq\mathcal{W}'^{n+1}$, where $\sigma_n(W^n_i) = W'^n_j$. So in particular, $\sigma_{n+1}(W)\subseteq \sigma_n(W^n_i)$ holds.
		
		\begin{claim}\label{claim:shrinking-subsets}
			Let $n\geq1$. For all $W\in\mathcal{W}^{2n}\cup\mathcal{W}^{2n-1}$ we have $\diam(W) < \epsilon_{2n-1}$.
			Similarly, for all $W'\in\mathcal{W'}^{2n}\cup\mathcal{W'}^{2n+1}$ we have $\diam(W') < \epsilon_{2n}$.
		\end{claim}
		\begin{claim-proof}[Proof of \Cref{claim:shrinking-subsets}]
			For $W\in\mathcal{W}^{2n-1}$ and $W'\in\mathcal{W}'^{2n}$ this follows immediately from the construction of the collection of subsets $\mathcal{W}^{2n-1}$ and $\mathcal{W}'^{2n}$. Moreover, every element of $\mathcal{W}^{2n}$ is a subset of an element in $\mathcal{W}^{2n-1}$ and every element of $\mathcal{W}'^{2n+1}$ is a subset of an element in $\mathcal{W}^{2n}$, so the diameter cannot increase.
		\end{claim-proof}
	
		We can finally focus on defining the map $f:\Cantor\to\Cantor'$. To every $x\in\Cantor$, we can associate a (unique) sequence of indices $\seq[n^x]$, such that $\forall k\geq1:~x\in W^k_{\seqel[n^x]}\in\mathcal{W}^k$. (The collection $\mathcal{W}^k$ consists of disjoint sets and covers $\Cantor$, so $\seqel[n^x]$ exists and is unique for all $k\geq1$.)\\
		We define $f(x)$ to be the unique point in $\bigcap_{k=1}^\infty \sigma_k(W^k_{\seqel[n^x]})$.
		\begin{claim}\label{claim:well-defined}
			The map $f$ constructed above is well-defined, i.e. the intersection is non-empty and consists of a single point.
		\end{claim}
		\begin{claim-proof}[Proof of \Cref{claim:well-defined}]
			Notice that the sequence $(W^k_{\seqel[n^x]})_k$ is nested, i.e.
			\[ W^1_{\seqel[n^x][1]}\supseteq W^2_{\seqel[n^x][2]}\supseteq W^3_{\seqel[n^x][3]}\dots\]
			hence, by the last lines above \Cref{claim:shrinking-subsets} also 
			\[ \sigma_1(W^1_{\seqel[n^x][1]})\supseteq \sigma_2(W^2_{\seqel[n^x][2]})\supseteq \sigma_3(W^3_{\seqel[n^x][3]})\dots \]
			holds. Therefore $(\sigma_k(W_{\seqel[n^x]}))_k$ is a sequence of non-empty nested closed (hence compact) subsets of $\Cantor'$, so their intersection is non-empty by Cantor's Intersection Theorem. Moreover, since $\diam(\sigma_{2k}(W^{2k}_{\seqel[n^x]}))_{2k}) < \epsilon^{2k} = 2^{-2k}$ the intersection cannot contain more than one point, hence $f$ is well-defined.
		\end{claim-proof}
	
		\begin{claim}\label{claim:bijective}
			The map  $f:\Cantor\to\Cantor'$ is bijective.
		\end{claim}
		\begin{claim-proof}[Proof of \Cref{claim:bijective}]
			Since for all $n\geq1$ the map $\sigma_n:\mathcal{W}^n\to\mathcal{W'}^n$ is a bijection, we can use the maps $\sigma_n^{-1}:\mathcal{W'}^n\to\mathcal{W}^n$, to define a map $g:\Cantor'\to\Cantor$, analogously to the construction of $f$.\\
			One can then easily deduce that $g\circ f = id$ and $f\circ g = id$, so $f$ is bijective.
		\end{claim-proof}
	
		\begin{claim}\label{claim:homeomorphism}
			The bijective map $f:\Cantor\to\Cantor'$ is continuous, hence a homeomorphism. 
		\end{claim}
		\begin{claim-proof}[Proof of Claim~\ref{claim:homeomorphism}]
			Let $U'\subseteq\Cantor'$ be an open subset and $x'\in U'$. Let $x := f^{-1}(x')$. For the sequence of indices $\seq[n^{x'}]$ associated to $x'$ (such that $\forall k\geq1:~x'\in W'^k_{\seqel[n^{x'}]}\in\mathcal{W'}^k$), the diameter of the sets $W'^k_{\seqel[n^{x'}]}$ converges to $0$, hence there is $N\geq1$, such that $W'^N_{\seqel[n^{x'}][N]}\subseteq U'$. Now the set $W := \sigma_n^{-1}(W'^N_{\seqel[n^{x'}][N]})\in\mathcal{W}_N$ is open and closed in $\Cantor$ and a neighborhood of $x$.\\			
			Observe that for every point $y\in W$, the associated sequence of indices $\seq[n^y]$ has to satisfy $W'^N_{\seqel[n^y][N]} = W$, by the uniqueness of the subset in $\mathcal{W}^N$ containing $y$.\\
			Thetefore it holds that $f(W) = f(\sigma_n^{-1}(W'^N_{\seqel[n^{x'}][N]})) \subseteq W'^N_{\seqel[n^{x'}][N]} \subseteq U'$, so in particular $f^{-1}(U')$ contains an open neighborhood $W$ of $x$, hence $f$ is continuous.
		\end{claim-proof}
		This concludes the implicit construction of the homeomorphism $f:\Cantor\to\Cantor'$ between any two given Cantor sets $\Cantor$ and $\Cantor'$.
	\end{proof}

	An immediate corollary of \Cref{APX_thm:Cantor-sets-are-homeomorphic} is the transitivity of $\Homeo(\Cantor)$, which was crucial in order to show that $\Homeo(\Cantor)$ is \bAc.
	
	\begin{corollary}\label{APX_cor:Homeo(Cantor)-acts-transitively}
		$\Homeo(\Cantor)$ acts transitively on $\Cantor$.
	\end{corollary}
	\begin{proof}
		Let $x,y\in\Cantor$. We want to find $f\in\Homeo(\Cantor)$ such that $f(x) = y$.\\
		We use the same notation as in the proof of \Cref{APX_thm:Cantor-sets-are-homeomorphic}.\\
		We can adapt this proof in such a way that for every $n\geq1$ we do not choose an arbitrary bijection $\sigma_n:\mathcal{W}^n\to\mathcal{W'}^n$, but choose $\sigma_n$ in such a way, that if $x\in W\in\mathcal{W}^n$, then $\sigma_n(W)$ is the unique element in $\mathcal{W'}^n$ containing $y$. This will eventually lead to $f(x) = y$. Indeed $f(x)$ is the unique point in $\bigcap_{k=1}^\infty f_k(W^k_{\seqel[n^x]})$. However, by construction of $\sigma_k$, for all $k\geq 1$, $f_k(W^k_{\seqel[n^x]})$ contains $y$, hence so will their intersection.
	\end{proof}
	
	\begin{remark}
		Even much more is true. One could choose two countable dense subsets $D,D'\subseteq\Cantor$, together with a homeomorphism $h:D\to D'$ and extend $h$ to a homeomorphism $f:\Cantor\to\Cantor$ on the whole Cantor set (see \cite[Section 12]{moise2013geometric}).
	\end{remark}

	\newpage
	
 \section{Ultrafilters and Ultralimits}\label{APX_sec:Ultrafilters-and-ultralimits}
	
	In the proof of \Cref{prop:generic-rel->LES} we give in this thesis we need the concepts of ultrafilters and ultralimits. For that reason we first give a short introduction into those tools and then conclude by giving the proof of \Cref{prop:generic-rel->LES} (\Cref{APX_prop:generic-rel->LES} below).
	
\subsection{Filters and Ultrafilters}

	\begin{definition}[Filter, Ultrafilter]
		Let $\Omega$ be a set. A \highlight{filter} $\filter$ over $\Omega$ is a family of subsets of $\Omega$ which satisfies the following three properties:
		\begin{enumerate}[label=(\roman*), itemsep=0pt, nosep]
			\item $\filter$ is not empty, \label{filter-itm:1}
			\item $\filter$ is closed under finite intersections: if $U,V \in \filter$, then $U\cap V \in \filter$ and \label{filter-itm:2}
			\item $\filter$ is closed under taking supersets: if $U \in \filter$ and $U\subseteq V$, then $V \in \filter$. \label{filter-itm:3}
		\end{enumerate}
		Moreover, a filter $\ultrafilter$ is called \highlight{ultrafilter} if it is not strictly contained in any other filter.
	\end{definition}
	
	\begin{lemma}\label{lem:family-of-subsets-give-rise-to-filter}
		Let $\mathscr{B}$ be a non-empty family of subsets of $\Omega$ which is closed under intersection, i.e. if $B,B'\in \mathscr{B}$, then also $B\cap B' \in \mathscr{B}$. Then there exists a filter $\filter$ on $\Omega$ which contains $\mathscr{B}$.
	\end{lemma}
	\begin{proof}
		We can simply define $\filter$ as the closure of $\mathscr{B}$ with respect to taking supersets, i.e. $\filter := \{U\subseteq\Omega~|~\exists B\in\mathscr{B}:~B\subseteq U\}$.\\
		Ad \ref{filter-itm:1}: $\filter$ clearly contains $\mathscr{B}$, so it is in particular non-empty.\\
		Ad \ref{filter-itm:2}: Now let $U,V\in \filter$, then there are two sets $B,B'\in\mathscr{B}$, such that $B\subseteq U$ and $B'\subseteq V$. However, since $\mathscr{B}$ is closed under intersections we get that $B\cap B'\in\mathscr{B}$, hence $U\cap V\in\filter$, as $B\cap B'\subseteq U\cap V$.
		Ad \ref{filter-itm:3}: This is clear, since $\filter$ contains all supersets of elements in $\mathscr{B}$.
	\end{proof}

	\begin{lemma}\label{lem:characterization-of-ultrafilter}
		Let $\Omega$ be a set and $\filter$ a filter on $\Omega$. The following are equivalent:
		\begin{enumerate}[label=(\roman*), itemsep=0pt, nosep]
			\item $\filter$ is an ultrafilter. \label{char-ultrafilter-itm:1}
			\item For every subset $U\subseteq\Omega$, we have either $U\in\filter$ or $\Omega\setminus U\in\filter$.\label{char-ultrafilter-itm:2}
		\end{enumerate}
	\end{lemma}
	\begin{proof}
		$\ref{char-ultrafilter-itm:1}\Rightarrow\ref{char-ultrafilter-itm:2}:$ Assume that there exists $U\subseteq\Omega$ such that $U$ and $\Omega\setminus U$ are not in the ultrafilter $\filter$.\\
		
		\begin{claim*}
			If $V \notin \filter$, then $\exists W\in \filter$ such that $V\cap W =\emptyset$.
		\end{claim*}
		\begin{claim-proof}
			Indeed, if $V\cap W \not = \emptyset$ for all $W\in\filter$, then we could define $\filter'$ as the closure with respect to taking supersets of the set $\filter\cup\{V\}$ to get a filter (\Cref{lem:family-of-subsets-give-rise-to-filter}) which strictly contains $\filter$. A contradiction to $\filter$ being an ultrafilter.
		\end{claim-proof}

		Applying this to the sets $U$ and $\Omega\setminus U$, we get two sets $V,W \in\filter$ such that $U\cap V=\emptyset \Leftrightarrow V\subseteq \Omega\setminus U$ and $\Omega\setminus U\cap W=\emptyset \Leftrightarrow W\subseteq U$. However, then $V\cap W=\emptyset\in\filter$, contradicting the properties of a filter.\\
		
		$\ref{char-ultrafilter-itm:2}\Rightarrow\ref{char-ultrafilter-itm:1}:$ Assume $\filter'$ is a filter which strictly contains $\filter$. Pick $U\in\filter'\setminus\filter$. Then by $\ref{char-ultrafilter-itm:2}$ we have $\Omega\setminus U\in\filter\subseteq\filter'$. But then $\Omega\setminus U\cap U=\emptyset\in\filter'$. Again contradicting the fact that $\filter'$ is a filter, hence $\filter$ is an ultrafilter.
	\end{proof}

	\begin{example}[Fr\'echet Filter]
		Consider $\Omega = \NN$ and let $\filter$ be the set of cofinite subsets of $\NN$, i.e. subsets, for which their complement in $\NN$ is finite.\\
		Then $\filter$ is a filter on $\NN$. It is called the \highlight{Frech\'et Filter}.
	\end{example}

	\begin{example}[Neighborhood filter]
		Let $X$ be a topological space and $x\in X$. Then the set $\filter_x := \{U\subseteq X~|~U \text{ is a neighborhood of } x\}$ is a filter on $X$. We call $\filter_x$ the \highlight{neighborhood} filter of $x$.
	\end{example}

	\begin{theorem}[Ultrafilter lemma \textbf{(UL)}]\label{thm:ultrafilter-lemma}
		Given a set $X$, every filter $\filter$ on $X$ may be extended to an ultrafilter $\ultrafilter$.
	\end{theorem}
	\begin{proof}
		Consider the non-empty partially order set $\mathcal{S}$ of filters over $X$ containing $\filter$, where the order is induced by inclusion. $\mathcal{S}$ is not empty, since $\filter\in\mathcal{S}$. We want to apply Zorn's Lemma to deduce that there exists a maximal element, which has to be an ultrafilter extending $\filter$. Hence let 
		\[ \filter_0 \subseteq \filter_1\subseteq \filter_2\dots \]
		be an ascending chain of filters containing in $\mathcal{S}$. Define $\filter_\infty := \bigcup_{n\geq 0}\filter_n$.
		\begin{claim*}
			$\filter_\infty\in\mathcal{S}$ is an upper bound for the ascending chain $\filter_0 \subseteq \filter_1\subseteq \filter_2\dots$
		\end{claim*}
		\begin{claim-proof}
			It is clear that $\filter_n \subseteq \filter_\infty$ for all $n\geq0$, hence also $\filter \subseteq \filter_\infty$. So we just have to show that $\filter_\infty$ is indeed a filter in order to conclude that $\filter_\infty\in\mathcal{S}$.\\
			Let $U,V\in\filter_\infty$ then there is an index $N\geq0$ such that $U,V\in\filter_N$, hence also $U\cap V\in\filter_N\filter_\infty$ since $\filter_N$ is a filter. Furthermore, if $U\in\filter_\infty$ and $U\subseteq V$, then again, there is an index $N\geq0$, such that $U\in\filter_N$, hence also $V\in\filter_N$. Therefore we conclude that $V\in\filter_\infty$ which shows that $\filter_\infty$ is indeed a filter.
		\end{claim-proof}
		So the assumptions for Zorn's Lemma are satisfied and we can conclude that there is some ultrafilter $\ultrafilter$ containing $\filter$.
	\end{proof}

	\begin{remark}\hfill
		\begin{itemize}[nosep]
			\item Filters (hence also ultrafilters) clearly exist, since for any set $X$, the set $\filter := \{X\}$ is a filter on $X$.
			\item We have seen that under Zermelo–Fraenkel set theory \textbf{(ZF)} the Axiom of Choice \textbf{(AC)} implies the Ultrafilter Lemma. However, the other implication is not true. One can show that the Ultrafilter Lemma is strictly weaker than the Axiom of Choice \cite[Corollary 2.4]{BellFremlin1972}.
			\item Nevertheless, under \textbf{(ZF)} the Ultrafilter Lemma together with the \highlight{Krein-Milmann-Theorem} are equivalent to \textbf{(AC)} \cite[Corollary 2.3]{BellFremlin1972}.
		\end{itemize}
		In the last two items one has to use that under \textbf{(ZF)} the ultrafilter lemma \textbf{(UL)} is equivalent to the \highlight{boolean prime ideal theorem} \textbf{(BPIT)}. This can be shown by considering the boolean ring $(\mathcal{P}(X),+,\cdot)$, where $x + y = x\setminus y \cup y\setminus x$, and explicitly constructing a prime ideal $\mathcal{I}\subsetneq\mathcal{P}(X)$ from an ultrafilter $\ultrafilter$ over $X$ and vice versa.
	\end{remark}
	
	\begin{definition}
		Let $\Omega$ be a set. A \highlight{principal filter} on $\Omega$ is the smallest filter induced by a given subset $\omega\subseteq\Omega$. Equivalently, it is the filter given by $\filter(\omega) := \{U\subseteq\Omega~|~ \omega\subseteq U\}$.\\
		If $\abs{\omega} = 1$, then $\filter(\omega)$ is even an ultrafilter, which is called $\highlight{trivial ultrafilter}$ based on $\omega$.
	\end{definition}
	
	\begin{remark}
		We can make the last statement even more precise. That is, a principal filter $\filter(\omega)$ is an ultrafilter if and only $\abs{\omega} = 1$.\\
		Indeed, if we pick $x_0\in\omega$, then $\filter(\omega)\subseteq\filter(\{x_0\})$. Hence $\filter(\omega)$ can only be an ultrafilter if this is an equality. This only happens if $\omega = \{x_0\}$, otherwise $\{x_0\}\not\in\filter(\omega)$.
	\end{remark}

	\begin{remark}
		Note that for a topological space $X$ and every point $x\in X$, the neighborhood filter $\filter_x$ is a subset of the trivial ultrafilter $\filter(\{x\})$.
	\end{remark}

	\begin{example}
		If $\Omega$ is a finite set, then every ultrafilter $\ultrafilter$ is a trivial filter, i.e. $\exists \omega\in\Omega$ such that  $\ultrafilter = \{U\subseteq\Omega~|~\omega\in U\}$.\\
		This can be easily seen by \Cref{lem:characterization-of-ultrafilter}.
	\end{example}

\subsection{Limits along Filters and Ultralimits}
	
	\begin{definition}[convergent filter]
		Let $X$ be a topological space and $\filter$ a filter on $X$. We say that $\filter$ is \highlight{convergent} if there exists a $x_0\in X$ such that every neighborhood of $x_0$ is in the filter $\filter$. In other words, if $\filter_{x_0}\subseteq\filter$ for some $x_0\in X$.\\
		We say that $\filter$ \highlight{converges to} $x_0$ or that $x_0$ \highlight{is a limit of}  $\filter$.
	\end{definition}
	
	\begin{lemma}\label{lem:uniqueness_of_limit_in_Hausdorff_spaces}
		Assume moreover that $X$ is Hausdorff, then the limit of a convergent filter $\filter$ is unique.
	\end{lemma}
	\begin{proof}
		Assume $x,y\in X$ are both limits of the filter $\filter$. Since $X$ is Hausdorff we can choose neighborhoods $U$ of $x$ and $V$ of $y$ such that $U\cap V = \emptyset$. However, by definition of convergence of a filter, we have $U,V\in\filter$.\\
		This clearly contradicts that $\filter$ is closed under intersections.
	\end{proof}
	
	\begin{example}
		Consider the two point set $X=\{0,1\}$ with the non-Hausdorff topology $\tau_X=\{\emptyset,\{0\}, \{0,1\}\}$. Then $\filter = \{\{0\}, \{0,1\}\}$ is a filter which converges to $0$ and to $1$.
	\end{example}

	\begin{proposition}
		Let $X$ be a topological space. Then the following are equivalent.
		\begin{enumerate}[label=(\roman*), itemsep=0pt, nosep]
			\item $X$ is compact. \label{char-convergent-ultrafilter-itm:1}
			\item Every ultrafilter $\ultrafilter$ is convergent. \label{char-convergent-ultrafilter-itm:2}
		\end{enumerate}
	\end{proposition}
	\begin{proof}
		$\ref{char-convergent-ultrafilter-itm:1}\Rightarrow\ref{char-convergent-ultrafilter-itm:2}:$ Assume that $X$ is compact and $\ultrafilter$ is a non-convergent ultrafilter. Then, by definition of convergence of a filter, for every $x\in X$ there exists a neighborhood $N_x$ of $x$, which is not in $\ultrafilter$. Hence we can find an open set $U_x$ containing $x$ which is also not in $\ultrafilter$, since $U_x\in\ultrafilter$ together with $U_x\subseteq N_x$ would imply that $N_x\in\ultrafilter$. Therefore we get that $X\setminus U_x\in\ultrafilter$.\\
		Now $X = \bigcup_{x\in X}U_x$, hence by compactness of $X$ we can choose finitely many points $x_1,\dots,x_n$ such that $X = \bigcup_{i=1}^nU_{x_i}$.\\
		However, then the finite intersection $\bigcap_{i=1}^nX\setminus U_{x_i} = X \setminus\bigcup_{i=1}^nU_{x_i} = \emptyset$ is in $\ultrafilter$. This is a contradiction.\\		
		
		$\ref{char-convergent-ultrafilter-itm:2}\Rightarrow\ref{char-convergent-ultrafilter-itm:1}:$ Assume that every ultrafilter on $X$ converges and that $(U_i)_{i\in I}$ is an open cover of $X$, which does not have a finite subcover. Define the family of subsets $\mathscr{B} := \{X\setminus \cup_{i\in F}U_i~|~F\subseteq I, \abs{F} <\infty\}$. We have $(X\setminus \cup_{i\in F}U_i)\cap(X\setminus \cup_{i\in F'}U_i) = X\setminus \cup_{i\in F\cup F'}U_i$, hence $\mathscr{B}$ is closed under taking intersections. Therefore, by \Cref{lem:family-of-subsets-give-rise-to-filter} and \Cref{thm:ultrafilter-lemma}, there is an ultrafilter $\ultrafilter$ containing $\mathscr{B}$. By assumption the ultrafilter converges, hence there is a $x_0\in X$ such that $\filter_{x_0}\subseteq\ultrafilter$. Now note that since the $U_i$ are open it follows that if $x_0\in U_i$, then $U_i\in\filter_{x_0}\subseteq\ultrafilter$. However, also $X\setminus U_i\in \mathscr{B}\subseteq\ultrafilter$. Hence $x_0\not\in U_i$ for all $i\in I$. This contradicts the fact $(U_i)_{i\in I}$ is an open cover of $X$, since $x_0\not\in\bigcup_{i\in I}U_i=X$.
	\end{proof}

	\begin{lemma}[Pushforward filter]
		Let $X,Y$ be sets, $\filter$ a filter on $X$ and $f:X\to Y$ a function. Then the set $f_*(\filter) := \{V\subseteq Y~|~f^{-1}(V) \in \filter\}$ is a filter on $Y$. We call it the \highlight{pushforward filter} of $\filter$ under $f$.\\
		Moreover, if $\filter$ is an ultrafilter, so is $f_*(\filter)$.
	\end{lemma}
	\begin{proof}		
		Observe that the definition above is well-posed, since $f_*(\filter)$ is indeed a filter. Non-emptyness is obvious, since for example $Y\in f_*(\filter)$. Moreover, for $U,V \in f_*(\filter)$, we have $f^{-1}(U\cap V) = f^{-1}(U)\cap f^{-1}(V) \in \filter$, so $U\cap V \in f_*(\filter)$. Finally, if $U \in f_*(\filter)$ and $U\subseteq V$, then $f^{-1}(V) \supseteq f^{-1}(U) \in\filter$, so also $f^{-1}(V)\in\filter$, since $\filter$ is a filter.\\
		Moreover the construction of the pushforward filter is functorial. For sets $X,Y,Z$ and $\filter$ a filter on $X$ we have
		\begin{align*}
			&W\in g_*(f_*(\filter)) \Leftrightarrow g^{-1}(W) \in f_*(\filter) \Leftrightarrow f^{-1}(g^{-1}(W)) \in \filter\\
			&\Leftrightarrow f^{-1}\circ g^{-1}(W) \in \filter \Leftrightarrow (g\circ f)^{-1}(W) \in \filter \Leftrightarrow W \in (g\circ f)_*(\filter).
		\end{align*}
		
		Finally assume that $\filter$ is an ultrafilter and $V\not\in f_*(\filter)$. Then $f^{-1}(V) \not\in \filter\Leftrightarrow X\setminus f^{-1}(V) \in \filter$. However, we also have that $f^{-1}(Y\setminus V) = X\setminus f^{-1}(V)$  and hence $Y\setminus V \in f_*(\filter)$. Using \Cref{lem:characterization-of-ultrafilter} we get that $f_*(\filter)$ is an ultrafilter.
	\end{proof}

	\begin{definition}[limit along a filter and ultralimit]
		Let $X$ be a set and $Y$ be a topological space. Moreover let $\filter$ be a filter on $X$ and $f:X\to Y$ a function. We say that $f$ \highlight{converges to $y_0\in Y$ along $\filter$}, denoted by $f\fconv y_0$, if the pushforward filter $f_*(\ultrafilter)$ converges to $y_0$.\\
		If $y_0$ is unique we call it the \highlight{limit of $f$ along $\filter$} and write $y_0 = \flim f(x)$.\\
		Moreover, if $\filter$ is an ultrafilter and the limit $y_0 = \flim f(x)$ exists, we call it the \highlight{ultralimit of $f$} (along $\filter$).
	\end{definition}
	
	\begin{remark}
		Note that this definition generalizes the concept of converging sequences (or nets) as the following example shows.\\
		However, in general the existence of a limit of a function, as well as the value of the limit, do not only depend on the function, but also on the filter itself, as the following example shows.
	\end{remark}

	\begin{example}
		Consider the Frech\'et Filter $\filter$ on $\NN$ as well as a sequence $(a_n)_{n\in\NN}$ (respectively a function $a:\NN\to Y$) in a topological space $Y$ (e.g. $\RR$). Then $(a_n)_{n\in\NN}$ converges to $y\in Y$ iff for all neighborhoods $V$ of $y$ only finitely many elements of the sequence do not lie in $V$, equivalently the preimage $a^{-1}(V)$ is a cofinite set. However, this corresponds exactly to the convergence of the pushforward filter $a_*(\NN)$ to $y$.\\
		If we pick the sequence $a_n = (-1)^n$, this will clearly not converge along $\filter$. However, if we choose any ultrafilter $\ultrafilter$ containing $\filter$, then $(a_n)_{n\in\NN}$ will converge to $1$, or $-1$, depending on whether the set of all even natural numbers is in $\ultrafilter$, or its complement, the set of all odd natural numbers is in $\ultrafilter$.
	\end{example}
	
	\begin{example}\label{exp:ultralimit-of-constant-function}
		Of course, constant functions should intuitively converge along any given filter and the limit should not depend on the filter and be equal to the value of the function. This is indeed the case.\\
		Let $f:X\to Y$ be a constant function with value $y\in Y$ and $\filter$ be any filter on $X$. Then $f_*(\filter) = \ultrafilter(\{y\})$ is the trivial filter induced by $\{y\}$, hence $f$ converges along $\filter$ to $y$.
	\end{example}

	We give two more special examples of limits along filter.\\
	
	\begin{example}
		Observe that any function $f:X\to Y$ between topological spaces $X,Y$ preserves trivial ultrafilters. We have $f_*(\ultrafilter(\{x\}) = \ultrafilter(\{f(x)\})$. Hence the limit of any function $f$ along a trivial ultrafilter induced by $\{x\}$ is $f(x)$.\\
	\end{example}

	This stands in contrast to the situation when we only look at the neighborhood filters of a point:
	
	\begin{example}\label{exp:nbhd_filter-under-continuous-function}
		Neighborhood filters are not preserved in general, i.e. we need not have that $\filter_{f(x)} \subseteq f_*(\filter_x)$. More precisely, preserving the neighborhood filter of a point $x$ is equivalent to being continuous (in the topological sense) at $x$.\\
		In particular, the limit of a continuous function along the neighborhood filter $\filter_x$ is $f(x)$.\\
		Indeed, first assume that $f:X\to Y$ is continuous at $x$. That is for every neighborhood $V$ of $f(x)$ there is a neighborhood $U$ of $x$ such that $f(U) \subseteq V$.\\
		Let $V\in\filter_{f(x)}$, then, as just described, $f^{-1}(V)$ contains a neighborhood $U$ of $x$. Since $U\in\filter_x$ we also get that $f^{-1}(V) \in \filter_{x}$ because it is a superset of $U$. However, that means exactly that $V\in f_*(\filter_{x})$. Hence $\filter_{f(x)}\subseteq f_*(\filter_{x})$, i.e. $f$ preserves the neighborhood filter $\filter_{f(x)}$\\
		Conversely, assume that $\filter_{f(x)}\subseteq f_*(\filter_x)$. That means that if $V$ is a neighborhood of $f(x)$, then $f^{-1}(V)\in\filter_x$. However, then $f^{-1}(V)$ is a neighborhood of $x$ such that $f(f^{-1}(V))\subseteq V$, so $f$ is continuous at $x$.
	\end{example}
	
	\begin{remark}
		Note that even if $f$ is continuous we do not have $\filter_{f(x)} = f_*(\filter_{x})$ in general. So $f_*(\filter_{x})$ can well be larger than the neighborhood filter of $f(x)$. It can even be an ultrafilter, which $\filter_{f(x)} $ definitely is not. This happens for example if $f$ is a constant function (see \Cref{exp:ultralimit-of-constant-function}).\\
		Nevertheless there are also examples where $f_*(\filter_{x})$ is not an ultrafilter. Consider $f:\RR\to\RR$, $f(x) = 0$ for $x\leq 0$ and $f(x) = x$ if $x>0$. $V = [0,\infty)$ is certainly not a neighborhood of $0 = f(0)$, but $f^{-1}(V) = \RR$ is inside every filter, so also inside $f_*(\filter_{0})$. Moreover $f_*(\filter_{0})$ is not an ultrafilter, since $f^{-1}(\{0\}) = (-\infty,0] \not \in \filter_{0}$, so $\{0\}\not\in f_*(\filter_{0})$, but also $f^{-1}(\RR\setminus\{0\}) = (0,\infty)\not\in \filter_{0}$, so $\RR\setminus\{0\}\not\in f_*(\filter_{0})$.
	\end{remark}

	By \Cref{lem:uniqueness_of_limit_in_Hausdorff_spaces} we know that in a topological Hausdorff space, the limit of a function along a filter is unique. Moreover, for a compact image space we even have the existence of a limit along an ultrafilter for any function. One particular case, which will be useful later on, is considered below.
	
	\begin{corollary}\label{cor:bounded-maps-attain-ultralimits}
		Let $X$ be a set and $f:X\to\RR$ a bounded function. Then for every ultrafilter $\ultrafilter$ an ultralimit $\ulim(f)$ along the ultrafilter exists.
	\end{corollary}
	\begin{proof}
		This follow directly from the fact that the pushforward ultrafilter $f_*(\ultrafilter)$ on $\RR$ induces (by restriction) an ultrafilter $\ultrafilter|_{[-\norm{f},\norm{f}]}$ on the compact set $[-\norm{f},\norm{f}]$ and that ultrafilters on compact set converge.\\
		Here we define $\ultrafilter|_{[-\norm{f},\norm{f}]} := \{ V\cap [-\norm{f},\norm{f}]~|~V\in \ultrafilter \}$, which is indeed an ultrafilter.
	\end{proof}
	
	\begin{lemma}\label{lem:limit-product}
		Let $X$ be a set, $\filter$ a filter on $X$ and $Y,Z$ be topological spaces. If $f:X\to Y$ and $g:X\to Z$ converge along $\filter$ then so does the function $H:X\to Y\times Z$, $x\mapsto (f(x),g(x))$ where $Y\times Z$ carries the product topology.\\
		More precisely, if $f\fconv y_0$ and $g\fconv z_0$ then $H \fconv (y_0,z_0)$.
	\end{lemma}
	\begin{proof}
		Assume that $f\fconv y_0$ and $g\fconv z_0$. Let $W\subseteq Y\times Z$ be any neighborhood of $(y_0,z_0)$. By definition of the product topology, there exist neighborhoods $U\subseteq Y$, $V\subseteq Z$ of $y_0,z_0$ respectively, such that $U\times V\subseteq W$. Now $H^{-1}(W)\supseteq H^{-1}(U\times V) = f^{-1}(U)\cap g^{-1}(V) \in \filter$. Therefore $\filter_{(y_0,z_0)} \subseteq H_*(\filter)$, so $H \fconv (y_0,z_0)$. 
	\end{proof}
	
	\begin{lemma}\label{lem:limit-composition}
		Let $X$ be a set, $\filter$ a filter on $X$ and $Y,Z$ be topological spaces. Assume $f:X\to Y$ converges along $\filter$ to $y_0$ and $g:Y\to Z$ is continuous at $y_0$ then $(g\circ f):X\to Z$ converges along $\filter$ to $z_0 = g(y_0)$.
	\end{lemma}
	\begin{proof}
		This follows from the facts that the construction of the pushforward filter is functorial, preserves inclusion and that the pushforward filter of the neighborhood filter of $y_0$ under $g$ contains the neighborhood filter of $z_0 = g(y_0)$. (see \Cref{exp:nbhd_filter-under-continuous-function})\\
		Indeed, by assumption we have $\filter_{y_0}\subseteq f_*(\filter)$ and therefore $\filter_{z_0} \subseteq g_*(\filter_{y_0})\subseteq g_*(f_*(\filter)) = (g\circ f)_*(\filter_{(y_0,z_0)})$. This means precisely that $(g\circ f)\fconv z_0$
	\end{proof}
	
	\begin{definition}[mean]\label{APX_def:mean}
		Let $X$ be a set. A linear function $m:\linf{X}\to\RR$ is a mean if:
		\begin{enumerate}[label=(\roman*), itemsep=0pt, nosep]
			\item $\mean(\mathbbm{1}) = 1$. \label{mean-itm:1}
			\item if $f\geq 0$ then $\mean(f) \geq 0$ for all $f\in\linf{X}$. \label{mean-itm:2}
			\item $\norm[op]{\mean} = 1$. \label{mean-itm:3}
		\end{enumerate}
	\end{definition}

	\begin{remark}
		Actually any two of the properties above imply the third one.\\
		$\ref{mean-itm:1} + \ref{mean-itm:2} \Rightarrow \ref{mean-itm:3}$: By \ref{mean-itm:1} it is clear that $\norm[op]{\mean} \geq 1$. By \ref{mean-itm:2} we get for all $f\in\linf{X}$ that $\mean(\norm{f} - f) \geq 0 \Leftrightarrow \norm{f}\mean(\mathbbm{1}) \geq \mean(f) \Leftrightarrow \norm{f} \geq \mean(f)$, where we used \ref{mean-itm:1} in the last step. Hence $\norm[op]{\mean} \leq 1$, which implies \ref{mean-itm:3}.\\
		$\ref{mean-itm:2} + \ref{mean-itm:3} \Rightarrow \ref{mean-itm:1}$: It follows directly from \ref{mean-itm:3} that $\mean(\mathbbm{1}) \leq 1$. Moreover, by \ref{mean-itm:2} it again holds that $\mean(\norm{f} - f) \geq 0 \Leftrightarrow \norm{f}\mean(\mathbbm{1}) \geq \mean(f)$ for all $f\in\linf{X}$. Taking the supremum on all $f\in\linf{X}$ with $\norm{f} = 1$, we get $\mean(\mathbbm{1}) \geq \sup_{\norm{f} = 1}\mean(f) = \norm[op]{\mean} = 1$. Hence $\mean(\mathbbm{1}) = 1$.\\
		$\ref{mean-itm:3} + \ref{mean-itm:1} \Rightarrow \ref{mean-itm:2}$: It is certainly enough to show that $\inf_{x\in X} f(x) \leq \mean(f) \leq \sup_{x\in X} f(x)$ for all $f\in\linf{X}$.\\
		Linearity of $\mean$ gives that $\mean(f) = \norm{f}\mean(f/\norm{f}) \leq \norm{f}$.\\
		Moreover, by \ref{mean-itm:1} and the preceding statement, we have $\mean(f) + \norm{f} = \mean(f + \norm{f}\mathbbm{1}) \leq \norm{f + \norm{f}\mathbbm{1}} = \sup_{x\in X}f(x) + \norm{f}$. Where the last equality holds because $f + \norm{f}\mathbbm{1} \geq 0$. Therefore $\mean(f) \leq \sup_{x\in X}$.\\
		Similarly we get $\norm{f} - \mean(f) = \mean(\norm{f}\mathbbm{1} - f) \leq \norm{\norm{f}\mathbbm{1}- f} = \norm{f} - \inf_{x\in X}f(x)$. Where the last equality holds because $\norm{f}\mathbbm{1} - f \geq 0$. Therefore $\inf_{x\in X}f(x) \leq \mean(f)$. This concludes the proof.
	\end{remark}

	\begin{proposition}\label{lem:properties-of-ultralimits}
		Let $X$ be a set and $\ultrafilter$ be an ultrafilter on $X$. Then the map $m_{\ultrafilter}:\linf{X}\to\RR$, $m_{\ultrafilter}(f) := \ulim f(x)$ is a mean.\\
		In particular, $m_{\ultrafilter}$ is a linear map.
	\end{proposition}
	\begin{proof}
		We already saw that $m_{\ultrafilter}$ is well-defined, since for any $f\in\linf{X}$, the pushforward ultrafilter $f_*(\ultrafilter)$ induces an ultrafilter on the compact set $[-\norm{f},\norm{f}]$, which converges to a unique point, since $[-\norm{f},\norm{f}]$ is Hausdorff (\Cref{lem:uniqueness_of_limit_in_Hausdorff_spaces}).\\
		Moreover we also know that the ultralimit of constant functions is the value of that function, so in particular $m_{\ultrafilter}(\mathbbm{1}) = 1$ (\Cref{exp:ultralimit-of-constant-function}).
		Concerning linearity, note that (for a fixed $c\in\RR$) the functions $G:\RR\to\RR$, $G(x) = cx$ and $H:\RR\times\RR\to \RR$, $H(x,y) = x+y$ are continuous. Hence by \Cref{lem:limit-product} and \Cref{lem:limit-composition} we have
		\[m_{\ultrafilter}(cf) = m_{\ultrafilter}\left (G\circ f\right ) =  \ulim (G\circ f)(x) = G(\ulim f(x)) = c\cdot\ulim f(x) = c\cdot m_{\ultrafilter}(f) \]
		as well as
		\begin{align*}
			m_{\ultrafilter}(f+g) &= m_{\ultrafilter}(H\circ (f,g))) =  \ulim (H\circ (f,g))(x) = H(\ulim (f,g)(x)))\\
			&= H(\ulim f(x), \ulim g(x)) = \ulim f(x) + \ulim g(x) = m_{\ultrafilter}(f) + m_{\ultrafilter}(g)
		\end{align*}
		for all $f,g\in\linf{X}$.\\
		What is left to show is that $m_{\ultrafilter}(f) \geq 0$ for all $f\in\linf{X}$ with $f\geq 0$. Let $f\geq 0$ and assume that $y = m_{\ultrafilter}(f) < 0$. Then by definition of the ultralimit, we have $\filter_y \subseteq f_*(\ultrafilter)$. Choose $\epsilon >0$ small enough such that $y+\epsilon < 0$. Then also $y\in(y-\epsilon,y+\epsilon) \in f_*(\ultrafilter)$. However, $f^{-1}(y-\epsilon,y+\epsilon) = \emptyset$, since $f\geq 0$. This is a contradiction.\\
		This concludes the proof that $m_{\ultrafilter}$ is a mean.
	\end{proof}

	Now we are ready to prove that generic relations give rise to \bAc semi-simplicial sets. For the convenience of the reader we recall the statement once more.

	\begin{proposition}\Cref{prop:generic-rel->LES}\label{APX_prop:generic-rel->LES}
		Let $\generic$ be a generic relation on a set $X$. Denote by $X^{\generic}_\bullet$ the associated semi-simplicial set. Then $X^{\generic}_\bullet$ is \bAc, i.e. the cohomology of the corresponding cochain complex $\ChainComplex[0]{\linf{X^{\generic}_1}, \linf{X^{\generic}_2}, \linf{X^{\generic}_3}}[][\delta^]$ vanishes in all degrees except $0$.
	\end{proposition}

	 Recall that here the differential maps are defined as follows:\\
	 \begin{align*}
 		\codiff[n-1]:\linf{X^{\generic}_{n-1}} &\to \linf{X^{\generic}_n}\\ \codiff[n-1](f)(x_0,\dots,x_n) &= \sum_{i=0}^{n}(-1)^{i+1} f(x_0,\dots,\widehat{x}_i,\dots,x_n).
	 \end{align*}
	 
	\begin{proof}[Proof of \Cref{prop:generic-rel->LES}]
		Recall that we want to prove that the cochain complex
		\[ \LES[0]{\linf{X^{\generic}_1}, \linf{X^{\generic}_2}, \linf{X^{\generic}_3}} \]
		is \bAc. To do this, we construct a rather explicit contracting homotopy $\Theta_{\bullet}$ between the identity and the $0$ map. We do this using ultrafilters and ultralimits.\\
		First of all we define a filter $\filter$ on $X$ as follows:\\
		For every finite subset $F\subseteq X$ define $B_F := \{x\in X~|~\forall y \in F:~x~\generic~y\}$. Now note that $B_F \cap B_{F'} = B_{F\cup F'}$. Hence, by \Cref{lem:family-of-subsets-give-rise-to-filter} there is a filter $\filter$ containing $B_F$ for all finite subsets $F\subseteq X$. Choose an ultrafilter $\ultrafilter$ containing $\filter$.\\
		We can now define the contracting homotopy $\Theta_{\bullet}$ degreewise via
		\begin{align*}
 			\Theta_n:\linf{X^{\generic}_n} &\longrightarrow \linf{X^{\generic}_{n-1}}\\
 			\Theta_n(f)(x_1,\dots,x_n) &:= \ulim f(x,x_1,\dots,x_n)
		\end{align*}
		
		Note that this ultralimit is well-defined, since for fixed points $x_1,\dots,x_n$ the map $x\mapsto f(x,x_1,\dots,x_n)$ is bounded, so we can apply \Cref{cor:bounded-maps-attain-ultralimits}. Moreover note, that $\Theta_n$ is not only well-defined but also bounded (we have $\norm{\Theta_n} \leq 1$) and linear, since passing to ultralimits satisfy these properties (\Cref{lem:properties-of-ultralimits}). What is left to show is that this gives indeed a contracting homotopy, i.e. $\codiff[n-1]\Theta_{n} + \Theta_{n+1}\codiff[n] = id_{n}$ for all $n\in\NN$. In the following we write $_0x_n$ for $x_0,\dots,x_n$ and $^{}_0x_n^i$ for $x_0,\dots,\widehat{x}_i,\dots,x_n$. 
		\begin{align*}
			(\codiff[n-1]\Theta_{n} + \Theta_{n+1}\codiff[n](f)(_0x_n) = \sum_{i=0}^{n}(-1)^i\Theta_{n}(f)(^{}_0x_n^i) + \ulim\codiff[n](f)(x,\-_0x_n) &=\\
			=\sum_{i=0}^{n}(-1)^i\Theta_{n}(f)(^{}_0x_n^i) + \ulim \left(f(_0x_n) + \sum_{i=1}^{n+1}(-1)^if(x,\-_0x_n^{i-1})\right)&=\\
			=\sum_{i=0}^{n}(-1)^i\ulim f(x,\-^{}_0x_n^i) + f(_0x_n) + \sum_{i=0}^{n}(-1)^{i+1}\ulim f(x,\-^{}_0x_n^i) &= id_{n}(f)(_0x_n)\\
		\end{align*}
		Here we used the fact that $\Theta_n$ is linear and that the ultralimit of a constant function is equal to the value of the function.\\
		The equation above holds for every tuple $(x_0,\dots, x_n)$, hence we get $(\codiff[n-1]\Theta_n + \Theta_{n+1}\codiff[n])(f) = f$ so $\Theta_{\bullet}$ is indeed a bounded contracting homotopy.
	\end{proof}

	\newpage
	
	\addcontentsline{toc}{section}{References}
	\printbibliography[title=References]
\end{document}